\documentclass[reqno]{amsart}
\usepackage{amsfonts}

\usepackage{graphicx}
\usepackage{amsfonts,amsmath, amssymb}
\usepackage{hyperref}

\numberwithin{equation}{section}

\makeatletter
\def\specialsection{\@startsection{section}{1}%
  \z@{\linespacing\@plus\linespacing}{.5\linespacing}%
  {\Large\bfseries}}
\def\section{\@startsection{section}{1}%
  \z@{.7\linespacing\@plus\linespacing}{.5\linespacing}%
  {\Large\scshape\bfseries}}
\makeatother

\usepackage{marginnote}

\usepackage{color}

\newcommand{\R}{\mathbb{R}}

\newtheorem{theorem}{Theorem}[section]

\newtheorem{lemma}[theorem]{Lemma}
\newtheorem{proposition}[theorem]{Proposition}

\def\v{\varepsilon}

\def\t{\theta}
\def\k{\kappa}

\def\m{\mu}
\def\a{\alpha}
\def\b{\beta}
\def\g{\gamma}
\def\d{\delta}
\def\l{\lambda}

\def\r{\rho}
\def\s{\sigma}

\def\f{\frac}

\usepackage{tikz}

\newcommand{\dd}{{\rm d}}

\renewcommand{\S}{\mathbb{S}}

\newcommand{\Fi}{\mathbf{1}}

\newcommand{\CA}{\mathcal{A}}

\newcommand{\CE}{\mathcal{E}}

\newcommand{\CJ}{\mathcal{J}}

\newcommand{\na}{\nabla}

\newcommand{\al}{\alpha}
\newcommand{\be}{\beta}
\newcommand{\ga}{\gamma}
\newcommand{\om}{\omega}
\newcommand{\la}{\lambda}
\newcommand{\de}{\delta}
\newcommand{\si}{\sigma}
\newcommand{\pa}{\partial}
\newcommand{\ka}{\kappa}
\newcommand{\eps}{\epsilon}

\newcommand{\Ga}{\Gamma}

\begin{document}

\title[Boltzmann equation with large-amplitude data]{The Boltzmann equation with large-amplitude initial data in bounded domains}

\author[R.-J. Duan]{Renjun Duan}
\address[R.-J. Duan]{Department of Mathematics, The Chinese University of Hong Kong, Hong Kong}
\email{rjduan@math.cuhk.edu.hk}

\author[Y. Wang]{Yong Wang}
\address[Y. Wang]{Institute of Applied Mathematics, Academy of Mathematics and Systems Science, Chinese Academy of Sciences, Beijing 100190, China, and University of Chinese Academy of Sciences}
\email{yongwang@amss.ac.cn}

\begin{abstract} 
The paper is devoted to constructing the global solutions around global Maxwellians to the initial-boundary value problem on the Boltzmann equation in general bounded domains with isothermal diffuse reflection boundaries. We allow a class of non-negative initial data which have arbitrary large amplitude and even contain vacuum. The result shows that the oscillation of solutions away from global Maxwellians becomes small after some positive time provided that they are initially close to each other in $L^2$.   This yields the disappearance of any initial vacuum  and the exponential convergence of large-amplitude solutions to equilibrium  in large time. The isothermal diffuse reflection boundary condition plays a vital role in the analysis.  The most key ingredients in our strategy of the proof include: (i) $L^2_{x,v}$--$L^\infty_xL^1_v$--$L^\infty_{x,v}$ estimates along a bootstrap argument; (ii) Pointwise estimates on the upper bound of the gain term by the product of $L^\infty$ norm and $L^2$ norm; (iii) An iterative procedure on the nonlinear term. 
\end{abstract}

\subjclass[2000]{Primary 35Q20, 76P05; Secondary 35A01, 35B45}


\keywords{Boltzmann equation; global well-posedness; bounded domain; large-amplitude initial data; diffuse reflection boundary; {\it a priori} estimate}
\date{\today}
\maketitle

\setcounter{tocdepth}{1}
\tableofcontents

\thispagestyle{empty}


\section{Introduction}

The paper is concerned with the initial-boundary value problem on the Boltzmann equation in bounded domains. The hard potentials with angular cutoff are considered. For the diffuse reflection boundary condition, we shall  establish the global-in-time existence and exponential time-decay of solutions in the velocity-weighted $L^\infty$ space for initial data whose amplitude around global Maxwellians have arbitrary large oscillations in the same weighted $L^\infty$ space but are small in $L^2$ space. Particularly, initial data are allowed to have large amplitude and contain vacuum. The weight function can be either purely polynomial or the product of polynomial and exponential functions. 

\subsection{Boltzmann equation}

We consider  a rarefied gas in a vessel $\Omega$ whose boundaries are kept at constant temperature. The domain $\Omega$ is assumed to be  a bounded open set of $\R^3$ with a  smooth boundary $\pa \Omega$. The time evolution of the gas is governed by the initial-boundary value problem on the Boltzmann equation
\begin{equation}\label{1.1}
\partial_tF+v\cdot\nabla_x F=Q(F,F),
\end{equation}
supplemented with initial data
\begin{equation}\label{1.5-1}
F(0,x,v)=F_0(x,v),
\end{equation}
and with the diffuse reflection boundary condition which will be clarified later on. Here, the unknown $F=F(t,x,v)\geq 0$ stands for the density distribution function of gas particles with position $x\in\Omega$ and  velocity $v\in\mathbb{R}^3$  at  time $t>0$. 
The collision term $Q(F,F)$ is an integral with respect to velocity variable only, and it takes the nonsymmetric bilinear form 
\begin{align}
Q(F_1,F_2)&=\int_{\mathbb{R}^3}\int_{\mathbb{S}^2} B(v-u,\om)\left[F_1(u')F_2(v')
-F_1(u)F_2(v)\right]\,\dd\omega\dd u.\nonumber
\end{align}
Here the post-collison velocity pair $(v',u')$ and the pre-collision velocity pair $(v,u)$  satisfy the relation
\begin{equation}\label{def.popr}
v'=v-[(v-u)\cdot\omega]\,\omega,\quad u'=u+[(v-u)\cdot\omega]\,\omega,\nonumber
\end{equation}
with $\om\in \S^2$,  according to conservations of momentum and energy of two particles before and after the collision
\begin{equation}\label{1.3-1}
v+u=v'+u',\quad 
|v|^2+|u|^2=|v'|^2+|u'|^2.
\end{equation}
The collision kernel $B(v-u,\om)$ depends only on the relative velocity $|v-u|$ and $\cos\theta:=(v-u)\cdot \om/|v-u|$, and it is assumed throughout the paper to take the form
\begin{equation}\label{1.4}
B(v-u,\om)=|v-u|^{\ka}b(\cos\theta),\quad 
0\leq\ka\leq1,\quad 0\leq b(\t)\leq C|\cos\t |,
\end{equation}
corresponding to the case of hard potentials with angular cutoff. Under the cutoff assumption, it is also convenient to write
\begin{equation}\nonumber
Q(F_1,F_2)=Q_+(F_1,F_2)-Q_-(F_1,F_2),
\end{equation}
with $Q_+(F_1,F_2)$, $Q_-(F_1,F_2)$ meaning the gain term and the loss term, respectively. 
 

\subsection{Boundary condition}
Throughout the paper, we assume  that $\Omega:=\{x\in \R^3:\xi (x)<0\}$ for a smooth function $\xi(x)$ is connected and bounded. 
Suppose that  
$\nabla \xi
(x)\neq 0$ at the boundary $\pa\Omega=\{x\in \R^3: \xi (x)=0\}$, and the outward normal vector at $
\partial \Omega $, given by 
$n(x)=\nabla \xi (x)/|\nabla \xi (x)|$, 
can be extended smoothly near $\partial \Omega$. 
Moreover,
$\Omega $ is said to be strictly convex 
if 
there exists $c_{\xi }>0$ such that 
\begin{equation} 
\sum_{ij=1}^3\frac{\pa^2 \xi}{\pa x_i \pa x_j}(x)\zeta_{i}\zeta_{j}\geq c_{\xi }|\zeta |^{2},\label{convexity}
\end{equation}
for all $x\in \overline{\Omega}$ 
and for all $\zeta \in \mathbb{R}^{3}$.
We denote the phase boundary of the phase space $\Omega \times \mathbb{R}^{3}_v$ by $
\gamma =
\partial \Omega \times \mathbb{R}^{3}_v,$ and split it into the outgoing set
$\gamma _{+},$ the incoming set 
$\gamma _{-},$ and the
singular set 
$\gamma _{0}$ for grazing velocities: 
\begin{align}
\gamma _{+} &=\{(x,v)\in \partial \Omega \times \mathbb{R}^{3}_v:~n(x)\cdot v>0\}, \nonumber\\
\gamma _{-} &=\{(x,v)\in \partial \Omega \times \mathbb{R}^{3}_v:~n(x)\cdot v<0\}, \nonumber\\
\gamma _{0} &=\{(x,v)\in \partial \Omega \times \mathbb{R}^{3}_v:~n(x)\cdot v=0\}.\nonumber
\end{align}
In this paper, we only consider the diffuse reflection boundary condition, namely, for $(x,v)\in \ga_-$,
\begin{equation}\label{diffuse1}
F(t,x,v)|_{\gamma _{-}}=c_{\mu }\mu (v)\int_{v^{\prime }\cdot
	n(x)>0}F(t,x,v^{\prime })\{v'\cdot n(x)
	\}\,\dd v^{\prime },
\end{equation}
where 
$\mu(v)=(2\pi)^{-3/2}\exp(-|v|^2/2)$ 
is the normalized global Maxwellian with zero bulk velocity, and the constant $c_{\mu}=\sqrt{2\pi}>0$ is chosen such that 
\begin{equation}\label{diffusenormal}
c_{\mu }\int_{v'\cdot n(x)>0}\mu (v)\{v'\cdot n(x)\}\,\dd v=1.
\end{equation}
Note that $\mu(v)$ satisfies \eqref{1.1} and \eqref{diffuse1}, and under the above diffuse reflection boundary condition, it holds that  
\begin{equation*}
\int_{\Omega \times \mathbb{R}^{3}}F(t,x,v)\,\dd x\dd v=\int_{\Omega \times \mathbb{R}^{3}}F_0(x,v)\,\dd x\dd v,\quad t>0,
\end{equation*}
for any solution to the initial-boundary value problem \eqref{1.1}, \eqref{1.5-1}, \eqref{diffuse1}.

\subsection{$L^\infty$ solution of small amplitude}


It was proved in \cite{Guo2} that

\begin{proposition}\label{prop1.1}
Let $w(v)=(1+\r^2|v|^2)^{\beta}e^{\varpi|v|^2}$ with $\r>1$ fixed and suitably large, $\beta>3/2$, and  $0\leq \varpi< 1/4$.  
Assume that  $F_0(x,v)=\mu+\sqrt{\mu}f_0(x,v)\geq0$, and the mass conservation holds:
\begin{equation}
\label{mc.i}
\int_{\Omega\times \R^3} \sqrt{\mu}f_0(x,v) \,\dd x \dd v=0.
\end{equation}
Then there are constants $\d>0$  and $C_0\geq1$ 
such that if 
\begin{align}\label{def.delta}
\|wf_0\|_{L^\infty}\leq \d, 
\end{align}
then the initial-boundary value problem \eqref{1.1}, \eqref{1.5-1}, \eqref{diffuse1} on the Boltzmann equation  
admits a unique solution $F(t,x,v)=\mu+\sqrt{\mu}f(t,x,v)\geq0$ 
satisfying
\begin{align}\label{g.rate}
\|wf(t)\|_{L^\infty}\leq C_0 \|wf_0\|_{L^\infty} e^{-\vartheta t},
\end{align}
for all $t\geq 0$, where $\vartheta>0$ is a constant. Moreover, if $\Omega$ is strictly convex, and 
$F_0(x,v)$ is continuous except on $\gamma_0$ with 
\begin{align}\label{bc.i}
F_0(x,v)|_{\gamma_-}=c_{\mu }\mu (v)\int_{v^{\prime }\cdot
	n(x)>0}F_0(x,v^{\prime })\{v'\cdot n(x)\}\,\dd v^{\prime },
\end{align}
then $F(t,x,v)$ is continous in $[0,\infty)\times\{\overline{\Omega}\times\mathbb{R}^3_v\backslash \gamma_0\}$.
\end{proposition}

Note that the original result in \cite{Guo2} considers only the purely polynomial velocity weight corresponding to $\varpi=0$ in the above Proposition \ref{prop1.1}. It is straightforward to extend it to include the exponential weight in the case $0<\varpi<1/4$, as also pointed out and used in \cite[see the paragraph below Theorem 1]{GKTT-IM}. Indeed, the exponential velocity weight function $\exp(\varpi|v|^2)$ with $\varpi>0$ was needed in \cite{GKTT-IM} mainly for the investigation of regularity of solutions to  the  Boltzmann equation in convex domains. For the same purpose, later on we shall also make some comments on regularity of weighted $L^\infty$ solutions obtained in the setting of the current work. Moreover, for the boundary condition, we focus on only the diffuse reflection boundary in this work. In fact, regarding the initial data with small oscillations around global Maxwellians in $L^\infty$, the global existence of solutions under other types of boundary conditions, such as in-flow, reverse reflection and specular reflection, was also established in \cite{Guo2}. Thus, for convenience of readers we will make some discussions on the prospective extensions of the current work to other types of boundary conditions in Section \ref{sec.pro} later on.


\subsection{Main results}

The goal of the paper is to remove the smallness of $\|wf_0\|_{L^\infty}$ in the above proposition with the price that $\|f_0\|_{L^2}$ is suitably small. Thus, initial data $F_0(x,v)$ are allowed to have arbitrary large oscillations near global Maxwellians, and hence $F_0(x,v)$ can have large amplitude  and even contain the {\it vacuum}.  

\begin{theorem}\label{thm1.2}
Let $w(v)=(1+\r^2|v|^2)^{\beta}e^{\varpi|v|^2}$ with $\r>1$ fixed and suitably large, $\beta\geq 5/2$, and $0\leq \varpi\leq 1/64$. 
Assume that  $F_0(x,v)=\mu+\sqrt{\mu}f_0(x,v)\geq0$ satisfying the mass conservation \eqref{mc.i}.
For any $M_0\geq 1$, there exists a  constant $\epsilon_0>0$, depending only on 
$\d$ and $M_0$, such that if 
\begin{equation*}
\|wf_0\|_{L^\infty}\leq M_0,\quad \|f_0\|_{L^2}\leq \epsilon_0,
\end{equation*}
then the initial-boundary value problem \eqref{1.1}, \eqref{1.5-1}, \eqref{diffuse1} on the Boltzmann equation  
admits a unique solution $F(t,x,v)=\mu+\sqrt{\mu}f(t,x,v)\geq0$ 
satisfying
\begin{align}\label{thm.ltb}
\|wf(t)\|_{L^\infty}\leq \widetilde{C}_0M_0^5\exp\left\{ \f2{\nu_0}\widetilde{C}_0M_0^5\right\} e^{-\vartheta_1 t},
\end{align}
for all $t\geq 0$, 
where $\widetilde{C}_0\geq1$ is a generic constant, 
$\vartheta_1:=\min\left\{\vartheta, \f{\nu_0}{16}\right\}>0$ where $\vartheta>0$ is the same constant as in \eqref{g.rate}, and $\nu_0:=\inf_{v\in \R^3} \int_{\R^3}\int_{\S^2}B(v-u,\omega) \mu(u)\, \dd \omega \dd u>0$. Moreover, if $\Omega$ is strictly convex, and 
$F_0(x,v)$ is continuous except on $\gamma_0$ 
and satisfies \eqref{bc.i}, 
then $F(t,x,v)$ is continous in $[0,\infty)\times\{\overline{\Omega}\times\mathbb{R}^3\backslash \gamma_0\}$.
\end{theorem}

The above theorem provides the first result on both the global-in-time existence and the  large-time behavior for solutions $F(t,x,v)$ to the Boltzmann equation \eqref{1.1} in bounded domains with initial data $F_0(x,v)$ which can be large in $L^\infty$ and even contains {\it vacuum}. Recently, important progress has been made in \cite{DV} to establish  the $t^{-\infty}$ rate of convergence to equilibrium for large-amplitude solutions to the  Boltzmann equation for general collision kernels and general boundary conditions, under some {\it a priori} high-order Sobolev bounds on solutions. As pointed out in \cite[Introduction, page 714]{Guo2}, even though those bounds can be verified in perturbation framework for spatially periodic domains (cf.~\cite{Guo-03}, for instance), their validity is unclear for the Boltzmann solutions in general bounded domains. Indeed, for a general non-convex domain, high-order Sobolev bounds may not be expected, and discontinuity of solutions near global Maxwellians can form and propagate along characteristics \cite{Kim-1}.   More surprisingly, even in a strictly convex domain with certain typical boundary conditions,  it has been shown in \cite[Appendix]{GKTT-IM} that some second order spatial derivative of close-to-Maxwellians solutions does not exist up to the boundary in general. Therefore, it is a challenging topic on studying the global existence and large-time behavior of solutions for large-amplitude initial data even with certain extra smallness in integrability.

Another important issue we concern in Theorem \ref{thm1.2} is that vacuum can be allowed for initial data, and even there could be no particles over a suitably small subdomain at initial time. This is different from \cite{DV} and \cite{GMM}, where either an explicit exponential lower bound on solutions for all time is needed, or the macroscopic density is  initially away from vacuum.  
On the other hand, in fact \cite{Br1} proved the immediate appearance of an exponential lower bound, uniform in time and space, for continuous mild solutions to the full Boltzmann equation in a $C^2$ convex bounded domain with the physical  diffusion boundary conditions, under the sole assumption of regularity of the solution; see also \cite{Br2} for specular boundary conditions. However, it is still a problem to justify the instantaneous disappearance of initial vacuum for general non-convex domains, and it is also unclear how to apply such result to study the global existence of large-amplitude Boltzmann solutions. 
We would point out that in the appearance of vacuum for initial data, it is hard to justify the long-term dynamics of non-negative solutions to the nonlinear Botlzmann equation. Our idea as explained in detail later on is to make use of smallness of perturbative solutions in $L^2$ space to bound solutions in weighted $L^\infty_xL^1_v$ space over a finite time interval after some positive time depending only on weighted $L^\infty$ norm of initial data, so that initial vacuum can be shown to disappear after a finite time.

In what follows we would further give a few remarks on the main result Theorem \ref{thm1.2}. First of all, we remark that the same strategy of the proof of Theorem \ref{thm1.2} together with \cite{Guo2}  should be able to be carried out to treat other types of boundary conditions, namely, in-flow, reverse reflection, and specular reflection. 
Second, it seems impossible for us to use the method developed in \cite{DHWY} to consider the global large-amplitude solutions in  the case of the initial-boundary value problem due to the complexity of characteristic lines. Here we shall develop a new strategy different from the previous work \cite{DHWY} to treat the problem in general bounded domains. Particularly, an {\it iteration} procedure on the {\it nonlinear} term will be developed to control solutions in $L^\infty$ norm. 

%

The last remark is concerned with regularity of the obtained $L^\infty$ solutions as extensively studied in \cite{GKTT-IM}. 
Let $\Omega$ be strictly convex. According to \cite{GKTT-IM}, under the assumptions of Theorem \ref{thm1.2} corresponding to the mixed polynomial and exponential weight $w(v)=(1+\r^2|v|^2)^{\beta}\exp(\varpi|v|^2)$ with $0<\varpi\leq 1/64$, then the solutions obtained in Theorem \ref{thm1.2} above must belong to the space $W^{1,p}(\Omega\times\mathbb{R}^3_v)$ $(1<p<2)$ and the weighted space $W^{1,p}(\Omega\times\mathbb{R}^3_v)$ $(2\leq p\leq\infty)$ for any time $t\in[0,\infty)$, provided that the initial data are in the same regular spaces, respectively. Therefore, for the diffuse reflection boundary condition, we have constructed a class of global-in-time $W^{1,p}$ solutions to the Boltzmann equation with initial data of large amplitude in  $L^\infty$ space in strictly convex  domains.
Since our main concern in this work is the global existence and conituity of solutions in $L^\infty$ space, we do not present the explicit statements of regularity results. We refer interested readers to \cite{GKTT-IM} for more details.

For the proof of Theorem \ref{thm1.2}, we first construct the local-in-time non-negative solutions of large amplitude with the lifespan explicitly determined by the only weighted $L^\infty$ norm of initial data, then develop new {\it a priori} estimates to treat the large-amplitude solutions in spaces $L^2_{x,v}$, $L^\infty_xL^1_v$ and $L^\infty_{x,v}$ with weights by a bootstrap argument along a new strategy, and hence the global existence follows from the extension of local-in-time solutions under additional smallness of $L^2$ norm of initial perturbation. One key point of our new strategy for the proof is that given any constants $M_0$ and $\delta$ with $0<\delta<M_0$, one can figure out a positive time $T_0$, depending only on $M_0$ and $\delta$, such that  the amplitude of solutions in weighted $L^\infty$ norm, if they exist, can reduce from $M_0$ at initial time to $\delta$ at time $T_0$, provided that $L^2$ norm of initial perturbation is  suitably small, see Figure 1 in Section \ref{sec1.6},  and then solutions can be extended to all time on $[T_0,\infty)$ and thus on all non-negative time $[0,\infty)$.  As a byproduct, even though initial data can have large amplitude and contain vacuum, the exponential time-decay of solutions is immediately obtained as a consequence in case of small-amplitude global solutions over $[T_0,\infty)$.  More details of the proof of Theorem \ref{thm1.2} will be given in Section \ref{sec1.6} later on.

\subsection{Known results on global well-posedness}

The Boltzmann equation is a fundamental model in the collisional kinetic theory. There has been an enormous literature on the study of well-posedness, particularly the global existence theory, for the initial and/or boundary value problem on the Boltzmann equation; see \cite{CIP} and \cite{Vi} and the references therein. In what follows we only mention some known results in the spatially inhomogeneous framework which are related to our work. 

First of all, the global existence of renormalized solutions for general initial data in $L^1_{x,v}$ with finite energy and entropy was proved by DiPerna and Lions \cite{D-Lion} for the case of the whole space; uniqueness still remains open. Such global existence result for weak solutions was later extended by Hamdache \cite{Ha} and also Arkeryd and Maslova \cite{AM} to the case of bounded domains for different types of boundary conditions. Mischler \cite{Mi-VPB} introduced some trace theorem to refine  those results on the sequence stability and global existence for the Boltzmann equation, and also generalized them to the initial-boundary value problem for the Vlasov-Poisson-Boltzmann system; some new results concerning weak-weak convergence were further obtained by Mischler \cite{Mi}. For the large time behavior of solutions with general initial data, the weak convergence of solutions in $L^1$ as $t\to \infty$ for the Boltzmann equation on bounded domains was considered by Desvillettes \cite{De}, and the strong convergence in $L^1$ to the stationary solution was studied by Arkeryd and Nouri \cite{AN} for the diffuse reflection boundary condition with constant wall temperature. A breakthrough was achieved by Desvillettes and Villani \cite{DV}  to study 
the convergence of a class of  large amplitude solutions  toward the global Maxwellian with an explicit almost exponential rate in large time, conditionally under  some assumptions on smoothness and polynomial moment bounds of the solutions, as well as a positive lower bound, i.e. away from vacuum. The result has been recently improved by Gualdani, Mischler and Mouhot \cite{GMM} to give the first constructive proof of exponential decay with sharp rate towards the global Maxwellian for the hard-sphere Boltzmann equation still under some a priori bounds on solutions, as well as for initial data whose macroscopic density is away from vacuum. 

The global existence and large time behavior of solutions are well established in perturbation framework. For the Cauchy problem on the Boltzmann equation near vacuum, the global existence of mild solutions was first obtained by Illner and Shinbrot \cite{IS}   in terms of a monotonicity argument introduced by Kaniel and Shinbrot \cite{KS}. Later on, the result of \cite{IS} has been extended in many aspects; see the references of \cite[Chapter 2]{Glassey}. An interesting extension in this direction was recently made by Bardos, Gamba, Golse and Levermore \cite{BGGL} for the study of  stability of a class of traveling Maxwellians connected to vacuum at infinity.    

On the other hand, in the perturbation framework near global Maxwellians, due to the fundamental study of the linearized operator, for instance Grad \cite{Gra}, the global well-posedness theory of the Boltzmann equation is indeed better understood in different kinds of settings. Under the angular cutoff assumption, the pioneering work on global existence and time-decay of mild solutions was done by Ukai \cite{Uk} for the hard potentials in terms of the spectral analysis and the bootstrap argument; see also \cite{NI,Sh,UY} and references therein. The result in  \cite{Uk} was also extended by Ukai and Asano \cite{UA} to the soft case $-1<\kappa<0$. For the full range of soft potentials $-3<\kappa<0$, Guo \cite{Guo-03} first obtained the global existence of classical solutions in the periodic box by using the energy method, and then Strain and Guo \cite{SG} established the time-decay of solutions. Since then, there have appeared many extensive applications of the $L^2$ energy method to the mathematical study of the kinetic equations with Boltzmann collisions in plasma physics, cf.~\cite{Guo-IM}, for instance. In the mean time, Liu-Yang-Yu \cite{LYY} also developed their energy method based on the macro-micro decomposition around local Maxwellians which admit the same fluid quantities as the solution itself,  and the approach has extensive applications to the study of stability of nontrivial wave patterns in the context of hyperbolic conservation laws, cf.~\cite{Yu2}, for instance. 

Moreover, for the global well-posedness theory around global Maxwellians, we would mention \cite{AMUXY-s} and \cite{GS} for the angular noncutoff Boltzmann equation, and \cite{DLX,MSa,SS} for existence and large-time behavior of solutions in Besov spaces. We also point out that for the hard-sphere Boltzmann equation in the torus, a non-symmetric energy method was developed in \cite{GMM} to obtain the large-time asymptotic stability of solutions to global Maxwellians with a sharp exponential  rate for initial data $F_0(x,v)$ such that $F_0-\mu$ is small enough in $L^1_vL^\infty_x((1+|v|)^k)$ with  a constant $k>2$.

Those works for perturbation theory near global Maxwellians we mentioned above are concerned with the Boltzmann equation either in the whole space or in the periodic box. As pointed out in Grad \cite{Gra-Pr}, it is one of the basic problems to study the boundary effect on the global-in-time dynamics in the Boltzmann theory. Indeed, Guo \cite{Guo2} developed a new mathematical theory to study the time decay and continuity of Boltzmann solutions for angular cutoff hard potentials with four basic types of boundary conditions. His approach inspired by Vidav \cite{Vidav} is based on the $L^2$ decay theory and its interplay with delicate $L^\infty$ analysis for the linearized Boltzmann equation in the presence of repeated interactions with the boundary. Such $L^2\cap L^\infty$ approach has proven very useful to study the global existence and regularity for the time-dependent or stationary Boltzmann equation on bounded domains, cf.~\cite{BG,EGKM,EGKM-hy,GKTT-BV,GKTT-IM,Kim-1}. Particularly, a great progress has been recently made by Guo, Kim, Tonon, and Trescases \cite{GKTT-IM} to construct weighted $C^1$ solutions away from the grazing set for all boundary conditions by introducing a distance function toward the grazing set. In the mean time, for the diffuse reflection boundary condition, they constructed $W^{1,p}$ solutions for $1<p<2$ and weighted $W^{1,p}$ solutions for $2\leq p\leq \infty$ as well. 
It was also proved in \cite{GKTT-IM} that the local-in-time $W^{1,p}$ solution exists for the case of  large initial data, and some second order spatial derivative does not exist up to the boundary in general. 

For the mathematical study of the Boltzmann equation with boundaries, we would further mention many other contributions in different aspects: \cite{AC, AEMN-1,AEMN-2,Cer, CCLS,Gui, KLT, LYa, LYu, MS,Yu1}; See also the books \cite{Cer-B,CIP,Ma,Yo} and references therein.  Here the book \cite{Yo} by Sone gives the systematic investigations of the initial and/or boundary value problem on the Boltzmann equation from the numerical point of view. 
 
%

At the moment we remark that in those works in the perturbation framework mentioned above, initial data are required to have small oscillations in the $L^\infty$ norm  around the global Maxwellian. A natural problem is how to treat initial data of large amplitude, particularly allowing to contain vacuum. To remove smallness in $L^\infty$, one possible idea is to replace it by the smallness in the $L^p$ norm with $1\leq p<\infty$, so that solutions can be of large oscillation over small parts in phase space.     
Indeed, recently \cite{DHWY} has developed a $L^\infty_xL^1_v\cap L^\infty_{x,v}$ method for the global well-posedness theory of the Boltzmann equation when initial data are allowed to have large amplitude. Precisely, it was proved that there exists a unique solution globally in time for the Boltzmann equation in the whole space or torus when 
$$
F_0-\mu\in L^\infty_{x,v} ((1+|v|)^\beta \mu^{-1/2})
$$
with some $\beta>\max\{3,3+\kappa\}$ satisfying an additional smallness condition that 
$$
\mathcal{E}(F_0)+\|F_0-\mu\|_{L^1_xL^\infty_v(\mu^{-1/2})}
$$
is small enough, where $\mathcal{E}(F_0)$ is related the defect quantities of mass, energy and entropy, cf.~\cite{DHWY}. In particular, the amplitude of initial data can have large oscillations.  Note that the result is valid for the full range of both the soft and  hard potentials $-3<\kappa\leq1$, and the 
$C^1$ regularity of solutions was also obtained.  

As it was mentioned before that the boundary effect on the global well-posedness is fundamental and important, we expect to further develop in the paper the large-amplitude $L^\infty$ theory for the mathematical study of the Boltzmann equation with boundaries.

\subsection{Strategy of the proof}\label{sec1.6}
%

In what follows we shall explain the main strategy and key points in the proof of Theorem \ref{thm1.2}.

\medskip
\noindent(a) {\it General strategy.} First of all, given $M_0\geq 1$ which can be large such that 
$\|wf_0\|_{L^\infty}\leq M_0$ 
at time $t=0$, our final goal is to find a large enough time $T_0=T_0(M_0,\delta)$, depending only on $M_0$ and $\delta$, such that the solution $f(t,x,v)$, if it exists up to $T_0$, satisfies
\begin{equation}
\label{s.gs}
\|wf(T_0)\|_{L^\infty}\leq \delta,
\end{equation}  
at time $t=T_0$. Note that $\delta>0$ is a fixed constant in Proposition \ref{prop1.1}. One can then apply Proposition \ref{prop1.1} to the Boltzmann equation over $[T_0,\infty)$ to obtain the global existence of solutions and also the exponential time-decay as 
\begin{equation*}
\|wf(t)\|_{L^\infty}\leq C_0 \|wf(T_0)\|_{L^\infty} e^{-\vartheta t},
\end{equation*}
for all $t\geq T_0$. See the Figure 1 below.

\begin{figure}[h]
	\begin{tikzpicture}[xscale=1,yscale=0.8]
	\draw [fill=lightgray]  (0,0) circle [radius=0.8];
	\draw [fill] (0,0) circle [radius=0.05];
	\draw [->] (0,0)--(-0.8,0);
	\node [above left]  at (-0.35,0) {\small $\delta$};
	\draw [fill] (-5,0)  circle [radius=0.05];
	\draw [->]  (-5,0) to [out=170, in=-70] (-7.5,1.8) to [out=110, in=180] (-4,2.3);
	\draw [->]   (-4,2.3) to [out=0, in=155] (0,2.1) to [out=-25, in=135] (1.2,1.5) to [out=-45, in=95] (1.7,0) to [out=-85, in=35] (1,-1.2) to [out=-140, in=5] (0,-1.4);
	\draw [->]   (0,-1.4) to [out=177, in=-45] (-1.4,-1) to [out=135, in=-85] (-1.8,0) to [out=85, in=220] (-1.2,1) to [out=45, in=180] (0,1.25) 
	to [out=0, in=95] (0.7,0.7) to [out=-100, in=80] (0.44, 0.1) to [out=-90, in=60] (0.35,-0.15);
	\draw [dashed, thick,->]   (0.35,-0.15) to [out=-120, in=10] (-0.09,-0.47)
	to [out=180, in=-80] (-0.39,-0.2) to [out=160, in=90] (-0.5,0) to [out=100, in=225] (-0.4,0.3)to [out=40, in=120] (0,0.4) to [out=-5, in=90] (0.2,0.2)  to [out=-90, in=0] (0.025,0.025);
	\draw [fill] (0.44,0.1) circle [radius=0.03];
	\node [below right] at (-5,-0.2) {\small $\|wf_0\|_{L^\infty}\leq M_0$};
	\node [right] at (-5,0) {\small $t=0$};
	\node [right] at (0.45,0.1) {\small $t=T_0,\ \|wf(T_0)\|_{L^\infty}\leq \f\d2$};
	\node [below] at (0,0.08) {\small $0$};
	\end{tikzpicture}
	
	\
	
	\begin{tikzpicture}[xscale=1.3, yscale=1.2]
	\draw [black]  (0,0)--(4,0);
	\draw [thick]  (0,1)--(4,1);
	\draw [->] (0,0)--(0,2.5);
	\draw [dashed] (0,1)--(4,1);
	\node [left] at (0,1) {$\mu$};
	\node [left] at (0,2.5) {$F$};
	\draw[fill=lightgray] {plot [smooth] coordinates{ (0,1) (0.3,0.95) (0.4,1.05) (0.45,2)  (0.5,1) (0.9,1) (0.95,0)}}-- (1, 0)-- (1.05,1);
	\draw[fill=lightgray] {plot [smooth] coordinates{  (1.05,1) (1.2,1.1) (1.5,0.98) (1.85,1.05) (1.9,2) (1.95,1.1) (2,1.05) (2.2,1) (2.5,0.95) (2.55,2.5) (2.63,1) (3,0.95) (3.15,1) (3.2,0)}}--(3.24,0)--(3.25,1);
	\draw[fill=lightgray] plot [smooth] coordinates{(3.25,1) (3.4,1.08) (4,1)};
	\node [below] at (2,-0.3) {$t=0$};
	\draw [->] (4.3,1.2)--(5.3,1.2);
	\draw  (6,0)--(10,0);
	\draw [->] (6,0)--(6,2.5);
	\draw [thick] (6,1)--(10,1);
	\node [left] at (6,1) {$\mu$};
	\node [left] at (6,2.5) {$F$};
	\node [below] at (8,-0.3) {$t=T_0$};			
	\draw[fill=lightgray] plot [smooth] coordinates{ (6,1) (6.1,1.05) (6.3,0.95) (6.4,1) (6.5,1.1) (6.8,0.93) (7,1) (7.1,1.05) (7.5,0.95) (7.8,1.1) (8,1) (8.1,1.1) (8.2,0.93) (9.2,1.06) (9.7,0.93) (10,1)};
	\end{tikzpicture}
	\caption{The time evolution  of  Boltzmann solutions}
\end{figure}
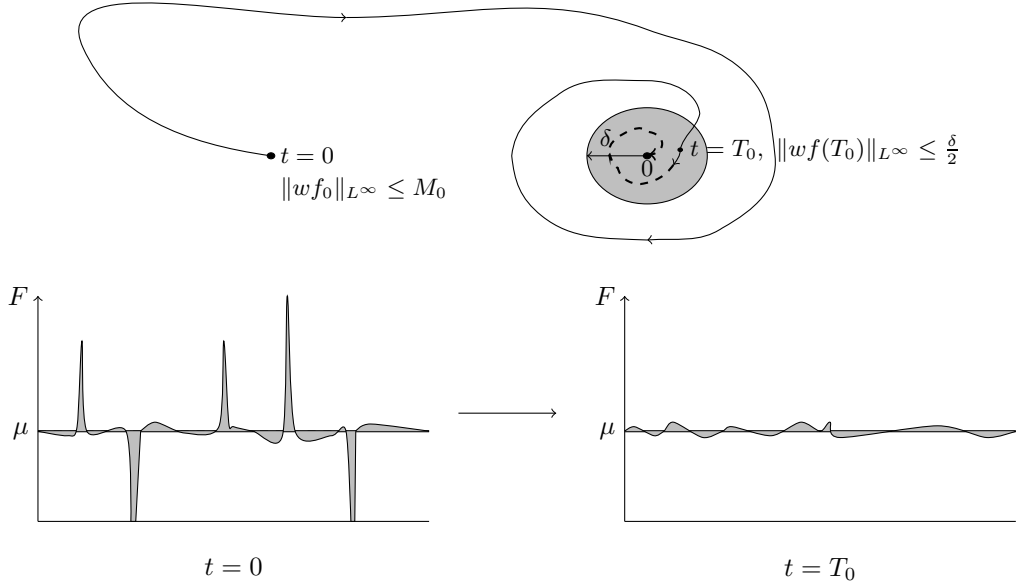

\medskip
\noindent{(b)} {\it $L^2$ estimate and time-growth}. Generally, it could be too hard to realize the final goal without postulating any smallness assumption on initial data. 
As mentioned before, an approach in the previous work \cite{DHWY} in terms of the smallness of both the relative entropy and the $L^1_xL^\infty_v$ norm for initial data was developed. One main idea there is to bound the possibly large $L^1_vL^\infty_x$ norm of the solution $f(t,x,v)$  when the time is on or after a generally small positive time slot $t_1$($\cong\frac{1}{1+M_0}$) which depends only on the $L^\infty$ norm of initial data in terms of the local-in-time existence result, namely,
$$
\int_{\R^3} |f(t,x,v)|\,\dd v\leq \int_{\R^3} e^{-\nu(v)t}|f_0(x-vt,v)|\,\dd v+\cdots\leq \frac{1}{t_1^3}\|f_0\|_{L^1_xL^\infty_v}+\cdots
$$
for any $t\geq t_1$ and $x\in \Omega=\R^3$, where $\|f_0\|_{L^1_xL^\infty_v}$ can be small enough. However, it seems difficult to carry the same procedure over to the current problem as the strategy we have to use to close the {\it a priori} estimates is based on a new Gronwall-type argument; see the item (d) below.
Inspired by \cite{Guo2}, we would make use of the interplay between $L^2$ norm and weighted $L^\infty$ norm for solutions. Thus, we assume 
\begin{equation*}
\|f_0\|_{L^2_{x,v}}\leq \eps_0,
\end{equation*}
where $\eps_0=\eps_0(M_0,\delta)$, depending only $M_0$ and $\delta$, is to be determined. Notice that since initial data can be large in the $L^\infty$ norm, the solution itself can also be large in the $L^\infty$ norm up to a finite time. Thus, it is hard to expect that solutions enjoy any time-decay property  either in $L^2$ or in $L^\infty$ even up to the same finite time. Indeed, the $L^2$ norm of the solution $f(t,x,v)$ could be increasing in time. Specifically, in terms of the elementary energy method, Lemma \ref{lem.l2} shows that 
\begin{equation*}
\|f(t)\|_{L^2}\leq \Psi\Big(\sup_{0\leq t< T}\{t\|wf(t)\|_{L^\infty}\}\Big) \|f_0\|_{L^2},
\end{equation*}
for all $0\leq t<T$. Here and in the sequel, $\Psi(\cdot)$ is a general function which is continuous, strictly positive and strictly increasing on $[0,\infty)$, and $T>0$ is an arbitrary time such that the solution exists over $[0,T)$.

\medskip
\noindent{(c)} {\it $L^\infty_xL^1_v$ estimate and time-decay in $L^\infty$}. We first explain the necessity of the $L^\infty_xL^1_v$ estimate. Setting $h(t,x,v)=wf(t,x,v)$, our $L^\infty$ estimate will be based on the mild formulation
 \begin{equation}\label{s.nmf}
h(t)={G^f(t,0)}h_0+\int_0^t{G^f(t,s)}K_wh(s)\,\dd s+\int_0^t{G^f(t,s})w\Gamma_+(f,f)(s)\,\dd s,
\end{equation}
according to the Boltzmann equation
\begin{equation}\nonumber
\pa_t h+v\cdot \na_x h +h \mathcal{A}F=K_wh+w\Ga_+(f,f).
\end{equation}
Here $\CA F(t,x,v)=\int_{\mathbb{R}^3}\int_{\mathbb{S}^2} B(v-u,\omega)F(t,x,v)\, \dd\omega \dd u$, and with given $F(t,x,v)$, ${G^f(t,s)}$ is the solution operator for the linear equation 
\begin{equation}
\label{s.eh}
\pa_t h+v\cdot \na_x h +h \mathcal{A}F=0,\quad {t>s,}
\end{equation}
with the corresponding diffuse reflection boundary. Note that we have put the loss term $\Ga_-(f,f)=f\CA(\sqrt{\mu}f)$ and the linear term $\nu(v)f=f\CA \mu$ together. Indeed, it is  essential to use \eqref{s.eh} when consider the existence of Boltzmann equation with large data. The other reason to do so is that we expect to bound the weighted nonlinear term in the pointwise sense by the product of the $L^\infty$ norm with the same weight and the $L^2_v$ norm with the fixed weight, see the key estimate \eqref{2.4} in Lemma \ref{lem2.2}, which can be regarded as an improvement of the product of the same two weighted $L^\infty$ norms as in \eqref{2.4-1}, but such estimate \eqref{2.4} is only true for the gain term, not for the loss term. 

The new reformulation \eqref{s.nmf} produces an additional difficulty in obtaining the linear time-decay of \eqref{s.eh}. Indeed, it is unclear that $\CA F(t,x,v)$ still admits a positive lower bound as for the collision frequency $\nu(v)=\CA \mu$ at the global Maxwellian. To overcome the extra difficulty, we observe that it is still possible to obtain the positive lower bound of $\CA F(t,x,v)$ over the time interval $[\tilde{t}, T_0]$ for some positive time {$\tilde{t}\cong 1+\ln{M_0}>0$} depending only on $M_0$, provided that $ f_0(x,v)$ is small enough in $L^2$ in the sense that $\|f_0\|_{L^2}\leq \eps_1=\eps_1(M_0,T_0)$ with $\eps_1$ depending only on $M_0$ and $T_0$; see Lemma \ref{lem5.3}. To achieve Lemma \ref{lem5.3}, as seen from its proof, the main goal is to obtain {a suitable upper bound of}
\begin{equation*}
\sup_{\tilde{t}\leq t\leq T_0}\sup_{x\in \Omega} \int_{\R^3} e^{-\frac{|v|^2}{8}} |h(t,x,v)|\,\dd v,
\end{equation*} 
in the mentioned setting. The estimate on the above term is based on the same mild formulation as in \cite{Guo2}:
 \begin{equation}\nonumber
h(t)=G(t)h_0+\int_0^tG(t-s)K_wh(s)\,\dd s+\int_0^tG(t-s)w\Gamma(f,f)(s)\,\dd s,
\end{equation}
where $G(t)$ is the linear solution operator for the equation $\pa_t h+v\cdot \na_x h+\nu h=0$ with the corresponding diffuse reflection boundary condition. 

An immediate and very important consequence of $L^\infty_xL^1_v$ estimate is the time-decay property of ${G^f(t,0)}$ up to the finite time $T_0$. Specifically, Lemma \ref{lem5.4} shows that 
\begin{equation*}
\|{G^f(t,0)}h_0\|_{L^\infty}\leq \Psi(M_0) e^{-\frac{1}{4}\nu_0 t}\|h_0\|_{L^\infty},
\end{equation*} 
for all $0\leq t\leq T_0=T_0(M_0,\delta)$, provided that $\|f_0\|_{L^2}\leq \eps_1(M_0,T_0)$. This time-decay property over the finite time interval is crucial for obtaining the global $L^\infty$ estimate.

\medskip
\noindent{(d)} {\it $L^\infty_{x,v}$ estimate}. As mentioned before, the weighted $L^\infty$ estimate is based on \eqref{s.nmf}. We must use not only the time-decay property of $G^f(t)$ in $L^\infty$ but also make the interplay between $L^\infty$ norm and $L^2$ norm. Indeed, as the first step for $L^\infty$ estimate, Lemma \ref{lem.ad1} shows that for each $(t,x,v)\in (0,T_0]\times \Omega\times \R^3$,
\begin{align}\label{s.h1}
|h(t,x,v)|
&\leq \Psi(M_0) e^{-\frac{\nu_0}{8} t}+\CJ(t)
+\Psi (M_0)\int_0^{t-\l}\Fi_{\{t_1\leq s\}} e^{-\f{\nu_0}{4}(t-s)}\Big|(K_wh)(s,x-v(t-s),v)\Big|\,\dd s\nonumber\\
&\quad+\Psi (M_0)\int_0^{t-\l}\Fi_{\{t_1\leq s\}} e^{-\f{\nu_0}{4}(t-s)}\Big|[w\Gamma_+(f,f)](s,x-v(t-s),v)\Big|\,\dd s.
\end{align}
Here, the second term on the right 
denotes
\begin{align*}
\CJ(t)=&
\psi\Big(\la,\varepsilon,\frac{C_{\varepsilon,T_0}}{N}\Big) \sup_{0\leq s\leq t}\left[\|h(s)\|_{L^\infty}+\|h(s)\|^2_{L^\infty}\right]\nonumber\\
&+C_{\varepsilon,N,T_0} \Psi\Big(\sup_{0\leq s\leq t}\|h(s)\|_{L^\infty}\Big)\left[\|f_0\|_{L^2}+\|f_0\|_{L^2}^2\right],
\end{align*}
where $\psi$ is a general nonnegative, continuous and strictly increasing function on $[0,\infty)$ with $\psi\to 0$ as all arguments tend to zero, $\la>0$ and $\v>0$ can be arbitrarily small, and $N>0$ can be arbitrarily large. 
Notice that if there were no last two integral terms on the right-hand side of \eqref{s.h1}, then in terms of 
$\|h(t)\|_{L^\infty}\leq \Psi(M_0) e^{-\frac{\nu_0}{8} t}+\CJ(t)$,
the $L^\infty$ bound over $[0,T_0]$ could have been closed under the smallness assumption of $\|f_0\|_{L^2}$. Indeed, for the first integral term on the right of \eqref{s.h1}, we have to make an iteration once to apply the pointwise estimate \eqref{s.h1} to $h$ so that it can be bounded by $\Psi(M_0) e^{-\frac{\nu_0}{8} t}+\CJ(t)$.   

The most difficult term comes from  the second integral term on the right of \eqref{s.h1}, since it involves the nonlinear quadratic term $\Ga_+(f,f)$. Generally it is hard to estimate it in case when the $L^\infty$ norm of solutions can be large. To bound this term, we develop a {\it nonlinear iteration procedure} for the gain term $\Ga_+(f,f)$ by making use of the important technical Lemma \ref{lem2.2},
so as to deduce that this term is bounded by
\begin{equation}
\Psi(M_0) e^{-\frac{\nu_0}{8} t}+\CJ(t)+\Psi(M_0) e^{-\frac{\nu_0}{8} t} 
\int_0^t \|h(s)\|_{L^\infty}\,\dd s. 
\nonumber
\end{equation}
Here we emphasize  that  such nonlinear iteration is crucially necessary in the proof, and the exponential time-decay coefficient $e^{-\frac{\nu_0}{8} t}$ in front of the integral term is a key point. Therefore, in the end we obtain that $\|h(t)\|_{L^\infty}$ is also bounded by the above term for all $0\leq t\leq T_0$.
The uniform-in-time {\it a priori} $L^\infty$ bound over $[0,T_0]$ for a suitably chosen time $T_0=T_0(M_0,\delta)$ follows by the Gronwall inequality under the smallness assumption of $\|f_0\|_{L^2}$, and \eqref{s.gs} can also be realized. For more details, refer to Section \ref{sec6}.

\medskip
\noindent{(e)} {\it Local-in-time existence}. The existence locally in time for solutions in the weighted $L^\infty$ norm is based on the iterated equations
\begin{equation}
\label{s.ite}
\pa_t F^{n+1} +v\cdot \na_x F^{n+1} +F^{n+1}\CA F^{n}=Q_+(F^n,F^n),
\end{equation} 
with initial data $F^{n+1}(0,x,v)=F_0(x,v)$ and the corresponding diffuse reflection boundary condition. The mild reformulations of those equations are given in terms of Lemma \ref{lem3.01}. The solvability of the linear equations is established in Lemma \ref{lem3.02}, and the existence of local-in-time solutions is then obtained as the limit of the sequence of approximation solutions. In the end we get 
\begin{equation}\label{local}
\sup_{0\leq t\leq \hat{t}_0}\|wf(t)\|_{L^\infty}\leq C \|wf_0\|_{L^\infty},
\end{equation}
where the time $\hat{t}_0$ depends only on the $\|wf_0\|_{L^\infty}$ with an explicit form, see Theorem \ref{LE}. 

We notice that the local existence of solutions to the Boltzmann equation with large data was also obtained in \cite[Lemma 7]{GKTT-IM} at the cost of losing velocity-weight. On the other hand, in our situation, the weight  $w(v)$ is kept in the estimate \eqref{local}, and this guarantees that one can extend the local solution to a global one with the help of the above {\it a priori} estimates.

\subsection{{Prospects}}\label{sec.pro}
{In this section we would make a few comments on the possible extensions of the current results. First of all, as mentioned before, it is also natural to study the problem on the large-amplitude initial data for other types of boundary conditions, such as in-flow, reserve reflection, and specular reflection, since the small-amplitude $L^\infty$ solutions were constructed in \cite{Guo2} for those boundary conditions; see also \cite{BG} for the mixed-type Maxwell boundary condition. Indeed, combined with \cite{Guo2,BG}, the extension of the current result on this direction seems to be doable in terms of the developed approach of the current work. Note that for the Boltzmann equation with specular reflection in convex domains, a new theory was developed in \cite{KD} under the assumption of the $C^3$ regularity of the domain, which is a great improvement of the boundary analyticity used in \cite{Guo2}. Thus, for possibly large initial data, it seems interesting to see if the current approach could be applied to the case of specular reflection boundary for the $C^3$ domain.}

{Moreover, motivated by the current work, we may also make a direct extension of the previous work \cite{DHWY} to the more general initial data for which only the smallness of $\CE(F_0)$ is needed but the smallness of the norm $\|f_0\|_{L^1_xL^\infty_v}$ can be disregarded. Indeed, the estimate in \cite[Lemma 2.7, page 385]{DHWY} is crucial for the integrability in $x$ variable when dealing with $L^\infty$ estimates, and thus the smallness of $\CE(F_0)$ should be enough to close the {\it a priori} estimates in terms of the Gronwall-type argument used in this work. Note that \cite{DHWY} only treats either the whole space or the periodic domain.  For the Boltzmann equation in general bounded domains with isothermal boundaries under the consideration of the current work, one still has the uniform-in-time bound on the relative entropy $$
\CE(F):=\int_{\Omega\times \R^3}\left(\frac{F}{\mu_{\theta}}\ln \frac{F}{\mu_{\theta}} -\frac{F}{\mu_{\theta}}+1\right)\mu_\theta \,\dd x\dd v,
$$
where $\mu_\theta=\frac{1}{2\pi \theta^2}e^{-\frac{|v|^2}{2\theta}}$ is the boundary equilibrium, cf.~\cite{Mi-VPB}, and thus it seems also possible to replace the small $L^2$ norm of initial data by the small nonnegative relative entropy $\CE(F_0)$. Note that such uniform-in-time  entropy bound fails for the non-isothermal boundary, cf.~\cite{AC}. In this sense, the approach in terms of the $L^2$-$L^\infty$ interplay could be more robust to deal with the general situation such as the diffuse reflection boundary with non-isothermal boundary temperature; see a recent work \cite{DHWZ} where $L^2$ is replaced by $L^p$ $(p<\infty)$ for a more general purpose.}

\subsection{Arrangement of the paper} The rest of the paper is arranged as follows. In Section \ref{sec2}, we present some preliminary facts on the reformulated Boltzmann equation. Particularly, we prove the key Lemma \ref{lem2.2} for the pointwise estimate on the gain term. In Section \ref{sec3}, we study the existence and continuity of solutions for the linear problem corresponding to the iteration scheme \eqref{s.ite}. The local existence is given in Section \ref{sec4}. Then, we give series of the {\it a priori} estimates in Section \ref{sec5} along the strategy described above and conclude the proof of global existence in Section \ref{sec6}. In the appendix Section \ref{sec7}, we list some basic properties obtained in \cite{Guo2} which will be used in the proof in the previous sections.

\subsection{Notations}
Throughout this paper, $C$ denotes a generic positive constant which may vary from line to line.  And $C_a,C_b,\cdots$ denote the generic positive constants depending on $a,~b,\cdots$, respectively, which also may vary from line to line. While $C_1, C_2,\cdots$ denote some specific positive constants.  $\|\cdot\|_{L^2}$ denotes the standard $L^2(\Omega\times\mathbb{R}^3_v)$-norm, and $\|\cdot\|_{L^\infty}$ denotes the $L^\infty(\Omega\times\mathbb{R}^3_v)$-norm.


\section{Preliminaries}\label{sec2}

\subsection{Reformulation}

Let $F(t,x,v)=\mu(v)+\sqrt{\mu(v)}f(t,x,v)$. Then, the Boltzmann equation \eqref{1.1}  is reformulated as 
\begin{equation}
\label{be.LG}
\partial_tf+v\cdot\nabla_xf+Lf=\Gamma(f,f).
\end{equation}
Here, $L$ is the linearized Boltzmann operator of the form $L=\nu(v)-K$, with 
\begin{equation}
\label{1.12-1}
\nu(v)=\int_{\mathbb{R}^3}\int_{\mathbb{S}^2}B(v-u,\omega)\mu(u)\,\dd\omega\dd u 
\sim (1+|v|)^{\k},
\end{equation} 
and $K=K_1-K_2$  defined by
\begin{align}
(K_1f)(t,x,v)&=\int_{\mathbb{R}^3}\int_{\mathbb{S}^2}B(v-u,\omega)\sqrt{\mu(v)\mu(u)}f(t,x,u)\,\dd\omega\dd u,\label{1.11}\\
(K_2f)(t,x,v)&=\int_{\mathbb{R}^3}\int_{\mathbb{S}^2}B(v-u,\omega)\sqrt{\mu(u)\mu(u')}f(t,x,v')\,\dd\omega\dd u \nonumber\\
&\quad+\int_{\mathbb{R}^3}\int_{\mathbb{S}^2}B(v-u,\omega)\sqrt{\mu(u)\mu(v')}f(t,x,u')\,\dd\omega\dd u.\label{1.12}
\end{align}
The nonlinear term $\Ga(f,f)=\Ga_+(f,f)-\Ga_-(f,f)$ is given by
\begin{align*}
\Gamma_+(f,f)
=\f1{\sqrt{\mu}}Q_+(\sqrt{\mu}f,\sqrt{\mu}f),\quad \Ga_-(f,f)=\f1{\sqrt{\mu}}Q_-(\sqrt{\mu}f,\sqrt{\mu}f).
\end{align*}
It is also convenient to rewrite \eqref{be.LG} as
\begin{align}\label{1.9}
\partial_tf+v\cdot\nabla_xf+f R(f)=Kf+\Gamma_+(f,f),
\end{align}
with
\begin{align}\label{1.9-1}
R(f)(t,x,v)=\int_{\mathbb{R}^3}\int_{\mathbb{S}^2} B(v-u,\omega)\left[\mu(u)+\sqrt{\mu(u)}f(t,x,u)\right] \dd\omega \dd u.
\end{align}
In terms of $f$,  we formulate the diffuse reflection boundary condition \eqref{diffuse1} as 
\begin{equation}\label{diffuse}
f(t,x,v)|_{\gamma _{-}}=c_{\mu }\sqrt{\mu (v)}\int_{v^{\prime }\cdot
	n(x)>0}f(t,x,v^{\prime })\sqrt{\mu (v^{\prime })}\{v^{\prime }\cdot
	n(x)\}\,\dd v^{\prime },
\end{equation}
for $(x,v)\in \ga_-$.

\subsection{Operator $K$}

Recall $K=K_1-K_2$ with \eqref{1.11} and \eqref{1.12}. One has 
\begin{align*}
(Kf)(v)=\int_{\mathbb{R}^3}k(v,\eta)f(\eta)\,\dd\eta,
\end{align*}
where $k(v,\eta)=k(\eta,v)$ is a symmetric integral kernel of $K$.

\begin{lemma}[\cite{Gra,Guo2}]
There is $C>0$ such that 
	\begin{align}\label{2.2}
	|k(v,\eta)|\leq C\left\{|v-\eta| +|v-\eta|^{-1} \right\}e^{-\f{|v-\eta|^2}{8}}e^{-\f{||v|^2-|\eta|^2|^2}{8|v-\eta|^2}},
	\end{align}
	for any $v, \eta\in \R^3$ with $v\neq \eta$. 
	Moreover, let $0\leq \varpi\leq \f1{64}$, $\al\geq 0$. There is $C_\a>0$ depending only on $\al$ such that  
	\begin{align}\label{2.3}
	\int_{\mathbb{R}^3} \left\{|v-\eta| +|v-\eta|^{-1} \right\}e^{-\f{|v-\eta|^2}{16}}e^{-\f{||v|^2-|\eta|^2|^2}{16|v-\eta|^2}} \f{e^{\varpi|v|^2}}{e^{\varpi|\eta|^2}}(1+|\eta|)^{-\a}\,\dd\eta\leq C_\a(1+|v|)^{-1-\a},
	\end{align}
	for any $v\in \R^3$.
\end{lemma}

\subsection{Pointwise estimate on gain term} Define the velocity weight function 
\begin{equation}\label{WF}
w=w(v):=(1+\r^2|v|^2)^{\beta}e^{\varpi|v|^2},
\end{equation}
with $\rho>0$, $0\leq \varpi\leq\f1{64}$ and $\beta>0$. The below lemma implies that one can bound in the pointwise sense the $w$-weighted gain term by the product of the $w$-weighted $L^\infty$ norm and the weighted $L^2$ norm with another fixed velocity weight. The estimate plays a crucial role in treating the nonlinear term whenever the solution could have a large amplitude.   


\begin{lemma}\label{lem2.2}
There is a generic constant $C_\b>0$ such that 
	\begin{align}\label{2.4}
	\left|w(v)\Gamma_{+}(f,f)(v)\right| \leq \frac{C_\b \|wf\|_{L_{v}^\infty}}{1+|v|} \left(\int_{\mathbb{R}^3}(1+|\eta|)^{4}|e^{\varpi|\eta|^2}f(\eta)|^2\,\dd\eta\right)^{\f12},
	\end{align}
	for all $v\in \R^3$.
	In particular, for $\rho\geq1$ and $\b\geq2$, one has 
	\begin{align}\label{2.4-1}
	\left|w(v)\Gamma_{+}(f,f)(v)\right| \leq \frac{C_\b \|wf\|^2_{L^\infty}}{1+|v|},
	\end{align}
	for all $v\in \R^3$, where the constant $C_\b>0$ is independent of $\rho$ and $\varpi$. 
\end{lemma}

\begin{proof}
By the energy conservation law in \eqref{1.3-1}, we see that
$
|v|^2\leq |u'|^2+|v'|^2.
$ 
This implies that 
\begin{equation}
\mbox{either}\quad \f12|v|^2\leq |u'|^2\quad \mbox{or}\quad \f12|v|^2\leq |v'|^2.\notag
\end{equation} 
Hence it holds that either $w(v)\leq 2^\b w(u')e^{\varpi|v'|^2}$ or $w(v)\leq 2^\b w(v') e^{\varpi|u'|^2}$.
Then it follows that
\begin{align}\label{2.6}
\Big|w(v)\Gamma_{+}(f,f)(v)\Big|&=
\f{w (v)}{\sqrt{\mu(v)}}\Big|Q_+(\sqrt{\mu}f,\sqrt{\mu}f)(v)\Big|\notag\\
&\leq   2^{\b}
\int_{\mathbb{R}^3}\int_{\mathbb{S}^2}B(v-u,\omega)\sqrt{\mu(u)}\Big|w(u')f(u')\cdot e^{\varpi|v'|^2}f(v')
\Big|\,\dd \omega \dd u \nonumber\\
&\quad+2^{\b}
\int_{\mathbb{R}^3}\int_{\mathbb{S}^2}B(v-u,\omega)\sqrt{\mu(u)}\Big|e^{\varpi|u'|^2}f(u')\cdot w(v')f(v')
\Big|\,\dd \omega \dd u.
\end{align}
We denote two integrals on the right as $I_1$ and $I_2$, respectively. 
To estimate $I_{1}$, 
as in \cite{Glassey},  
we rewrite \eqref{def.popr} as 
$u'=v+z_{\perp}$, 
$v'=v+z_{\shortparallel}$, 
with
$z=u-v$, 
$z_{\shortparallel}=(z\cdot\omega)\omega$,
$z_{\perp}=z-(z\cdot\omega)\omega$. 
Moreover, by \eqref{1.4}, one has 
\begin{equation}\nonumber
B(v-u,\omega)\leq C |v-u|^\ka \left|\frac{v-u}{|v-u|}\cdot \omega\right|\leq C |z_{\shortparallel}|\cdot |z|^{\ka-1}= C|z_{\shortparallel}|\left(|z_{\shortparallel}|^2+|z_{\perp}|^2\right)^{\frac{\ka-1}{2}}.
\end{equation}
Therefore, by making a change of variables $u\to z$, $I_1$ is bounded as
\begin{equation}\nonumber
I_{1}
\leq C_\b\|wf\|_{L^\infty_v}\int_{\mathbb{R}^3}\int_{\mathbb{S}^2}
|z_{\shortparallel}|\left(|z_{\shortparallel}|^2+|z_{\perp}|^2\right)^{\frac{\ka-1}{2}}
e^{-\f{|v+z|^2}{4}}|e^{\varpi|v+z_{\shortparallel}|^2}f(v+z_{\shortparallel})|\,\dd\omega \dd z.
\end{equation} 
By noting $\dd z=\dd z_\perp \dd |z_{\shortparallel}|$ and further making a change of variables $(\omega,|z_{\shortparallel}|)\to \eta:=v+z_{\shortparallel}$ with $\dd |z_{\shortparallel}|\dd \omega=\frac{1}{|z_{\shortparallel}|^2}\dd \eta$, it follows that
\begin{equation} \nonumber
I_1\leq C_\b\|wf\|_{L^\infty_v}\int_{\R^3_\eta}\frac{|e^{\varpi|\eta|^2}f(\eta)|}{|\eta-v|}\int_{\Pi_{v,\eta}}|z_\perp|^{\kappa-1}e^{-\f{|\eta+z_\perp|^2}{4}} \,\dd z_\perp \dd \eta,
\end{equation}
where $\Pi_{v,\eta}:=\{z_\perp\in \R^3: (\eta-v)\cdot z_\perp=0\}$ and we have used the identity $|v+z|=|\eta+z_\perp|$. As $0\leq \kappa\leq 1$, it holds that 
\begin{eqnarray}
\int_{\Pi_{v,\eta}}|z_\perp|^{\kappa-1}e^{-\f{|\eta+z_\perp|^2}{4}} \,\dd z_\perp 
&&= \left(\int_{|z_\perp|\leq 1}+\int_{|z_\perp|\geq 1}\right)|z_\perp|^{\kappa-1}e^{-\f{|\eta+z_\perp|^2}{4}} \,\dd z_\perp \nonumber\\
&&\leq \int_{|z_\perp|\leq 1} |z_\perp|^{\kappa-1} \,\dd z_\perp +\int_{|z_\perp|\geq 1}e^{-\f{|\eta+z_\perp|^2}{4}} \,\dd z_\perp \nonumber\\
&&\leq C,\nonumber
\end{eqnarray}
for a constant $C>0$.  Thus, it follows that
\begin{equation}\nonumber
I_1\leq C_\b\|w f\|_{L^\infty} \int_{\mathbb{R}^3}\f{|e^{\varpi|\eta|^2}f(\eta)|}{|\eta-v|}
\, \dd\eta.
\end{equation}
Using the Cauchy-Schwarz to deduce that
\begin{equation}\nonumber
\int_{\mathbb{R}^3}\f{|e^{\varpi|\eta|^2}f(\eta)|}{|\eta-v|}
\, \dd\eta\leq \left(\int_{\R^3} \frac{1}{|\eta-v|^2} \frac{1}{(1+|\eta|)^4}\dd \eta\right)^{1/2}\left(\int_{\R^3} (1+|\eta|)^4|e^{\varpi|\eta|^2}f(\eta)|^2\dd \eta\right)^{1/2},
\end{equation}
and noting that
\begin{equation}\nonumber
\int_{\R^3} \frac{1}{|\eta-v|^2} \frac{1}{(1+|\eta|)^4}\dd \eta\leq \frac{C}{(1+|v|)^2},
\end{equation}
for a constant $C>0$, it further follows that
\begin{align*}
I_{1}
\leq \frac{C_\b\|wf\|_{L^\infty_v}}{1+|v|}\left(\int_{\mathbb{R}^3}(1+|\eta|)^4|e^{\varpi|\eta|^2}f(\eta)|^2\dd\eta\right)^{\f12}.
\end{align*}
With \eqref{2.6}, this proves \eqref{2.4} for $I_1$ on one hand. 
On the other hand, for $I_2$ corresponding to the second term on the right hand of \eqref{2.6}, it is noticed that by the rotation, one can make an interchange of $v'$ and $u'$, and thus  change $I_{2}$  to the same form as $I_{1}$, cf~\cite{Glassey}.  Hence, one can bound $I_2$ as for $I_1$ above. Finally, \eqref{2.4-1} is an immediate consequence of \eqref{2.4} for $\rho\geq1$ and $\beta\geq 2$. This then completes the proof of Lemma \ref{lem2.2}.  
\end{proof}

\subsection{Back-time cycle and relative definitions}

For each $x\in \pa\Omega$, we introduce the velocity space for the outgoing particles
\begin{equation}
\label{def.snu}
\mathcal{V}(x)=\{v'\in\mathbb{R}^3:~v'\cdot n(x)>0\},
\end{equation}
associated with the probability measure $\dd\s=\dd\s(x)$ by
\begin{align}\label{def.bm}
	\dd\s(x)=c_\mu \mu(v')\{v'\cdot n(x)\}\,\dd v'.
\end{align}
%

Given $(t,x,v),$ let $[X(s),V(s)]$ 
be the backward bi-characteristics for the Boltzmann equation, which is determined by 
\begin{equation*}
\left\{\begin{aligned}
&\frac{dX(s)}{ds}=V(s),\\
&\frac{dV(s)}{ds}=0,\\
&[X(t),V(t)]=[x,v].
\end{aligned}\right.
\end{equation*}
The solution is then given by
\begin{equation}
\label{def.bic}
[X(s),V(s)]=[X(s;t,x,v),V(s;t,x,v)]=[x-(t-s)v,v].
\end{equation}

For each $(x,v)$ with $x\in \overline{\Omega}$ and $v\neq 0,$ we define its {\it backward exit time} $t_{\mathbf{b}}(x,v)\geq 0$ to be the the last moment at which the
back-time straight line $[X(s;0,x,v),V(s;0,x,v)]$ remains in $\overline{\Omega}:$ 
\begin{equation}\label{exit}
t_{\mathbf{b}}(x,v)=\inf \{\tau \geq 0:x-\tau v\notin\Omega\}.
\end{equation}
We therefore have $x-t_{\mathbf{b}}v\in \partial \Omega $ and $\xi (x-t_{\mathbf{b}}v)=0.$ We also define 
\begin{equation}\label{xb}
x_{\mathbf{b}}(x,v)=x(t_{\mathbf{b}})=x-t_{\mathbf{b}}v\in \partial \Omega .
\end{equation}
Note that $v\cdot n(x_{\mathbf{b}})=v\cdot n(x_{\mathbf{b}}(x,v)) \leq 0$ always holds true.

Let $x\in \overline{\Omega}$, $(x,v)\notin \gamma _{0}\cup \g_{-}$ and 
$
(t_{0},x_{0},v_{0})=(t,x,v)$. For $v_{k+1}\in \mathcal{V}_{k+1}:=\{v_{k+1}\cdot n(x_{k+1})>0\}$, the back-time cycle is defined as 
\begin{equation}\label{2.18}
\left\{\begin{aligned}
X_{cl}(s;t,x,v)&=\sum_{k}\mathbf{1}_{[t_{k+1},t_{k})}(s)\{x_{k}-v_k(t_{k}-s)\},\\[1.5mm]
V_{cl}(s;t,x,v)&=\sum_{k}\mathbf{1}_{[t_{k+1},t_{k})}(s)v_{k},
\end{aligned}\right.
\end{equation}
with
\begin{equation}\label{diffusecycle}
(t_{k+1},x_{k+1},v_{k+1})=(t_{k}-t_{\mathbf{b}}(x_{k},v_{k}),x_{\mathbf{b}}(x_{k},v_{k}),v_{k+1}).
\end{equation}
For $k\geq 2$, 
the iterated integral  means that 
\begin{equation*}
\int_{\Pi _{l=1}^{k-1}\mathcal{V}_{l}}\Pi _{l=1}^{k-1}\dd\sigma _{l}:=
\int_{\mathcal{V}_{1}}\cdots\left\{ \int_{\mathcal{V}_{k-1}}\,\dd\sigma_{k-1}\right\} \cdots\dd\sigma _{1}.  
\end{equation*}
Note that all $v_{l}$ ($l=1,2,\cdots$) are independent variables, and $t_k$, $x_k$ depend on $t_l$, $x_l$, $v_l$ for $l\leq k-1$, and the velocity space $\mathcal{V}_{l}$ implicitly depends on $(t,x,v,v_{1},v_{2},\cdots,v_{l-1}).$  

\section{Linear problem}\label{sec3}

In the next section we will consider the local existence of the Boltzmann equation with large-amplitude initial data under the diffuse reflection boundary condition. For that, we are concerned in this section with the linear problem only. 

\subsection{Mild form with variable collision frequency}

Recall the homogeneous transport equation \eqref{lhe.nu} with the collision frequency $\nu(v)$ which is induced by the global Maxwellian $\mu(v)$ as in \eqref{1.12-1}. In order to avoid treating the loss term in the nonlinear $L^\infty$ estimate, we combine $\Ga_-(f,f)$ together with $\nu(v)f$ so as to deal with \eqref{1.9} with the collision frequency given by $R(f)$ in \eqref{1.9-1} which is equivalent with
\begin{equation}
R(f)(t,x,v)=\int_{\R^3}\int_{\S^2} B(v-u,\omega) F(t,x,v)\,\dd\omega\dd u \geq 0,\nonumber
\end{equation} 
as $F(t,x,v)\geq 0$. Thus, to treat weighted $L^\infty$ estimates on \eqref{1.9}, we start from
 the following linear  transport equation with a nonnegative variable collision frequency function  
\begin{align}\label{3.1}
\partial_th+v\cdot\nabla h+h R(\varphi)=q(t,x,v),
\end{align}
where $\varphi=\varphi(t,x,v)$ is a given function satisfying 
\begin{equation}\label{3.3}
\mu(v)+\sqrt{\mu(v)}\varphi(t,x,v)\geq 0,\quad \|\varphi(t)\|_{L^\infty}<\infty.
\end{equation}

\begin{lemma}\label{lem3.01}
Let $\varphi(t,x,v)$ satisfy \eqref{3.3}, and $q(t,x,v)\in L^\infty$ be given. Assume that $h(t,x,v)\in L^\infty$ is a solution to  \eqref{3.1}
supplemented with the diffuse reflection boundary condition
\begin{align}\label{lp.bc}
h(t,x,v)|_{\g_-}=\f{1}{\widetilde{w}(v)}\int_{\mathcal{V}(x)}h(t,x,v')\widetilde{w}(v')\,\dd\sigma,
\end{align}
for $(x,v)\in \ga_-$, where $\widetilde{w}(v):=\f1{w(v)\sqrt{\mu(v)}}$. Recall the diffusive back-time cycles in \eqref{2.18} and \eqref{diffusecycle}. Define
\begin{equation}\label{def.iv}
I^\varphi(t,s):=\exp\left\{-\int_{s}^tR(\varphi)\big(\tau,X_{cl}(\tau),V_{cl}(\tau)\big)\,\dd\tau\right\}.
\end{equation}
Then for any $0\leq s\leq t$, for almost every $x,v$, if $t_1(t,x,v)\leq s$, we have 
\begin{align*}
h(t,x,v)=I^\varphi(t,s)h(s,x-v(t-s),v)+\int_s^tI^\varphi(t,\tau)q(\tau,x-v(t-\tau),v)\,\dd\tau. 
\end{align*}
If $t_1(t,x,v)>s$, then we have that for $k\geq2$,
\begin{align}\label{3.5}
h(t,x,v)=&\int_{t_{1}}^{t}I^\varphi(t,\tau)q(\tau ,x-v(t-\tau ),v)\,\dd\tau\nonumber\\
&+\frac{I^\varphi(t,t_1)}{\widetilde{w}(v)}\int_{\Pi _{j=1}^{k-1}\mathcal{V}_{j}}\sum_{l=1}^{k-1}\mathbf{1}_{\{t_{l+1}\leq s<t_l\}}h(s,x_{l}-(t_{l}-s)v_{l},v_{l})\,\dd\Sigma^\varphi_{l}(s) \nonumber\\
&+\frac{I^\varphi(t,t_1)}{\widetilde{w}(v)}\int_{\Pi _{j=1}^{k-1}\mathcal{V}_{j}}\sum_{l=1}^{k-1}\int_{s}^{t_{l}}\mathbf{1}_{\{t_{l+1}\leq s<t_l\}}q(\tau ,x_{l}-(t_{l}-\tau)v_{l},v_{l})\,\dd\Sigma^\varphi _{l}(\tau)\dd \tau  \nonumber\\
&+\frac{I^\varphi(t,t_1)}{\widetilde{w}(v)}\int_{\Pi _{j=1}^{k-1}\mathcal{V}_{j}}\sum_{l=1}^{k-1}\int_{t_{l+1}}^{t_{l}}\mathbf{1}_{\{t_{l+1}>s\}}q(\tau,x_{l}-(t_{l}-\tau)v_{l},v_{l})\,\dd\Sigma^\varphi _{l}(\tau)\dd\tau  \nonumber \\
&+\frac{I^\varphi(t,t_1)}{\widetilde{w}(v)}\int_{\Pi _{j=1}^{k-1}\mathcal{V}_{j}}\mathbf{1}_{\{t_{k}>s\}}h(t_{k},x_{k},v_{k-1})\,\dd\Sigma^\varphi _{k-1}(t_k),
\end{align}
where
$\dd\Sigma^\varphi_{l}(s)$ is given by
\begin{equation*}
\dd\Sigma^\varphi _{l}(s):=\{\Pi _{j=l+1}^{k-1}\dd\sigma _{j}\}\cdot \widetilde{w}(v_l) I^\varphi(t_l,s)\dd\s_l\cdot\Pi_{j=1}^{l-1}\{I^\varphi(t_{j},t_{j+1})\dd\s_j\},
\end{equation*}
and $\dd\Sigma^\varphi _{k-1}(t_k)$ is the value of $\dd\Sigma^\varphi _{k-1}(s)$ at $s=t_k$.

\end{lemma}

\begin{proof}
It is similar to \cite[Lemma 24]{Guo2}, so we omit the details here. 
\end{proof}

\subsection{
Existence for inflow boundary}

For later use we still consider \eqref{3.1} but supplemented with the inflow boundary condition.

%

\begin{lemma}\label{lemA1}
Let $\varphi(t,x,v)$ satisfy \eqref{3.3}, and $q(t,x,v)=0$. Let $h(t,x,v)$ be the solution to  the following initial-boundary value problem:
\begin{equation}\label{lemA1.p}
\left\{\begin{aligned}
&\partial_th+v\cdot\nabla_x h+h R(\varphi)=0,\\
&h(t,x,v)|_{t=0}=h_0(x,v)\in L^\infty_{x,v},\\
&h(t,x,v)|_{\gamma_-}=g(t,x,v)\in L^\infty(\gamma_-).
\end{aligned}\right.
\end{equation}
Then, for $(t,x,v)\in (0,\infty)\times \overline{\Omega}\times \R^3$ with $(x,v)\notin \gamma_0\cup \gamma_-$, it holds that 
\begin{align}\label{A1}
h(t,x,v)&=\Fi_{t-t_{\bf b}< 0}\   h_0(x-vt,v) \exp\left\{-\int_0^tR(\varphi)(\tau,x-v(t-\tau),v)\,\dd\tau\right\}\nonumber\\
&\quad+I_{t-t_{\bf b}\geq 0} \ g(t-t_{\bf b}, x_{\bf b}, v) \exp\left\{-\int_{t-t_{\bf b}}^tR(\varphi)(\tau,x-v(t-\tau),v)\,\dd\tau\right\},
\end{align}
where $t_{\bf b}=t_{\bf b}(x,v)$ and $x_{\bf b}=x_{\bf b}(x,v)$ are  defined in \eqref{exit} and \eqref{xb}, respectively. Moreover, it holds that 
\begin{align}\label{A2}
\|h(t)\|_{L^\infty}\leq \|h_0\|_{L^\infty}+\sup_{0\leq s\leq t}\|g(s)\|_{L^\infty},
\end{align}
for all $t\geq 0$.
\end{lemma}

\begin{proof}
For almost every $x,v$, along the bicharacteristics \eqref{def.bic}, 
one may write the first equation in \eqref{lemA1.p} as 
\begin{align}\label{A3}
\f{d}{ds}\left[h(s,x-v(t-s),v)\exp\Big\{-\int_0^s R(\varphi)(\tau,x-v(t-\tau),v)\,\dd\tau\Big\} \right]=0.
\end{align}
If $t-t_{\bf b}< 0$, then the backward trajectory first hits the initial plane. Thus, integrating the above equation over $s\in[0,t]$, one obtains the first part on the right hand of \eqref{A1}.
On the other hand, if $t-t_{\bf b}\geq 0$, then the  backward trajectory first hits the boundary. Integrating \eqref{A3} over $s\in [t-t_{\bf b}, t]$ yields the second part of the right hand of \eqref{A1}. 
Finally, the estimate \eqref{A2} follows directly from \eqref{A1} under the assumption $R(\varphi)\geq 0$ in  \eqref{3.3}. This then completes the proof of Lemma \ref{lemA1}.
\end{proof}

\subsection{
Existence for diffusive boundary}

We treat $L^\infty$ estimates on solutions to  the initial-boundary value problem for the linear equation \eqref{3.1} with the diffuse reflection boundary condition \eqref{lp.bc} under the assumption \eqref{3.3}. To do that, we first use Lemma \ref{lemA1} to establish an abstract lemma.

\begin{lemma}\label{lemA2}
Let $\varphi(t,x,v)$ satisfy \eqref{3.3}, and $q(t,x,v)=0$.  Let $\mathcal{M}: L^\infty(\gamma_+)\rightarrow L^\infty(\gamma_-)$ be a bounded linear operator with $\|\mathcal{M}\|_{\mathcal{L}(L^\infty_{\ga_+},L^\infty_{\ga_-})}=1$. Then, for each $0<\v<1$, there exist a unique solution $h(t,x,v)\in L^\infty$ with  $h_\gamma\in L^\infty$ for the following initial-boundary value problem 
\begin{equation}\label{lemA2.lp.dbc}
\left\{\begin{aligned}
&\partial_th+v\cdot\nabla_x h+h R(\varphi)=0, \\ %
& h(t,x,v)|_{t=0}=h_0(x,v),\\
&h_{\gamma_-}=(1-\v)\mathcal{M} h_{\gamma_+}. 
\end{aligned}\right.
\end{equation}
\end{lemma}

\begin{proof}
Take $0<\v<1$. We consider the approximation solution sequence $\{h^m(t,x,v)\}$ defined by 
\begin{equation}\label{A6}
\left\{\begin{aligned}
&\left\{\partial_t +v\cdot\nabla_x +R(\varphi)\right\}h^{m+1} =0, \\
& h^{m+1}(t,x,v)|_{t=0}=h_0(x,v),\\
& h_{\gamma_-}^{m+1}=(1-\v)\mathcal{M} h^m_{\gamma_+}, 
\end{aligned}\right.
\end{equation}
with $m=0,1,2,\cdots$, and $h^0_{\gamma_+}\equiv 0$. Note that for each $m\geq 0$, \eqref{A6} is an inflow boundary-value problem, so the existence of solutions to the problem \eqref{A6} follows directly from Lemma \ref{lemA1}. Thus, $\{h^m(t,x,v)\}$ is well-defined. 

Next we claim that $h^m$ and $h_{\gamma}^m$ are Cauchy in $L^\infty$. Indeed, from \eqref{A6}, one has 
\begin{equation}\label{A7}
\left\{\begin{aligned}
&\left\{\partial_t +v\cdot\nabla_x +R(\varphi)\right\}(h^{m+1}-h^m) =0, \\ 
& h_{\gamma_-}^{m+1}-h_{\gamma_-}^{m}=(1-\v)\mathcal{M} \left[h^m_{\gamma_+}-h^{m-1}_{\gamma_+}\right] , \\  
& (h^{m+1}-h^m)(t,x,v)|_{t=0}=0.
\end{aligned}\right.
\end{equation}
Then, by Lemma \ref{lemA1}, 
\begin{equation}\nonumber
\sup_{s}\|(h^{m+1}_{\gamma_+}-h^{m}_{\gamma_+})(s)\|_{L^\infty}\leq \sup_{s}\|(h^{m+1}_{\gamma_-}-h^{m}_{\gamma_-})(s)\|_{L^\infty}.
\end{equation}
Using the second equation of \eqref{A7} and also the assumption that $\|\mathcal{M}\|_{\mathcal{L}(L^\infty_{\ga_+},L^\infty_{\ga_-})}=1$, it follows that 
\begin{equation}\nonumber
\sup_{s}\|(h^{m+1}_{\gamma_-}-h^{m}_{\gamma_-})(s)\|_{L^\infty}\leq (1-\v) \sup_{s}\|(h^{m}_{\gamma_+}-h^{m-1}_{\gamma_+})(s)\|_{L^\infty}.
\end{equation}
Combining both estimates above, it holds that
\begin{equation}\nonumber
\sup_{s}\|(h^{m+1}_{\gamma_+}-h^{m}_{\gamma_+})(s)\|_{L^\infty}\leq (1-\v) \sup_{s}\|(h^{m}_{\gamma_+}-h^{m-1}_{\gamma_+})(s)\|_{L^\infty},
\end{equation}
which by induction further gives
\begin{align}\nonumber
\sup_{s}\|(h^{m+1}_{\gamma_+}-h^{m}_{\gamma_+})(s)\|_{L^\infty}
\leq (1-\v)^{m}\sup_{s}\|h^{1}_{\gamma_+}(s)\|_{L^\infty}.
\end{align}
Hence, $\{h^m_{\gamma_+}\}_{m=1}^{\infty}$ is a Cauchy sequence in $L^\infty(\mathbb{R}_+\times \gamma_{+})$. Again by applying Lemma \ref{lemA1}, one sees that both $\{h^m_{\gamma_-}\}_{m=1}^{\infty}$ and $\{h^m\}_{m=1}^{\infty}$ are Cauchy in $L^\infty$. Then, the existence of solutions to \eqref{lemA2.lp.dbc} follows by taking $m\rightarrow\infty$.   This completes the proof of Lemma \ref{lemA2}.
\end{proof}

\begin{lemma}\label{lem3.02}
Let $\r>1$ be suitably large and $\beta>3/2$ in the velocity weight function  \eqref{WF}. 
Assume $h_{0}=h_0(x,v)\in L^{\infty }$ and $q(t,x,v) \equiv 0$.
Then there exists a unique solution $h(t)={{G}^\varphi(t,s)}h_0$ to the equation \eqref{3.1} with {$h(t)|_{t=s}=h_0$} and the diffuse reflection boundary condition \eqref{lp.bc} under the assumption \eqref{3.3} on $\varphi(t,x,v)$. 
Furthermore, there is a constant $C_3\geq 1$  independent of $\rho$ such that 
\begin{align}\label{3.9}
\|{{G}^\varphi(t,s)}h_0 \|_{L^\infty} \leq C_3\r^{2\b} \|h_0\|_{L^\infty}, 
\end{align}
for all $0\leq s\leq t\leq \rho$.
\end{lemma}

\begin{proof}
{We only construct the solution operator $G^\varphi(t,0)$ because $G^\varphi(t,s)$ with $t>s>0$ can be constructed in a similar way.} We firstly construct the approximate solution sequence $\{h^m(t,x,v)\}_{m\geq 2}$ by solving the following initial-boundary value problems
\begin{equation}\label{3.10}
\left\{\begin{aligned}
&\partial_t h^m+v\cdot\nabla h^m+h^m R(\varphi)=0,\\
&h^m (t,x,v)|_{\g_-}=\left(1-\f1m\right)\f1{\widetilde{w}(v)} \int_{\mathcal{V}} h^m(t,x,v')\widetilde{w}(v')\,\dd\s(x),\quad (x,v)\in \ga_-,\\
&h^m(t)|_{t=0}=h^m_0:=h_0 \Fi_{\{|v|\leq m\}}.
\end{aligned}\right.
\end{equation}
To solve \eqref{3.10} for each $m\geq 2$, we 
introduce $\widetilde{h}^m(t,x,v):=\widetilde{w}(v)h^m(t,x,v)$ which satisfies 
\begin{equation}\label{3.11}
\left\{\begin{aligned}
&\partial_t\widetilde{h}^m+v\cdot\nabla \widetilde{h}^m+\widetilde{h}^m R(\varphi)=0,\\
&\widetilde{h}^m (t,x,v)|_{\g_-}=\left(1-\f1m\right)\int_{\mathcal{V}} \widetilde{h}^m(t,x,v')\,\dd\s(x),\quad (x,v)\in \ga_-,\\
&\widetilde{h}^m(t)|_{t=0}=\widetilde{h}_0^m:=\widetilde{w}(v)h_0 \Fi_{\{|v|\leq m\}}.
\end{aligned}\right.
\end{equation}
It is direct to see that the boundary operator maps $L^\infty_{\ga_+}$ to $L^\infty_{\ga_-}$ with the norm $1-\f1m<1$, and the initial data $\widetilde{h}^m_0$ satisfy
\begin{align*}
\|\widetilde{h}_0^m\|_{L^\infty}\leq C(m) \|h_0\|_{L^\infty}<\infty,
\end{align*}
for a constant $C(m)$ depending on $m$.
From Lemma \ref{lemA2} above, there exists a unique solution $\widetilde{h}^m(t,x,v)\in L^\infty$ to the IBVP  \eqref{3.11}. Thus we have constructed an approximate solution sequence $h^m(t,x,v)=\f1{\widetilde{w}(v)}\widetilde{h}^m(t,x,v)$ to the IBVP \eqref{3.10}.

In order to take the limit $m\rightarrow \infty$, we need to obtain a uniform $L^\infty$ bound for 
$h^m(t,x,v)$.
Let $(t,x,v)\in (0,\infty)\times \overline\Omega\times \R^3$ with 
$(x,v)\notin\g_0\cup\g_-$. We divide the proof by two cases. 

\medskip
\noindent{\it Case 1:} $t_1(t,x,v)\leq 0$. Then it holds that $\Fi_{\{t_1\leq 0\}}h^m(t,x,v)=I^\varphi(t,0)h_0^m(x-vt,v)$, and 
\begin{align}\label{3.13}
\|\Fi_{\{t_1\leq 0\}}h^m(t)\|_{L^\infty}\leq \|h_0\|_{L^\infty}.
\end{align}

\noindent{\it Case 2:} $t_1(t,x,v)> 0$. Then the back-time trajectory first hits the boundary. It follows from \eqref{3.5} with $q(t,x,v)=0$ that 
\begin{align}\label{3.14}
\left|\Fi_{\{t_1>0\}}h^m(t,x,v)\right|\leq&\frac{I^\varphi(t,t_1)}{\widetilde{w}(v)}\int_{\Pi _{j=1}^{k-1}\mathcal{V}_{j}}\sum_{l=1}^{k-1}\mathbf{1}_{\{t_{l+1}\leq 0<t_l\}}|h^m_0(x_{l}-t_{l}v_{l}, v_{l})| \,\dd\Sigma^\varphi _{l}(0) \nonumber\\
&+\frac{I^\varphi(t,t_1)}{\widetilde{w}(v)}\int_{\Pi _{j=1}^{k-1}\mathcal{V}_{j}} \Fi_{\{t_k>0\}} \left|h^m(t_k,x_{k},v_{k-1})\right| \,\dd\Sigma^\varphi _{k-1}(t_k).
\end{align}
Using the boundary condition \eqref{3.10},  the second term on the right-hand side of \eqref{3.14} can be bounded by
\begin{align}
&\frac{I^\varphi(t,t_1)}{\widetilde{w}(v)}\int_{\Pi _{j=1}^{k}\mathcal{V}_{j}} \Fi_{\{t_{k+1}\leq0<t_k\}} |h^m_0(x_{k}-v_kt_k,v_{k})| \,\dd\Sigma^\varphi _{k}(0) \nonumber\\
&\qquad\qquad+\frac{I^\varphi(t,t_1)}{\widetilde{w}(v)}\int_{\Pi _{j=1}^{k}\mathcal{V}_{j}} \Fi_{\{t_{k+1}>0\}} |h^m(t_k,x_{k},v_{k})| \,\dd\Sigma^\varphi _{k}(t_k).\nonumber
\end{align}
Substituting 
the above bound back into \eqref{3.14}, one obtains that
\begin{align}\label{3.16}
\left|\Fi_{\{t_1>0\}}h^m(t,x,v)\right|&\leq\frac{I^\varphi(t,t_1)}{\widetilde{w}(v)}\int_{\Pi _{j=1}^{k}\mathcal{V}_{j}}\sum_{l=1}^{k}\mathbf{1}_{\{t_{l+1}\leq 0<t_l\}}|h^m_0(x_{l}-t_{l}v_{l}, v_{l})| \,\dd\Sigma^\varphi _{l}(0) \nonumber\\
&\qquad+\frac{I^\varphi(t,t_1)}{\widetilde{w}(v)}\int_{\Pi _{j=1}^{k}\mathcal{V}_{j}} \Fi_{\{t_{k+1}>0\}} |h^m(t_k,x_{k},v_{k})| \,\dd\Sigma^\varphi _{k}(t_k)\nonumber\\
&:=J_1+J_2,
\end{align}
where $J_1$, $J_2$ denote two terms on the right, respectively. 
Recall \eqref{4.8} in the Appendix and note that
\begin{equation}\nonumber
\left|\Fi_{\{t_{k+1}>0\}}h^m(t_k,x_k,v_k)\right|\leq \sup_{y,u}\left|\Fi_{\{t_{1}(t_k,y,u)>0\}}h^m(t_k,y,u)\right|.
\end{equation}
Then, for $J_2$, it follows that
\begin{equation}\nonumber
J_2\leq\f{1}{\widetilde{w}(v)}\sup_{y,u}\left|\Fi_{\{t_{1}(t_k,y,u)>0\}}h^m(t_k,y,u)\right| \int_{\Pi _{j=1}^{k}\mathcal{V}_{j}} \Fi_{\{t_{k+1>0}\}} \widetilde{w}(v_k) \,\dd\s_k\cdots \dd\s_1.
\end{equation}
Using the facts that
\begin{align*}
\f{1}{\widetilde{w}(v)}\leq C_\b \r^{2\b}(1+|v|)^{-1}, 
\end{align*}
and 
\begin{equation*}
\int_{\mathcal{V}(x)} \widetilde{w}(v)\,\dd\s(x) \leq C_\b \r^{-3},
\end{equation*}
and taking 
$k=C_1\r^{\f54}$,
one can further bound $J_2$ as
\begin{align}\label{3.17}
J_2
&\leq C\r^{2\b}\sup_{0\leq s\leq t\leq \r}\|\Fi_{\{t_{1}>0\}}h^m(s)\|_{L^\infty} \int_{\Pi _{j=1}^{k-1}\mathcal{V}_{j}} \Fi_{\{t_{k}>0\}}\, \dd\s_{k-1}\cdots \dd\s_1\nonumber\\
&\leq C_4\r^{2\b-3}\left(\f12\right)^{C_2\r^{\f54}}\sup_{0\leq s\leq t\leq \r}\|\Fi_{\{t_{1}>0\}}h^m(s)\|_{L^\infty},
\end{align}
where $C_4\geq 1$ is some constant. 
Next, for $J_1$, similarly as for treating $J_1$ before, it can be bounded as
\begin{align}
J_1&\leq \f{1}{\widetilde{w}(v)} \int_{\Pi _{j=1}^{k}\mathcal{V}_{j}}\sum_{l=1}^{k}\mathbf{1}_{\{t_{l+1}\leq 0<t_l\}}|h^m_0(x_{l}-t_{l}v_{l}, v_{l})| \widetilde{w}(v_l)\, \dd\s_k\cdots \dd\s_l\cdots \dd\s_1\nonumber\\
&\leq C\r^{2\b} \|h_0\|_{L^\infty} \Pi _{j=1}^{k} \int_{\mathcal{V}_j} \left[1+\widetilde{w}(v_j)\right]\, \dd\s_j,\nonumber
\end{align}
and hence it further follows that
\begin{equation}\label{3.19}
J_1\leq C\r^{2\b} \left(1+\f{C_\b}{\r^3}\right)^{k}\|h_0\|_{L^\infty} 
=C\r^{2\b} \left(1+\f{C_\b}{\r^3}\right)^{C_2\r^{\f54}}\|h_0\|_{L^\infty} 
\leq C\r^{2\b}\|h_0\|_{L^\infty}.
\end{equation}
Substituting  \eqref{3.17} and \eqref{3.19} into \eqref{3.16}, one obtains that 
\begin{align}
\|\Fi_{\{t_1>0\}}h^m(t,x,v)\|_{L^\infty}&\leq C_4\r^{2\b-3}\left(\f12\right)^{C_2\r^{\f54}}\sup_{0\leq s\leq \r}\|\Fi_{\{t_{1}>0\}}h^m(s)\|_{L^\infty}+C\r^{2\b}\|h_0\|_{L^\infty}.\nonumber
\end{align}
Taking $\r$ large enough such that 
\begin{equation}\nonumber
C_4\r^{2\b-3}\left(\f12\right)^{C_2\r^{\f54}}\leq \f12,
\end{equation}
 one then gets that 
\begin{equation}\label{3.21}
\|\Fi_{\{t_1>0\}}h^m(t,x,v)\|_{L^\infty}\leq C_3\r^{2\b}\|h_0\|_{L^\infty},
\end{equation}
for all $0\leq t\leq \r$. 
Based on the uniform estimate \eqref{3.21} as well as \eqref{3.13} obtained before, it holds that 
$h^m(t,x,v)\rightarrow h(t,x,v)$ a.e.~$x$, $v$ by using \eqref{3.5} and the Lebesgue convergence Theorem.  Thus one can conclude the proof of Lemma \ref{lem3.02}
by taking $m\rightarrow\infty$.   
\end{proof}

\subsection{Exponential time-decay in $L^\infty$ }

Note that Lemma \ref{lem3.02} is used  for  the linear $L^\infty$ estimates in finite time. In the long time, the behavior of ${G^\varphi(t,0)}$ is given by the following lemma which will be used in Section \ref{sec5} later.
  

\begin{lemma}\label{lem3.3}
Let all the assumptions in Lemma \ref{lem3.02} hold. We further suppose that 
\begin{equation}\label{3.22}
R(\varphi)(t,x,v)\geq \f12\nu(v),
\end{equation}
for all $(t,x,v)\in [0,\infty)\times \Omega\times \R^3$.
Then it holds that 
\begin{align}\label{3.23}
\|{{G}^\varphi(t,s)}h_0 \|_{L^\infty} \leq C_3\r^{2\b} {e^{-\f12\nu_0 (t-s)}}\|h_0\|_{L^\infty}\quad\mbox{for all}\  0\leq s\leq t\leq \rho,
\end{align}
where $C_3$ is the constant  given in Lemma \ref{lem3.02}. Moreover,  there exists a 
constant $C_\r\geq1$ such that 
\begin{align}\label{3.24}
\|{{G}^\varphi (t,s)}h_0\|_{L^\infty}\leq C_\r {e^{-\f14 \nu_0 (t-s)}} \|h_0\|_{L^\infty},
\end{align}
for all $t\geq s\geq 0$, {where the constant $C_\rho$ is independent of $s$.}
\end{lemma}

\begin{proof}
{We only prove the estimate for ${G}^\varphi (t,0)$ since the estimate for ${G}^\varphi (t,s)$ can be obtained by similar arguments.} First of all, recalling \eqref{1.9-1} and \eqref{def.iv}, and noting the assumption \eqref{3.22}, one has 
\begin{equation}\nonumber
I^\varphi(t,0)\leq e^{-\f12\nu_0 t}.
\end{equation}
As in Lemma \ref{lem3.02}, we denote {$h(t)={G}^\varphi (t,0)h_0$}. Then, for $t_1(t,x,v)\leq 0$, we have 
\begin{align*}
|\Fi_{\{t_1\leq 0\}} h(t,x,v)|=|I^\varphi(t,0) h_0(x-vt,v)|
\leq e^{-\f12\nu_0 t} \|h_0\|_{L^\infty}.
\end{align*}
For $t_1(t,x,v)> 0$, similar to obtaining \eqref{3.16}, it holds  that
\begin{align}\label{3.26}
|\Fi_{\{t_1>0\}}h(t,x,v)|\leq&\frac{I^\varphi(t,t_1)}{\widetilde{w}(v)}\int_{\Pi _{j=1}^{k}\mathcal{V}_{j}}\sum_{l=1}^{k}\mathbf{1}_{\{t_{l+1}\leq 0<t_l\}}|h_0(x_{l}-t_{l}v_{l}, v_{l})| \,\dd\Sigma^\varphi _{l}(0) \nonumber\\
&\qquad+\frac{I^\varphi(t,t_1)}{\widetilde{w}(v)}\int_{\Pi _{j=1}^{k}\mathcal{V}_{j}} \Fi_{\{t_{k+1}>0\}} |h(t_k,x_{k},v_{k})| \,\dd\Sigma^\varphi _{k}(t_k)\nonumber\\
&:=J_1+J_2,
\end{align}
where $J_1$, $J_2$ denote those respective terms on the right.
It follows from \eqref{3.22} that 
\begin{equation}\nonumber
I^\varphi(t,t_1)I^\varphi(t_1,t_2)\cdots I^\varphi(t_l,s)\leq e^{-\f12\nu_0(t-s)}.
\end{equation}
Then one can deduce that
\begin{align}
J_1&\leq C\r^{2\b}e^{-\f12\nu_0t} \|h_0\|_{L^\infty} \Pi _{j=1}^{k} \int_{\mathcal{V}_j} [1+\widetilde{w}(v_j)] \,\dd\s_j \nonumber\\
&\leq C\r^{2\b}e^{-\f12\nu_0t}  \left(1+\f{C_\b}{\r^3}\right)^{C_2\r^{\f54}}\|h_0\|_{L^\infty}\nonumber \\
&\leq C\r^{2\b}e^{-\f12\nu_0t} \|h_0\|_{L^\infty},\label{3.28}
\end{align}
and 
\begin{align}\label{3.29}
J_2&\leq\f{1}{\widetilde{w}(v)}\int_{\Pi _{j=1}^{k}\mathcal{V}_{j}} e^{-\f12\nu_0(t-t_k)} \sup_{y,u}|\Fi_{\{t_{1}(t_k,y,u)>0\}}h(t_k,y,u)| \Fi_{\{t_{k+1>0}\}} \widetilde{w}(v_k) \,\dd\s_k\cdots \dd\s_1\nonumber\\
&\leq C\r^{2\b} e^{-\f12\nu_0t}\sup_{0\leq s\leq t\leq \r}\left\|\Fi_{\{t_{1}>0\}}e^{\f12\nu_0s}h(s)\right\|_{L^\infty} \int_{\Pi _{j=1}^{k-1}\mathcal{V}_{j}} \Fi_{\{t_{k}>0\}} \,\dd\s_{k-1}\cdots \dd\s_1\nonumber\\
&\leq C_4\r^{2\b-3}\left(\f12\right)^{C_2\r^{\f54}}e^{-\f12\nu_0t}\sup_{0\leq s\leq t\leq \r}\left\|\Fi_{\{t_{1}>0\}}e^{\f12\nu_0s}h(s)\right\|_{L^\infty}.
\end{align}
Since $\r$ can be fixed large enough such that 
\begin{equation}\nonumber
C_4\r^{2\b-3}\left(\f12\right)^{C_2\r^{\f54}}\leq \f12,
\end{equation}
by substituting \eqref{3.28} and \eqref{3.29} into \eqref{3.26}, one obtains that
\begin{equation*}
\sup_{0\leq t\leq \r} \left\{e^{\f12\nu_0 t} \|\Fi_{\{t_1>0\}}h(t)\|_{L^\infty}\right\}\leq C_3\r^{2\b} \|h_0\|_{L^\infty}.
\end{equation*}
This then proves \eqref{3.23}. Finally, \eqref{3.24} follows from \eqref{3.23} with the help of similar arguments as in \cite[Lemma 5]{Guo2}; details are omitted for simplicity of presentation. Thus, the proof of Lemma \ref{lem3.3} is complete. 
\end{proof}

%

\subsection{Continuity 
}

As for continuity of $L^\infty$ solutions, one has

\begin{lemma}\label{lemA3}
Let $\Omega$ be strictly convex as in \eqref{convexity} and $T>0$ be given. Assume that $h_0(x,v)$ is continuous for $(x,v)\notin\gamma_0$, $\varphi(t,x,v)$ and $q(t,x,v)$ are continuous in 
$[0,T)\times\Omega\times\mathbb{R}^3 $ with 
\begin{equation}\nonumber
\sup_{0\leq t< T}\|\varphi(t,x,v)\|_{L^\infty_{x,v}}<\infty,\quad \sup_{0\leq t<T}\|q(t,x,v)\|_{L^\infty_{x,v}}<\infty, 
\end{equation}
and $\varphi$ satisfies  \eqref{3.3}. 
Let $h=h(t,x,v)$ satisfy
\begin{equation}\nonumber
\left\{\begin{aligned}
&\partial_th+v\cdot\nabla_x h+R(\varphi) h=q,\\  
& h(t,x,v)|_{t=0}=h_0(x,v),\\
& h(t,x,v)|_{\g_-}=\f{1}{\widetilde{w}(v)}\int_{\mathcal{V}(x)}h(t,x,v')\widetilde{w}(v')\,\dd\sigma,
\end{aligned}\right.
\end{equation}
over $[0,T)$, where the boundary condition also holds for $h_0(x,v)$, namely,
\begin{align}\label{3.41-1}
h_0(x,v)|_{\gamma_-}=\f{1}{\widetilde{w}(v)}\int_{\mathcal{V}(x)}h_0(x,v')\widetilde{w}(v')\,\dd\sigma.
\end{align}
Then $h(t,x,v)$ is continuous  at any $(t,x,v)\in[0,T)\times \Omega\times \R^3$ with $(x,v)\notin\gamma_0$.
\end{lemma}

\begin{proof}
Using the continuity of $\varphi(t,x,v)$  as well as 
Lemma \ref{lem3.02}, 
one can prove this lemma  by similar arguments as 
in \cite[Lemma 26]{Guo2}. The details are omited for simplicity of presentation. 
\end{proof}

\section{Local-in-time existence in $L^\infty_{x,v}$}\label{sec4}
This section is devoted to showing the following:

\begin{theorem}[Local Existence]\label{LE}
Let $\r>1$ be sufficiently large, $0\leq\varpi\leq \f1{64}$ and $\beta>5/2$ in the velocity weight function  \eqref{WF}. 
Supose $F_0=\mu(v)+\sqrt{\mu(v)}f_0(x,v)\geq0$ and $\|wf_0\|_{L^\infty}<\infty$. Then  exists a positive time $\hat{t}_0:=(\hat{C}_\r[1+\|h_0\|_{L^\infty}])^{-1}$ such that 
 the initial-boundary value problem \eqref{1.1}, \eqref{1.5-1}, \eqref{diffuse1} on the Boltzmann equation
has a unique solution $F(t,x,v)=\mu(v)+\sqrt{\mu(v)}f(t,x,v)\geq0$ on $t\in[0,\hat{t}_0]$, satisfying
\begin{equation*}
\sup_{0\leq t\leq \hat{t}_0}\|wf(t)\|_{L^\infty}\leq 2C_3 \r^{2\b}\|wf_0\|_{L^\infty},
\end{equation*}
where $\hat{C}_\r$ is a positive constant depending only on $\r$, and  $C_3$ is the positive constant defined in Lemma \ref{lem3.02}.   Moreover, if domain $\Omega$ is strictly convex, and if the initial data $f_0(x,v)$ are continuous except on $\g_0$ with 
\begin{equation}\label{3.31-1}
f_0(x,v)|_{\gamma_-}=c_{\mu }\sqrt{\mu (v)}\int_{v^{\prime }\cdot
	n(x)>0}f_0(x,v^{\prime })\sqrt{\mu (v^{\prime })}\{v'\cdot n(x)\}\,\dd v^{\prime },
\end{equation} 
then the solution $f(t,x,v)$ is continuous in $[0,\hat{t}_0]\times\{\Omega\times\mathbb{R}^3\backslash \g_0\}$.
\end{theorem}

\begin{proof}
We divide it by five steps. 

\medskip
\noindent{\bf Step 1.} To consider the local existence of solutions for the Boltzmann equation \eqref{1.1}, we start from 
the iteration that for $n=0,1,2,\cdots$, 
\begin{equation}\label{3.32}
\left\{\begin{aligned}
&\partial_tF^{n+1}+v\cdot\nabla_x F^{n+1}+F^{n+1}\cdot\int_{\mathbb{R}^3}\int_{\mathbb{S}^2}B(v-u,\omega)F^{n}(t,x,u)\,\dd\omega \dd u=Q_+(F^n,F^n),\\ 
&F^{n+1}(t,x,v)\Big|_{t=0}=F_0(x,v)\geq0,
\end{aligned}\right.
\end{equation}
with the diffuse reflection boundary condition
\begin{equation}\label{3.33}
F^{n+1}(t,x,v)|_{\g_-}=c_\m \mu(v)\int_{\mathcal{V}} F^{n+1}(t,x,v') \mu(v') \{v'\cdot n(x)\}\,\dd v',
\end{equation}
where we set $F^0(t,x,v)\equiv \mu(v)$ for $n=0$. As before, we 
denote 
\begin{equation}\nonumber
f^{n+1}(t,x,v)=\f{F^{n+1}(t,x,v)-\mu(v)}{\sqrt{\mu(v)}},
\end{equation}
and 
\begin{equation}
h^{n+1}(t,x,v)=w(v)f^{n+1}(t,x,v).\nonumber
\end{equation}
Then 
the above iteration
can be written equivalently as, for $n=0,1,2,\cdots$, 
\begin{equation}\label{3.34}
\left\{\begin{aligned}
&\partial_th^{n+1}+v\cdot\nabla_x h^{n+1}+h^{n+1}\cdot R(f^n)=K_wh^n+w(v)\Gamma_+(f^n,f^n),\\ 
&h^{n+1}(t,x,v)\Big|_{t=0}=h_0(x,v),
\end{aligned}\right.
\end{equation}
with 
\begin{align}\label{3.35}
h^{n+1}(t,x,v)|_{\g_-}=\f{1}{\widetilde{w}(v)}\int_{\mathcal{V}} h^{n+1}(t,x,v') \widetilde{w}(v') \,\dd\s,
\end{align}
and $h^0(t,x,v)\equiv 0$. Here we have denoted $K_wh:=\int_{\mathbb{R}^3}k_w(v,\eta) h(\eta)d\eta$ with $k_w(v,\eta)\equiv k(v,\eta)\f{w(v)}{w(\eta)}$.


\medskip
\noindent{\bf Step 2.} 
Next, we shall use the induction argument on $n=0,1,\cdots$ to  prove that there exists a positive time $\hat{t}_1>0$,  independent of $n$, such that \eqref{3.34} and \eqref{3.35}, or equivalently \eqref{3.32} and \eqref{3.33}, admits a unique mild solution on the time interval $[0,\hat{t}_1]$,  and  the following uniform bound and positivity hold true:
\begin{align}\label{3.40}
\|h^n(t)\|_{L^\infty}\leq 2C_3 \r^{2\b}\|h_0\|_{L^\infty}, \quad 
F^{n}(t,x,v)\geq0, 
\end{align}
for $0\leq t\leq \hat{t}_1$.


Firstly, for $n=0$,  recalling $f^0(t,x,v)\equiv 0$, it follows from Lemma \ref{lem3.02} that there exists a solution operator {${G}^1(t,0)$} to the IBVP \eqref{3.34}, \eqref{3.35} with $n=0$, such that 
\begin{align}\label{3.36}
\|h^1(t)\|_{L^\infty}=\|{{G}^1(t,0)}h_0\|_{L^\infty}\leq C_3\r^{2\b} \|h_0\|_{L^\infty}, 
\end{align}
for $0\leq t\leq \r$. 
For the positivity of $F^1(t,x,v)$,  we use the mild formulation 
\begin{align}\label{3.37}
F^1(t,x,v)&= \Fi_{\{t_1\leq 0\}} I^0(t,0) F_0(x-vt,v)\nonumber\\
&\quad+  \Fi_{\{t_1>0\}}c_\mu\mu(v) I^0(t,t_1)  \int_{\Pi_{j=1}^{k-1}\mathcal{V}_j} \sum_{l=1}^{k-1} \Fi_{\{t_{l+1}\leq0< t_l\}} F_0(x-vt,v)\,\dd\Sigma_{l}^1(0)\nonumber\\
&\quad+\Fi_{\{t_1>0\}}c_\mu\mu(v) I^0(t,t_1)\int_{\Pi_{j=1}^{k-1}\mathcal{V}_j} \Fi_{\{t_k>0\}}F^1(t_k,x_k,v_{k-1})\,\dd\Sigma_{k-1}^1(t_k).
\end{align}
Here and in the sequel we denote that for $n=0,1,\cdots$,
\begin{equation}\nonumber
I^n(t,s):=\exp\Big\{-\int_{s}^t(Rf^n)(\tau,X_{cl}(\tau),V_{cl}(\tau))\dd\tau\Big\},
\end{equation}
and 
\begin{equation}\nonumber
\dd\Sigma _{l}^n(\tau):=\{\Pi _{j=l+1}^{k-1}\dd\sigma _{j}\}\cdot  I^n(t_l,\tau)[v_l\cdot n(v_l)]dv_l\cdot\Pi_{j=1}^{l-1}\{I^n(t_{j},t_{j+1})\dd\s_j\}.
\end{equation}
It follows from \eqref{3.36} that 
\begin{equation*}
|F^1(t,x,v)|=\left|\mu(v)+\sqrt{\mu(v)}\f{h^1}{w}\right|\leq C_3\r^{2\b}(1+\|h_0\|_{L^\infty})\sqrt{\mu(v)},
\end{equation*}
which, together with \eqref{3.37} and \eqref{4.7}, yields that 
\begin{align}
F^1(t,x,v)&\geq -C\r^{2\b}\mu(v) (1+\|h_0\|_{L^\infty})  \int_{\Pi_{j=1}^{k-1}\mathcal{V}_j} \Fi_{\{t_k>0\}}\sqrt{\mu{(v_{k-1})}} |v_{k-1}| \,\dd v_{k-1} \dd\s_{k-2}\cdots \dd\s_1\nonumber\\
&\geq -C\r^{2\b}(1+\|h_0\|_{L^\infty})  \int_{\Pi_{j=1}^{k-2}\mathcal{V}_j} \Fi_{\{t_{k-1}>0\}} \,\dd\s_{k-2}\cdots \dd\s_1\nonumber\\
&\geq -C\r^{2\b} (1+\|h_0\|_{L^\infty})  \v,\nonumber
\end{align}
provided that $k\geq k_0(\v,\r)$ is true.  Since $\v>0$ is arbitrary, by letting $\v\to 0$, we deduce that $F^1(t,x,v)\geq0$ holds true for $0\leq t\leq \r$. Thus we have proved \eqref{3.40} for $n=0$.

Now, by supposing that \eqref{3.40} holds true for $n=0,1,2,\cdots m$, we consider it for $n=m+1$. Let {${G}^{m+1}(t,s)$} be the solution operator defined in Lemma \ref{lem3.02} with $\varphi(t,x,v)\equiv f^{m}(t,x,v)$. Then it holds that 
\begin{align}
h^{m+1}(t)={{G}^{m+1}(t,0)}h_0+\int_0^t {{G}^{m+1}(t,s)}\left[K_wh^m(s)+w\Gamma_+(f^m,f^m)(s)\right]\,\dd s.\nonumber
\end{align}
This form together with \eqref{3.40}, \eqref{3.9} and \eqref{2.4-1} give 
\begin{align}
\|h^{m+1}(t)\|_{L^\infty}
&\leq C_3\r^{2\b}\|h_0\|_{L^\infty}+C_3\r^{2\b}\int_0^t\|K_wh^m(s)\|_{L^\infty}+\|w\Gamma_+(f^m,f^m)(s)\|_{L^\infty}\,\dd s\nonumber\\
&\leq C_3\r^{2\b}\|h_0\|_{L^\infty}+CC_3\r^{2\b}\int_0^t\|h^m(s)\|_{L^\infty}+\|h^m(s)\|^2_{L^\infty}\,\dd s\nonumber\\
&\leq  C_3\r^{2\b}\|h_0\|_{L^\infty} 
\left[1+C_{\r,1} t\left(1+\|h_0\|_{L^\infty}\right)\right],\nonumber
\end{align}
where $C_{\r,1}\geq1$ is some positive constant depending only on $\r$.
Taking $\hat{t}_1:=(C_{\r,1}[1+\|h_0\|_{L^\infty}])^{-1}$, then it holds that 
\begin{align}
\|h^{m+1}(t)\|_{L^\infty}\leq 2 C_3\r^{2\b}\|h_0\|_{L^\infty},\nonumber
\end{align}
for $0\leq t\leq\hat{t}_1$.
Next, to show the positivity of $F^{m+1}$, we notice 
\begin{align}\label{3.44}
F^{m+1}(t,x,v)&= \Fi_{\{t_1\leq 0\}} \Big\{I^m(t,0) F_0(x-vt,v)+\int_0^tI^m(t,s) Q_+(F^m,F^m)(s,x-v(t-s),v)\,\dd s\Big\}\nonumber\\
&\quad+  \Fi_{\{t_1>0\}}c_\mu\mu(v) I^m(t,t_1)  \Big\{\int_{\Pi_{j=1}^{k-1}\mathcal{V}_j} \sum_{l=1}^{k-1} I_{\{t_{l+1}\leq0< t_l\}} F_0(x-vt,v)\,\dd\Sigma_{l}^m(0)\nonumber\\
&\quad +\int_{\Pi_{j=1}^{k-1}\mathcal{V}_j} \sum_{l=1}^{k-1} \int_0^{t_l}I_{\{t_{l+1}\leq0< t_l\}} Q_+(F^m,F^m)(\tau,X_{cl}(\tau),V_{cl}(\tau))\, \dd\Sigma_{l}^{m}(\tau) \dd\tau\nonumber\\
&\quad +\int_{\Pi_{j=1}^{k-1}\mathcal{V}_j} \sum_{l=1}^{k-1} \int_{t_{l+1}}^{t_l}I_{\{t_{l+1}\leq0< t_l\}} Q_+(F^m,F^m)(\tau,X_{cl}(\tau),V_{cl}(\tau)) \,\dd\Sigma_{l}^{m}(\tau) \dd\tau\nonumber\\
&\quad+\int_{\Pi_{j=1}^{k-1}\mathcal{V}_j} I_{\{t_k>0\}}F^m(t_k,x_k,v_{k-1})\,\dd\Sigma_{k-1}^m(t_k)\Big\}.
\end{align} 
Then, under the induction assumption \eqref{3.40}, it holds that 
\begin{equation}\nonumber
|F^m(t,x,v)|\leq 2C_3\r^{2\b}(1+\|h_0\|_{L^\infty})\sqrt{\mu(v)},
\end{equation}
then it follows from \eqref{3.44} that $F^{m+1}(t,x,v)$ has the lower bound as 
\begin{equation}\nonumber
-C\r^{2\b} \mu(v)\int_{\Pi_{j=1}^{k-2}\mathcal{V}_j} \Fi_{\{t_k>0\}}\, \dd\s_{k-2}\cdots \dd\s_1 \int_{\mathcal{V}_{k-1}} (1+\|h^m_0\|_{L^\infty} ) \sqrt{\mu(v_{k-1})} |v_{k-1}|\,\dd v_{k-1}.
\end{equation}
Using \eqref{4.7}, it further follows that
\begin{equation}\nonumber
F^{m+1}(t,x,v)\geq -C\r^{2\b}(1+\|h^m_0\|_{L^\infty} ) \v.
\end{equation}
Since $\v>0$ is arbitrary, we deduce $F^{m+1}(t,x,v)\geq0$ by letting $\v\to 0$.
Hence, by induction, we have proved \eqref{3.40} with $t\in[0, \hat{t}_1]$ for all $n=0,1,2,\cdots$.  Note that $\hat{t}_1>0$ is independent of $n$.

\medskip
\noindent{\bf Step 3.} 
To prove the convergence, we denote $\hat{w}(v)=\f{w(v)}{\sqrt{1+\r^2|v|^2}}$ and
\begin{align*}
\hat{h}^{n+1}(t,x,v):=\hat{w}(v)f^{n+1}(t,x,v)\equiv(1+\r^2|v|^2)^{-\f12} h^{n+1}(t,x,v),\quad n=0,1,\cdots.
\end{align*}
In the following, we shall prove that $\hat{h}^{n+1},n=0,1,\cdots$ is a Cauchy sequence. 

\medskip
\noindent{\it Case $t_1(t,x,v)\leq 0$}. In this case, 
it holds, for $n=0,1,2,\cdots$, that
\begin{equation}\label{3.46}
\Fi_{\{t_1\leq 0\}}\hat{h}^{n+1}(t,x,v)
=I^{n}(t,0) \hat{h}_0(x-vt,v)
+\int_0^t I^{n}(t,s) q^{n}(s,x-v(t-s),v)\,\dd s,
\end{equation}
where we have denoted 
\begin{equation}
q^n(t,x,v):=(K_{\hat{w}}\hat{h}^{n})(t,x,v)+(\hat{w}\Gamma_+(f^n,f^n))(t,x,v),\quad 
n=0,1,2,\cdots.\nonumber
\end{equation}
It then follows that
\begin{align}\label{3.47}
&\left|\Fi_{\{t_1\leq 0\}}(\hat{h}^{n+2}-\hat{h}^{n+1})(t,x,v)\right|\nonumber\\
&\leq |I^{n+1}(t,0)-I^n(t,0)| \cdot |\hat{h}_0(x-vt,v)|\nonumber\\
&\quad+\int_0^t |I^{n+1}(t,s)-I^n(t,s)| \cdot |q^{n+1}(s,x-v(t-s),v)|\,\dd s\nonumber\\
&\quad +\int_0^t I^n(t,s) \cdot |(q^{n+1}-q^n)(s,x-v(t-s),v)|\,\dd s\nonumber\\
&:=J_1^n+J_2^n+J_3^n.
\end{align}
We first estimate $J_1^n+J_2^n$ and then $J_3^n$. Since $t_{1}(t,x,v)\leq 0$, a direct calculation shows, for $0\leq s\leq t$, that 
\begin{align}\label{3.48}
|I^{n+1}(t,s)-I^n(t,s)| \leq C\nu(v)\int_s^t\|(f^{n+1}-f^n)(\tau)\|_{L^\infty}\,\dd\tau. 
\end{align}
Then, in terms of \eqref{3.40} and $t\leq \hat{t}_1$,  it follows from \eqref{3.48} that 
\begin{align}\label{3.49}
J_1^n+J_2^n&\leq C\nu(v)|\hat{h}_0(x-vt,v)|\cdot\int_0^t\|(f^{n+1}-f^n)(\tau)\|_{L^\infty}\,\dd\tau\nonumber\\
&\quad +C\int_0^t\|(f^{n+1}-f^n)(\tau)\|_{L^\infty}\dd\tau\cdot \int_0^t  \nu(v) |q^{n+1}(s,x-v(t-s),v)|\,\dd s\nonumber\\
&\leq Ct\Big[\|h_0\|_{L^\infty}+t\sup_{0\leq s\leq t}\left\{\|h^{n+1}(s)\|_{L^\infty}+\|h^{n+1}(s)\|^2_{L^\infty}\right\}\Big]\cdot \sup_{0\leq \tau\leq t}\|(f^{n+1}-f^n)(\tau)\|_{L^\infty}\nonumber\\
&\leq C_\r t[1+\|h_0\|_{L^\infty}]\cdot\sup_{0\leq \tau\leq t}\|(f^{n+1}-f^n)(\tau)\|_{L^\infty},
\end{align}
where we also have used 
\begin{align}
|\nu(v)q^{n+1}(s)|&\leq 
 C|K_{w}h^{n+1}|+|w(v)\Gamma_+(f^{n+1},f^{n+1})(s)|\nonumber\\
&\leq C\|h^{n+1}(s)\|_{L^\infty}+C\|h^{n+1}(s)\|^2_{L^\infty}.\nonumber
\end{align}
For $J_3$, we notice from \eqref{2.4-1}  that 
\begin{align}
|(q^{n+1}-q^n)(s)|&\leq \Big|K_{\hat{w}}(\hat{h}^{n+1}-\hat{h}^{n})(s)\Big|+\hat{w}(v)\Big|\Gamma_+(f^{n+1},f^{n+1})(s)-\Gamma_+(f^n,f^n)(s)\Big| \nonumber\\
&\leq C\left[1+\|\hat{h}^{n}(s)\|_{L^\infty}+\|\hat{h}^{n+1}(s)\|_{L^\infty}\right]\cdot \|(\hat{h}^{n+1}-\hat{h}^{n+1})(s)\|_{L^\infty},\nonumber
\end{align}
where $\b\geq 5/2 
$ is needed when using \eqref{2.4-1}.
Then, 
using \eqref{3.40} again,  we have
\begin{align}\label{3.52}
J_3^n\leq C_\r t[1+\|h_0\|_{L^\infty}]\cdot \sup_{0\leq s\leq t} \|(\hat{h}^{n+1}-\hat{h}^{n+1})(s)\|_{L^\infty}.
\end{align}
Thus, substituting  \eqref{3.52} and \eqref{3.49} into \eqref{3.47}, one obtains that 
\begin{equation}\label{3.53}
\left|\Fi_{\{t_1\leq 0\}}(\hat{h}^{n+2}-\hat{h}^{n+1})(t,x,v)\right|\leq Ct[1+\|h_0\|_{L^\infty}]\cdot \sup_{0\leq s\leq t} \|(\hat{h}^{n+1}-\hat{h}^{n+1})(s)\|_{L^\infty}.
\end{equation}

\noindent{\it Case $t_1(t,x,v)> 0$}.
In this case, 
we have, for $n=0,1,2,\cdots$, that 
\begin{align}\label{3.55}
\hat{h}^{n+1} \Fi_{\{t_1>0\}}
&=\int_{t_1}^{t} I^{n}(t,s) q^n(\tau,x-v(t-\tau),v)\,\dd\tau\nonumber\\
&\quad+\f{I^{n}(t,t_1)}{\widetilde{w}(v)\sqrt{1+\r^2|v|^2}} \int_{\Pi_{j=1}^{k-1}\mathcal{V}_j} \sum_{l=1}^{k-1} \Fi_{\{t_{l+1}\leq0< t_l\}} \hat{h}_0(x_l-v_lt_l,v_l)\,\dd\hat\Sigma_{l}^n(0)\nonumber\\
&\quad +\f{I^{n}(t,t_1)}{\widetilde{w}(v)\sqrt{1+\r^2|v|^2}}\int_{\Pi_{j=1}^{k-1}\mathcal{V}_j} \sum_{l=1}^{k-1} \int_0^{t_l}\Fi_{\{t_{l+1}\leq0< t_l\}} q^n(\tau,X_{cl}(\tau),V_{cl}(\tau)) \,\dd\hat\Sigma_{l}^{n}(\tau) \dd\tau\nonumber\\
&\quad +\f{I^{n}(t,t_1)}{\widetilde{w}(v)\sqrt{1+\r^2|v|^2}}\int_{\Pi_{j=1}^{k-1}\mathcal{V}_j} \sum_{l=1}^{k-1} \int_{t_{l+1}}^{t_l}\Fi_{\{t_{l+1}>0\}} q^n(\tau,X_{cl}(\tau),V_{cl}(\tau))\, \dd\hat\Sigma_{l}^{n}(\tau) \dd\tau\nonumber\\
&\quad+\f{I^{n}(t,t_1)}{\widetilde{w}(v)\sqrt{1+\r^2|v|^2}}\int_{\Pi_{j=1}^{k-1}\mathcal{V}_j} \Fi_{\{t_k>0\}}h^{n+1}(t_k,x_k,v_{k-1})\,\dd\hat\Sigma_{k-1}^n(t_k)\nonumber\\
&:=J_4^n+J_5^n+J_6^n+J_7^n+J_8^{n},
\end{align}
where  we have denoted
\begin{equation*}
d\hat\Sigma_{k-1}^n(\tau):=\{\Pi _{j=l+1}^{k-1}\dd\sigma _{j}\}\cdot \sqrt{1+\r^2|v|^2} \widetilde{w}(v_l) I^n(t_l,\tau)\dd\s_l\cdot\Pi_{j=1}^{l-1}\{I^n(t_{j},t_{j+1})\dd\s_j\}.
\end{equation*}
Then it follows from \eqref{3.55} that 
\begin{align}\label{3.56}
 \Fi_{\{t_1>0\}} \left|\hat{h}^{n+2}-\hat{h}^{n+1}\right |\leq \sum_{i=4}^8 |J_i^{n+1}-J_i^n|.
\end{align}
We estimate the right-hand difference terms as follows. First of all,  
using similar arguments in \eqref{3.49} and \eqref{3.52}, it holds that 
\begin{equation}\label{3.58}
|J_4^{n+1}-J_4^n|\leq C_\r t[1+\|h_0\|_{L^\infty}]\cdot \sup_{0\leq s\leq t} \|(\hat{h}^{n+1}-\hat{h}^{n+1})(s)\|_{L^\infty}.
\end{equation}
To consider other difference terms, 
we notice that 
\begin{align}\label{3.59-a1}
&\left|\int_{\Pi_{j=1}^{k-1}\mathcal{V}_j} \sum_{l=1}^{k-1} \Fi_{\{t_{l+1}\leq0< t_l\}}\,\dd\hat\Sigma_{l}^{n+1}(\tau)-\int_{\Pi_{j=1}^{k-1}\mathcal{V}_j} \sum_{l=1}^{k-1} \Fi_{\{t_{l+1}\leq0\leq t_l\}}\,\dd\hat\Sigma_{l}^n(\tau)\right| \nonumber\\
&\leq\int_{\Pi_{j=1}^{k-1}\mathcal{V}_j} \sum_{l=1}^{k-1} \Fi_{\{t_{l+1}\leq0< t_l\}}\sqrt{1+\r^2|v_l|^2}\widetilde{w}(v_l)\Big|
I^{n+1}(t_l,\tau)I^{n+1}(t_{l-1},t_{l})\cdots I^{n+1}(t_1,t_2)\nonumber\\
&\qquad\qquad\qquad\qquad\qquad\qquad\qquad\qquad-I^{n}(t_l,\tau)I^{n}(t_{l-1},t_{l})\cdots I^{n}(t_1,t_2)\Big|
\,\dd \s_{k-1}\cdots \dd\s_1.
\end{align}
Here, to estimate the above integral on the right, one can deduce that 
\begin{align}
&\Big|I^{n+1}(t_l,\tau)I^{n+1}(t_{l-1},t_{l})\cdots I^{n+1}(t_1,t_2)-I^{n}(t_l,\tau)I^{n}(t_{l-1},t_{l})\cdots I^{n}(t_1,t_2)\Big|\nonumber\\
&\leq |I^{n+1}(t_l,\tau)-I^{n}(t_l,\tau)+|I^{n+1}(t_{l-1},t_{l})-I^{n}(t_{l-1},t_{l})|
+\cdots+|I^{n+1}(t_1,t_2)-I^{n}(t_1,t_2)|\nonumber\\
&\leq C\nu(v_l)\int_{\tau}^{t_l}\|(f^{n+1}-f^n)(s)\|_{L^\infty}\, \dd s+C\nu(v_{l-1})\int_{t_l}^{t_{l-1}}\|(f^{n+1}-f^n)(\tau)\|_{L^\infty}\,\dd\tau\nonumber\\
&\quad+\cdots+C\nu(v_1)\int_{t_2}^{t_1}\|(f^{n+1}-f^n)(\tau)\|_{L^\infty}\,\dd\tau \nonumber\\
&\leq C\left[\nu(v_1)+\cdots+\nu(v_l)\right]\int_0^t\|(f^{n+1}-f^n)(s)\|_{L^\infty}\,\dd s.\label{3.59-a2}
\end{align}
Notice that 
\begin{align}
& \int_{\Pi_{j=1}^{k-1}\mathcal{V}_j} \sum_{l=1}^{k-1} \Fi_{\{t_{l+1}\leq0< t_l\}} \sqrt{1+\r^2|v_l|^2}\widetilde{w}(v_l)[\nu(v_l)+\cdots+\nu(v_1)] \,\Pi_{j=1}^{k-1} \dd \si_j\notag\\
&\leq  \sum_{l=1}^{k-1} \int_{\Pi_{j=1}^{k-1}\mathcal{V}_j}  \sqrt{1+\r^2|v_l|^2}\widetilde{w}(v_l)[\nu(v_l)+\cdots+\nu(v_1)] \,\Pi_{j=1}^{k-1} \dd \si_j\notag\\
&=  \sum_{l=1}^{k-1} \int_{\Pi_{j=1}^{l}\mathcal{V}_j}  \sqrt{1+\r^2|v_l|^2}\widetilde{w}(v_l)[\nu(v_l)+\cdots+\nu(v_1)] \,\Pi_{j=1}^{l} \dd \si_j \notag\\
&\leq  \sum_{l=1}^{k-1} k\times \frac{C}{\rho^3}\leq \frac{Ck^2}{\rho^3},\notag
\end{align}
where we have used the inequalities that $\nu(v_j)\leq C (1+|v_j|)$ $(j=1,2,\cdots,l)$, and
\begin{equation}\label{3.59-a3}
\left\{\begin{split}
&\int_{\mathcal{V}_l}(1+|v_l|)\sqrt{1+\r^2|v_l|^2}\,\widetilde{w}(v_l)\,\dd\si_l\leq \f{C}{\r^3},\\
&\int_{\mathcal{V}_j} (1+|v_j|)\,\dd \si_j\leq C,\quad j=1,2,\cdots,l-1, 
\end{split}\right.
\end{equation}
for a generic constant $C>0$.
Then \eqref{3.59-a1} is further bounded by
\begin{align}
&\left|\int_{\Pi_{j=1}^{k-1}\mathcal{V}_j} \sum_{l=1}^{k-1} \Fi_{\{t_{l+1}\leq0< t_l\}}\,\dd\hat\Sigma_{l}^{n+1}(\tau)-\int_{\Pi_{j=1}^{k-1}\mathcal{V}_j} \sum_{l=1}^{k-1} \Fi_{\{t_{l+1}\leq0\leq t_l\}}\,\dd\hat\Sigma_{l}^n(\tau)\right|\nonumber\\
&\leq C\frac{k^2}{\rho^3}\int_{0}^{t}\|(f^{n+1}-f^n)(\tau)\|_{L^\infty}\dd\tau\notag\\
&=\frac{C}{\rho^{\frac{1}{2}}}\int_{0}^{t}\|(f^{n+1}-f^n)(\tau)\|_{L^\infty}\dd\tau,
\label{3.59}
\end{align}
where we have taken $k=C_1\r^{\f54}$. In a similar way, by applying \eqref{3.59-a2} and \eqref{3.59-a3} and then taking $k=C_1\r^{\f54}$, one has  
\begin{align}\label{3.60}
&\left|\int_{\Pi_{j=1}^{k-1}\mathcal{V}_j} \sum_{l=1}^{k-1} \Fi_{\{t_{l+1}>0\}}\,\dd\hat\Sigma_{l}^{n+1}(\tau)-\int_{\Pi_{j=1}^{k-1}\mathcal{V}_j} \sum_{l=1}^{k-1} \Fi_{\{t_{l+1}>0\}}\,\dd\hat\Sigma_{l}^n(\tau)\right|\nonumber\\
&\leq \int_{\Pi_{j=1}^{k-1}\mathcal{V}_j} \sum_{l=1}^{k-1} \Fi_{\{t_{l+1}>0\}}\sqrt{1+\r^2|v_l|^2}\widetilde{w}(v_l)\Big|I^{n+1}(t_l,\tau)I^{n+1}(t_{l-1},t_{l})\cdots I^{n+1}(t_1,t_2)\nonumber\\
&\qquad\qquad\qquad\qquad\qquad-I^{n}(t_l,\tau)I^{n}(t_{l-1},t_{l})\cdots I^{n}(t_1,t_2)\Big|\,\dd\s_{k-1}\cdots \dd\s_1\nonumber\\
&\leq \frac{C}{\r^{\f12}}\int_{0}^{t}\|(f^{n+1}-f^n)(\tau)\|_{L^\infty}\,\dd\tau.
\end{align}
Further note that
\begin{equation}\label{3.63}
\Big[ \sqrt{1+\r^2|v|^2}\widetilde{w}(v)\Big]^{-1}\leq C\f{\r^{2\b-1}}{1+|v|}.
\end{equation}
Thus, to treat \eqref{3.56}, with the estimates  \eqref{3.59} and \eqref{3.60} as well as \eqref{3.63}, one has
\begin{align}\label{3.64}
&\left|J^{n+1}_5-J^n_5\right|\nonumber\\
&\leq C\f{\r^{2\b}}{1+|v|}\left|I^{n+1}(t,t_1)-I^n(t,t_1)\right|\int_{\Pi_{j=1}^{k-1}\mathcal{V}_j} \sum_{l=1}^{k-1} \Fi_{\{t_{l+1}\leq0< t_l\}} |\hat{h}_0(x_l-v_lt_l,v_l)|\,\dd\hat\Sigma_{l}^{n+1}(0)\nonumber\\
&\quad +C\r^{2\b}\|h_0\|_{L^\infty} \left|\int_{\Pi_{j=1}^{k-1}\mathcal{V}_j} \sum_{l=1}^{k-1} \Fi_{\{t_{l+1}\leq0< t_l\}}\left\{\dd\hat\Sigma_{l}^{n+1}(0)-\dd\hat\Sigma_{l}^n(0)\right\}\right|\nonumber\\
&\leq C_\r t\|h_0\|_{L^\infty}\cdot\sup_{0\leq\tau\leq t}\|(\hat{h}^{n+1}-\hat{h}^{n})(\tau)\|_{L^\infty},
\end{align}
and 
\begin{align}\label{3.65}
|J^{n+1}_6-J^n_6|+|J^{n+1}_7-J^n_7|\leq C_\r t(1+\|h_0\|_{L^\infty})\cdot\sup_{0\leq\tau\leq t}\|(\hat{h}^{n+1}-\hat{h}^{n})(\tau)\|_{L^\infty}.
\end{align}
For the last difference term on the right-hand of \eqref{3.56}, 
it follows from \eqref{3.63} and \eqref{4.8} that 
\begin{align}
\left|J_8^{n+1}-J^{n}_8\right|&\leq C\r^{2\b-1}\int_{\Pi_{j=1}^{k-1}\mathcal{V}_j} \Fi_{\{t_k>0\}}|(\hat{h}^{n+2}-\hat{h}^{n+1})(t_k,x_k,v_{k-1})|\,\dd\hat\Sigma_{k-1}^n(t_k)\nonumber\\
&\quad+C\r^{2\b-1}\f{|I^{n+1}(t,t_1)-I^n(t,t_1)|}{1+|v|}\sup_{0\leq\tau\leq t}\|\hat{h}^{n+2}(\tau)\|_{L^\infty}\int_{\Pi_{j=1}^{k-1}\mathcal{V}_j} \Fi_{\{t_k>0\}}\,\dd\hat\Sigma_{k-1}^{n+1}(t_k)\nonumber\\
&\quad+C\r^{2\b-1}\sup_{0\leq\tau\leq t}\|\hat{h}^{n+1}(\tau)\|_{L^\infty}\left|\int_{\Pi_{j=1}^{k-1}\mathcal{V}_j} \Fi_{\{t_k>0\}}\left\{\dd\hat\Sigma_{k-1}^{n+1}(t_k)-\dd\hat\Sigma_{k-1}^{n}(t_k)\right\}\right|,\nonumber 
\end{align}
and hence
\begin{equation}
\label{3.66}
\left|J_8^{n+1}-J^{n}_8\right|\leq C\r^{2\b}\left(\f12\right)^{C_2\r^{\f54}} \sup_{0\leq\tau\leq t}\|(\hat{h}^{n+2}-\hat{h}^{n+1})(\tau)\|_{L^\infty} 
+C_{\r}t \|h_0\|_{L^\infty}\sup_{0\leq\tau\leq t}\|(\hat{h}^{n+1}-\hat{h}^{n})(\tau)\|_{L^\infty}.
\end{equation}
Now, by substituting \eqref{3.58}, \eqref{3.64}, \eqref{3.65} and \eqref{3.66} into \eqref{3.56}, and then combining \eqref{3.53}, one obtains that 
\begin{multline}\label{3.67}
\|(\hat{h}^{n+1}-\hat{h}^{n})(t)\|_{L^\infty}
\leq C_5\r^{2\b}\left(\f12\right)^{C_2\r^{\f54}} \sup_{0\leq\tau\leq t}\|(\hat{h}^{n+2}-\hat{h}^{n+1})(\tau)\|_{L^\infty}\\
+C_{\r}t \left(1+\|h_0\|_{L^\infty}\right)\sup_{0\leq\tau\leq t}\|(\hat{h}^{n+1}-\hat{h}^{n})(\tau)\|_{L^\infty},
\end{multline}
where $C_5\geq1$ is a positive constant.
As before, we let $\r>0$ be fixed suitably large such that 
\begin{equation}
C_5\r^{2\b}\left(\f12\right)^{C_2\r^{\f54}} \leq \f12.\nonumber
\end{equation}
Then it follows from \eqref{3.67} that
\begin{align}\label{3.68}
\sup_{0\leq s\leq t}\|(\hat{h}^{n+2}-\hat{h}^{n+1})(s)\|_{L^\infty}
\leq C_{\r,2}t[1+\|h_0\|_{L^\infty}]\cdot\sup_{0\leq\tau\leq t}\|(\hat{h}^{n+1}-\hat{h}^{n})(\tau)\|_{L^\infty},
\end{align}
where $C_{\r,2}$ is a positive constant depending only on $\r$.  We denote 
$\hat{C}_\r=\max\left\{C_{\r,1}, 2C_{\r,2}\right\}$,
and choose 
\begin{equation*}
\hat{t}_0:= \frac{1}{\hat{C}_{\r}\left(1+\|h_0\|_{L^\infty}\right)}(\leq \hat{t}_1).
\end{equation*} 
Then it holds that 
$\hat{C}_{\r,2}\hat{t}_0[1+\|h_0\|_{L^\infty}]\leq \f12$. 
Thus it follows from \eqref{3.68} that 
\begin{equation}\nonumber
\sup_{0\leq t\leq \hat{t}_0}\|(\hat{h}^{n+2}-\hat{h}^{n+1})(t)\|_{L^\infty}
\leq \f12 \sup_{0\leq t\leq \hat{t}_0}\|(\hat{h}^{n+1}-\hat{h}^{n})(t)\|_{L^\infty}.
\end{equation}
Further by induction in $n$, it is direct to obtain 
\begin{equation}\nonumber
\sup_{0\leq t\leq \hat{t}_0}\|(\hat{h}^{n+2}-\hat{h}^{n+1})(t)\|_{L^\infty}\leq \left(\f12\right)^n\sup_{0\leq s\leq \hat{t}_0}\|\hat{h}^1(s)\|_{L^\infty}\leq 2C_3\r^{2\b}\|h_0\|_{L^\infty}\left(\f12\right)^n.
\end{equation}
This implies that $f^{n+1}(t,x,v), n=0,1,2,\cdots$, is a Cauchy sequence. Therefore, there exists a limit function $f(t,x,v)$ such that 
\begin{equation}\label{3.71}
\sup_{0\leq t\leq\hat{t}_0}\left\|\f{w(v)}{\sqrt{1+\r^2|v|^2}}(f^{n+1}-f)(t)\right\|_{L^\infty}\rightarrow 0,\quad\mbox{as}\quad n\rightarrow\infty.
\end{equation}
The limit function $f(t,x,v)$ is indeed a solution to the Boltzmann equation \eqref{1.9} with the diffuse reflection boundary conditon \eqref{diffuse}.

\medskip
\noindent{\bf Step 4.}
We now consider the uniqueness of solutions. Let $\tilde{f}(t,x,v)$ be another solution to the Boltzmann equation \eqref{1.9} and \eqref{diffuse} with the same initial data as for $f(t,x,v)$ and the bound 
\begin{equation}\nonumber
\sup_{0\leq t\leq \hat{t}_0}\|w\tilde{f}(t)\|_{L^\infty}<\infty.
\end{equation}
By similar arguments as in \eqref{3.46}-\eqref{3.71}, one has that 
\begin{equation}
\left\|\f{w(v)}{\sqrt{1+\r^2|v|^2}}(f-\tilde{f})(t)\right\|_{L^\infty} 
\leq C\{1+\sup_{0\leq t\leq \hat{t}_0}[\|w\tilde{f}(t)\|_{L^\infty}+\|wf(t)\|_{L^\infty}]\}\int_0^t\left\|\f{w(v)}{\sqrt{1+\r^2|v|^2}}(f-\tilde{f})(\tau)\right\|_{L^\infty}\dd\tau.\nonumber
\end{equation}
Then the uniqueness follows immedately from the Gronwall inequality, i.e., $f(t,x,v)\equiv\tilde{f}(t,x,v)$.

\medskip
\noindent{\bf Step 5.}
Finally, we consider the continuity of $f(t,x,v)$ if the domain $\Omega$ is strictly convex. It is noted that \eqref{3.31-1}  is equivalent to \eqref{3.41-1}. Hence,  it follows from Lemma \ref{lemA3} that  $h^{1}(t,x,v)$ is continuous in $[0,\hat{t}_0]\times\{\Omega\times\mathbb{R}^3\backslash \g_0\}$   since $h^0\equiv0$. By induction arguments and also Lemma \ref{lemA3}, it is straightforward to verify that $h^{n+1}(t,x,v)$ is continuous in $[0,\hat{t}_0]\times\{\Omega\times\mathbb{R}^3\backslash \g_0\}$. Then it follows from \eqref{3.71}  that $f(t,x,v)$ is continuous in $[0,\hat{t}_0]\times\{\Omega\times\mathbb{R}^3\backslash \g_0\}$. 
Therefore, the proof of 
Theorem \ref{LE} is completed. 
\end{proof}

\section{A priori estimates}\label{sec5}

With the local existence of unique solutions to the IBVP \eqref{1.9} and \eqref{diffuse}
with arbitrarily large initial data established in Theorem \ref{LE},  in order to further prove the global-in-time existence of solutions, it suffices to get uniform estimates on solutions.  In this section, we devoted ourselves to do that.

From now on, we fix $\r>0$ and $\b\geq\f52$ such that all the results in Sections 2, 3 and 4 hold and also the following inequality
\begin{equation}\label{5.00}
\left(C_3\r^{2\b}\right)^{\f{1}{\r}}\leq e^{\f12\nu_0},
\end{equation}
is satisfied which will be used in Lemma \ref{lem5.4} below. 
Recall that initial data satisfy $\|wf_0\|\leq M_0$ for a given constant $M_0\geq 1$ which could be large. Let $f(t,x,v)$ with
\begin{equation}
\label{ap.po}
F(t,x,v)=\mu(v)+\sqrt{\mu(v)} f(t,x,v)\geq 0
\end{equation}
 be the solution to  the IBVP \eqref{1.9} and \eqref{diffuse} with initial data $f_0(x,v)$ over the time interval $[0,T)$ for $0<T\leq \infty$. 
 Throughout this section, we make the {\it a priori} {assumption}:
\begin{equation}\label{5.0}
\sup_{0\leq t< T}\|h(t)\|_{L^\infty}\leq \bar{M},
\end{equation}
for $h(t,x,v):=w(v)f(t,x,v)$, where $\bar{M}\geq1$ is a large positive constant depending only on $M_0$, not on the solution and $T$, and  it will be  determined in the end of the proof. We also remark that the constants in this section  may depend on $\r$ and $\beta$, but we shall not point it out explicitly when without any confusion. 

\subsection{Estimate in $L^2_{x,v}$}

Recall that $f(t,x,v)$ satisfies
\begin{align}\label{1.9-10}
\partial_tf+v\cdot\nabla_xf+Lf=\Gamma(f,f),
\end{align}
with
\begin{align}
Lf=\nu(v) f-Kf=-\f1{\sqrt{\mu}}\Big\{Q(\mu,\sqrt{\mu}f)+Q(\sqrt{\mu}f,\mu)\Big\},\nonumber
\end{align}
and 
\begin{align}
\Gamma(f,f)=\f1{\sqrt{\mu}}Q(\sqrt{\mu}f,\sqrt{\mu}f)
&
=\Gamma_+(f,f)-\Gamma_-(f,f).\nonumber
\end{align}
For later use we also introduce some notations. For a function $g(x,v)$ defined on the phase boundary $\ga=\pa\Omega\times \R^3$, 
we denote the boundary integration 
\begin{align}\nonumber
\int_{\gamma_{\pm}}g(x,v)\,\dd\gamma=\int_{\gamma_{\pm}}g(x,v)|v\cdot n(x)|\,\dd S_x \dd v,
\end{align}
where $\dd S_x$ is the standard surface measure on $\partial\Omega$. We also denote 
$$
\|g\|_{L^2_\g}=\|g\|_{L^2_{\g_+}}+\|g\|_{L^2_{\g_-}},
$$
to be the norm of $g$ in the space $L^2(\g)$ with respect to the measure $|v\cdot n(x)|\,\dd S_x\dd v$. For any fixed $x\in\partial\Omega$, we denote the boundary inner product over $\gamma_{\pm}$ in $v$ as 
\begin{align}\nonumber
\langle g_1,g_2\rangle_{\g_{\pm}}(x)=\int_{\pm v\cdot n(x)>0} g_1(x,v) g_2(x,v) |v\cdot n(x)|\,\dd v.
\end{align}
By \eqref{diffusenormal}, we define the projection operator $P_\gamma$ by 
\begin{align}\nonumber
P_{\g}g=c_\mu \sqrt{\mu(v)}\int_{n(x)\cdot v'>0}g(t,x,v') \sqrt{\mu(v')}\{n(x)\cdot v' \}\,\dd v'.
\end{align}

The goal of this section is to obtain the $L^2$-estimate in terms of the direct energy method. Note that as the $L^\infty$-norm of solutions is not small, the $L^2$-norm of solutions may increase exponentially in time.   

\begin{lemma}\label{lem.l2}
Under the a priori assumption \eqref{5.0}, it holds that 
\begin{align}\label{5.5}
\|f(t)\|_{L^2}\leq e^{\tilde{C}_1\bar{M}t}\|f_0\|_{L^2},
\end{align}
for all $0\leq t<T$, 
where $\tilde{C}_1\geq 1$ is a generic constant.
\end{lemma}

\begin{proof}
By multiplying the Boltzmann equation \eqref{1.9-10} by $f(t,x,v)$ and integrating it over $\Omega\times\mathbb{R}^3$, one obtains that 
\begin{align}\label{A8}
&\frac{1}{2}\f{d}{dt}\|f(t)\|_{L^2}^2+\int_{\Omega\times\mathbb{R}^3}v\cdot\nabla_x(\frac{1}{2}|f(t)|^2)\,\dd x\dd v
+\int_{\Omega\times\mathbb{R}^3}f(t)Lf(t)\,\dd x\dd v\nonumber\\
&=\int_{\Omega\times\mathbb{R}^3}\Gamma(f,f)(t)f(t)\,\dd x\dd v.
\end{align}
For the second term on the left, it follows from integration by part that 
\begin{align}\label{A9}
\int_{\Omega\times\mathbb{R}^3}v\cdot\nabla_x(\frac{1}{2}|f(t)|^2)\,\dd x\dd v
&=\frac{1}{2}\int_{\gamma_+}|f(t)|^2\,\dd\gamma-\frac{1}{2}\int_{\gamma_-}|f(t)|^2\,\dd\gamma.
\end{align}
Note from the diffuse reflection boundary condition \eqref{diffuse} that 
\begin{align}\nonumber
\int_{\gamma_-}|f(t)|^2\,\dd\gamma=\int_{\gamma_+}|P_\g f(t)|^2\,\dd\gamma.
\end{align}
Plugging it back to  \eqref{A9} yields  that 
\begin{align}\label{A10}
\int_{\Omega\times\mathbb{R}^3}v\cdot\nabla_x(\frac{1}{2}|f(t)|^2)\,\dd x\dd v
&=\frac{1}{2}\int_{\gamma_+}|f(t)|^2\,\dd\gamma-\frac{1}{2}\int_{\gamma_+}|P_\g f(t)|^2\,\dd\gamma\nonumber\\
&=\frac{1}{2}\int_{\gamma_+}|(I-P_\g) f(t)|^2\,\dd\gamma\geq 0.
\end{align}
For the third term on the left-hand side of \eqref{A8}, it holds that 
\begin{equation}\label{A11}
\int_{\Omega\times\mathbb{R}^3}f(t)Lf(t)\,\dd x\dd v\geq0.
\end{equation}
Finally, 
to estimate the right-hand nonlinear term of \eqref{A8}, we note that 
\begin{align}
\int_{\mathbb{R}^3} \Gamma(f,f)(t)f(t)\,\dd v&\leq C\left(\int_{\mathbb{R}^3}|\nu(v)f(t)|^2\,\dd v\right)^{\f12}\|f(t)\|^2_{L^2_v}\nonumber\\
&\leq C \|wf(t)\|_{L^\infty_v}\left(\int_{\mathbb{R}^3}(1+|v|)^{-4\b+2}\,\dd v\right)^{\f12} \|f(t)\|^2_{L^2_v}\nonumber\\
&\leq C\|h(t)\|_{L^\infty_v}\|f(t)\|^2_{L^2_v},\nonumber
\end{align}
where  $\b\geq 5/2$ has been used. Hence, it holds that
\begin{equation}
\label{A12}
\int_{\mathbb{R}^3} \Gamma(f,f)(t)f(t)\,\dd x\dd v \leq C\|h(t)\|_{L^\infty_{x,v}}\|f(t)\|^2_{L^2_{x,v}}.
\end{equation}
Substituting \eqref{A12}, \eqref{A11} and \eqref{A10} back into \eqref{A8}, one obtains that 
\begin{align}\nonumber
\f{d}{dt}\|f(t)\|_{L^2}^2\leq \tilde{C}_1\|h(t)\|_{L^\infty}\|f(t)\|^2_{L^2},
\end{align}
where $\tilde{C}_1>0$ is a generic positive constant. Then it follows from the {\it a priori} assumption \eqref{5.0} and Gronwall inequality that 
\begin{align}\nonumber
\|f(t)\|_{L^2}\leq \|f_0\|_{L^2}e^{\tilde{C}_1\bar{M}t}.
\end{align}
Therefore the proof of this lemma is completed.  
\end{proof} 

\subsection{Estimate in $L^\infty_xL^1_v$}

By \eqref{def.iv} and \eqref{ap.po}, it is direct to see that 
\begin{equation*}
I^f(t,s)=\exp\left\{-\int_{s}^tR(f)\big(\tau,X_{cl}(\tau),V_{cl}(\tau)\big)\,\dd\tau\right\}\leq 1,
\end{equation*}
as $R(f)\geq 0$. 
In order to extend the local-in-time solutions to all the positive time, it is necessary to further obtain the time-decay property of $I^f(t,s)$, namely,
\begin{align*}
I^f(t,s)\leq Ce^{-\f{1}{C}(t-s)},
\end{align*}
for some generic large constant $C$. For solutions of large amplitude, the above time-decay property seems impossible to be obtained, because the vacuum could not be excluded.
However, if the $L^2$-norm of initial data is small,  such time-decay property may still hold in some sense even if the  solution is allowed to have large oscillations.  Indeed, we have the following useful lemma.

\begin{lemma}\label{lem5.3}
Under the a priori assumption \eqref{5.0},  there exists a generic constant $\tilde{C}_2\geq1 $  such that given any $T_0>\tilde{t}$ with
\begin{equation}\label{5.12}
\tilde{t}:=\f{2}{\nu_0}\ln\left(\tilde{C}_2 M_0\right)>0,
\end{equation}
there is a generally small positive constant $\epsilon_1=\epsilon_1(\bar{M},T_0)>0$, depending only on $\bar{M}$ and $T_0$, 
such that if $\|f_0\|_{L^2}\leq \epsilon_1$, then one has
\begin{align}\label{5.13}
R(f)(t,x,v)\geq \f12\nu(v),
\end{align}
for all $(t,x,v)\in [\tilde{t},T_0)\times \Omega\times \R^3$. 
Here $\epsilon_1$ is decreasing in $\bar{M}$ and $T_0$.
\end{lemma}

\begin{proof}
Recall that
\begin{equation*}
R(f)(t,x,v)=\int_{\mathbb{R}^3}\int_{\mathbb{S}^2} B(v-u,\omega)[\mu(u)+\sqrt{\mu(u)}f(t,x,u)]\,\dd\omega \dd u.
\end{equation*}
Then $R(f)(t,x,v)$ has the lower bound as 
\begin{align}
R(f)(t,x,v)
\geq \nu(v)\left[1-C_6\int_{\mathbb{R}^3} e^{-\f{|u|^2}8}|f(t,x,u)|\,\dd u\right],\nonumber
\end{align}
for a generic constat $C_6\geq1 $ which is  
independent of $\r$.
Therefore, to show \eqref{5.13}, it suffices to prove that
\begin{equation}\label{5.13.ap1}
\int_{\mathbb{R}^3} e^{-\f{|v|^2}8}|h(t,x,v)|\,\dd v\leq \f1{2C_6},
\end{equation}
for all $t\geq \tilde{t}$ and $x\in \Omega$, where $\tilde{t}$  is a positive constant to be suitably chosen.
Note that $f(t,x,v)$ satisfies \eqref{1.9-10}. By using the operator $G(t)$ defined in Proposition \ref{lem4.2}, we can represent the solution 
to \eqref{1.9-10} in terms of $h(t,x,v)=w(v)f(t,x,v)$ as
\begin{align}
h(t,x,v)=G(t)h_0+\int_0^t G(t-s)\Big[K_wh(s)+w(v)\Gamma(f,f)(s)\Big]\,\dd s.\nonumber
\end{align}
This expression  immediately yields  that 
\begin{align}
&\int_{\mathbb{R}^3} e^{-\f{|v|^2}8}|h(t,x,v)|\,\dd v\nonumber\\
&\leq \int_{\mathbb{R}^3} e^{-\f{|v|^2}8}|(G(t)h_0)(t,x,v)|\,\dd v\nonumber\\
&\quad+\left\{\int_{t-\l}^t+\int_0^{t-\l}\right\}\int_{\mathbb{R}^3}e^{-\f{|v|^2}8} |\Big(G(t-s)K_wh(s)\Big)(t,x,v)|\,\dd v\dd s\nonumber\\
&\quad+\left\{\int_{t-\l}^t+\int_0^{t-\l}\right\}\int_{\mathbb{R}^3}e^{-\f{|v|^2}8} \Big|\Big(G(t-s)[w\Gamma(f,f)(s)]\Big)(t,x,v)\Big|\,\dd v\dd s,\label{l53.p1}
\end{align}
where $\l>0$ is a small positive constant to be chosen later.

We estimate the right-hand three terms of \eqref{l53.p1} as follows. For the first term, it is direct to deduce from \eqref{4.17} that 
\begin{align*}
\int_{\mathbb{R}^3} e^{-\f{|v|^2}8}|(G(t)h_0)(t,x,v)|\,\dd v\leq C\|G(t)h_0\|_{L^\infty}\leq C e^{-\f{\nu_0}2t}\|h_0\|_{L^\infty}.
\end{align*}
For other two terms, we note from \eqref{4.17} that 
\begin{multline}\label{5.19}
\Big|\Big(G(t-s)K_wh(s)\Big)(t,x,v)\Big|+\Big|\Big(G(t-s)[w\Gamma(f,f)(s)]\Big)(t,x,v)\Big|\\
\leq Ce^{-\f{\nu_0}2(t-s)}\Big[\|h(s)\|_{L^\infty}+\|h(s)\|_{L^\infty}^2\Big].
\end{multline}
Hence, for integrals over $[t-\lambda,t]$, they are bounded as 
\begin{multline*}
\int_{t-\l}^t\int_{\mathbb{R}^3}e^{-\f{|v|^2}8} \left\{\Big|\Big(G(t-s)K_wh(s)\Big)(t,x,v)\Big|+\Big|\Big(G(t-s)[w\Gamma(f,f)(s)]\Big)(t,x,v)\Big| \right\}\,\dd v\dd s\\
\leq C\l \sup_{0\leq s\leq t}\Big[\|h(s)\|_{L^\infty}+\|h(s)\|_{L^\infty}^2\Big].
\end{multline*}
Over $[0,t-\lambda]$ and for large velocities $|v|\geq N$, it follows from \eqref{5.19} that
\begin{multline*}
\int_0^{t-\l}\int_{|v|\geq N}e^{-\f{|v|^2}8} \left\{\Big|\Big(G(t-s)K_wh(s)\Big)(t,x,v)\Big|+\Big|\Big(G(t-s)[w\Gamma(f,f)(s)]\Big)(t,x,v)\Big| \right\}\,\dd v\dd s\\
\leq \f{C}{N}\sup_{0\leq s\leq t}\Big[\|h(s)\|_{L^\infty}+\|h(s)\|_{L^\infty}^2\Big].
\end{multline*}

It remains to estimate integrals over $\{0\leq s\leq t-\la, |v|\leq N\}$ for the last two terms on the right-hand side of \eqref{l53.p1}.  Firstly, it follows from Proposition \ref{lem3.1} that 
\begin{align}\label{5.21}
&\Big(G(t-s)K_wh(s)\Big)(t,x,v)\nonumber\\
&=\Fi_{\{t_1\leq s\}} e^{-\nu(v)(t-s)} (K_wh)(s,x-v(t-s),v)\nonumber\\
&\quad+\f{e^{-\nu(v)(t-t_1)}}{\widetilde{w}(v)}
\int_{\Pi_{j=1}^{k-1}\mathcal{V}_j} \sum_{l=1}^{k-1} \Fi_{\{t_{l+1}\leq s< t_l\}} (K_wh)(s,x_l-v_l(t_l-s),v_l)\,\dd\Sigma_{l}(s)\nonumber\\
&\quad+\f{e^{-\nu(v)(t-t_1)}}{\widetilde{w}(v)}
\int_{\Pi_{j=1}^{k-1}\mathcal{V}_j} \Fi_{\{t_k>s\}} \Big(G(t-s)K_wh(s)\Big)(t_k,x_k,v_{k-1})\,\dd\Sigma_{k-1}(t_k)\nonumber\\
&:=J_{11}+J_{12}+J_{13},
\end{align}
and
\begin{align}\label{5.22}
&\Big(G(t-s)[w\Gamma(f,f)(s)]\Big)(t,x,v)\nonumber\\
&=\Fi_{\{t_1\leq s\}} e^{-\nu(v)(t-s)} [w\Gamma(f,f)](s,x-v(t-s),v)\nonumber\\
&\quad+\f{e^{-\nu(v)(t-t_1)}}{\widetilde{w}(v)}
\int_{\Pi_{j=1}^{k-1}\mathcal{V}_j} \sum_{l=1}^{k-1} \Fi_{\{t_{l+1}\leq s< t_l\}} [w\Gamma(f,f)](s,x_l-v_l(t_l-s),v_l)\,\dd\Sigma_{l}(s)\nonumber\\
&\quad+\f{e^{-\nu(v)(t-t_1)}}{\widetilde{w}(v)}
\int_{\Pi_{j=1}^{k-1}\mathcal{V}_j} \Fi_{\{t_k>s\}} \Big(G(t-s)[w\Gamma(f,f)(s)]\Big)(t_k,x_k,v_{k-1})\,\dd\Sigma_{k-1}(t_k)\nonumber\\
&:=J_{21}+J_{22}+J_{23}.
\end{align}
We estimate integrals of those terms in \eqref{5.21} and \eqref{5.22} over $\{0\leq s\leq t-\la, |v|\leq N\}$ as follows. 
From  \eqref{4.17}, it is noted  that 
\begin{multline*}
\Big|\Big(G(t-s)K_wh(s)\Big)(t_k,x_k,v_{k-1})\Big|+\Big|\Big(G(t-s)[w\Gamma(f,f)(s)]\Big)(t_k,x_k,v_{k-1})\Big|\\
\leq Ce^{-\f{\nu_0}2(t_k-s)}\Big[\|h(s)\|_{L^\infty}+\|h(s)\|_{L^\infty}^2\Big].
\end{multline*}
For $J_{13}$ and $J_{23}$, it then follows that 
\begin{align}
&\int_0^{t-\l}\int_{|v|\leq N} e^{-\f{|v|^2}{8}}(|J_{13}|+|J_{23}|)\,\dd v\dd s\nonumber\\
&\leq C \sup_{0\leq s\leq t} \Big[\|h(s)\|_{L^\infty}+\|h(s)\|_{L^\infty}^2\Big]\nonumber\\
&
\qquad\times\int_0^{t-\l} \int_{|v|\leq N} e^{-\f{|v|^2}{10}}e^{-\f{\nu_0}2(t-s)}
\int_{\Pi_{j=1}^{k-1}\mathcal{V}_j} \Fi_{\{t_k>s\}}\widetilde{w}(v_{k-1})\,\dd\s_{k-1} \cdots \dd\s_1 \dd v\dd s\nonumber\\
&\leq C \sup_{0\leq s\leq t} \Big[\|h(s)\|_{L^\infty}+\|h(s)\|_{L^\infty}^2\Big]\nonumber\\
&
\qquad\times \int_0^{t-\l} \int_{|v|\leq N} e^{-\f{|v|^2}{10}}e^{-\f{\nu_0}2(t-s)}
\int_{\Pi_{j=1}^{k-1}\mathcal{V}_j} \Fi_{\{t_k>s\}}\,\dd\s_{k-2} \cdots \dd\s_1 \dd v\dd s.
\nonumber
\end{align}
Since $t-s\geq\l>0$, one can choose $k=k(\v,T_0)+1$ large enough as in Proposition \ref{lem3.1} such that \eqref{4.7} is valid. Then the above estimate further implies that 
\begin{align*}
&\int_0^{t-\l}\int_{|v|\leq N} e^{-\f{|v|^2}{8}}(|J_{13}|+|J_{23}|)\,\dd v\dd s\leq C\v \sup_{0\leq s\leq t} \Big[\|h(s)\|_{L^\infty}+\|h(s)\|_{L^\infty}^2\Big].
\end{align*}
For $J_{11}$, one can first write that
\begin{align}
&\int_0^{t-\l}\int_{|v|\leq N} e^{-\f{|v|^2}{10}} |J_{11}|\,\dd v\dd s\nonumber\\
&\leq \int_0^{t-\l}\int_{|v|\leq N}e^{-\nu_0(t-s)} \Fi_{\{t_1\leq s\}} e^{-\f{|v|^2}{10}}\left\{\int_{|v'|\leq 2N} +\int_{|v'|\geq2N} \right\}\nonumber\\
&\qquad\qquad\qquad\qquad\qquad\times |k_{w}(v,v')h(s,x-v(t-s),v')|\,\dd v' \dd v\dd s.
\nonumber
\end{align}
It can be further bounded by
\begin{multline}
\label{l53.p2}
Ce^{-\f{N^2}{32}}\sup_{0\leq s\leq t} \|h(s)\|_{L^\infty} 
\int_0^{t-\l}\int_{|v|\leq N}e^{-\nu_0(t-s)} e^{-\f{|v|^2}{10}}\int_{|v'|\geq2N}|k_w(v,v')e^{\f{|v-v'|^2}{32}}|\,\dd v'\dd v\dd s\\
+\int_0^{t-\l}\int_{|v|\leq N}e^{-\nu_0(t-s)} \Fi_{\{t_1\leq s\}} e^{-\f{|v|^2}{10}} \int_{|v'|\leq 2N} |k_{w}(v,v')h(s,x-v(t-s),v')|\,\dd v' \dd v\dd s.
\end{multline}
Here, the first term on the right is bounded by
$\f{C}{N}\sup_{0\leq s\leq t} \|h(s)\|_{L^\infty}$,
due to the fact that
\begin{equation}\nonumber
\int_{\mathbb{R}^3}|k_w(v,v')|e^{\f{|v-v'|^2}{32}}\,\dd v'\leq C,
\end{equation}
by \eqref{2.3}. Using \eqref{2.2}, it is direct to verify
\begin{equation}\nonumber
\int_{|v'|\leq 2N}|k_{w}(v,v')|^2\,\dd v'\leq C,
\end{equation}
which, together with   the Cauchy-Schwarz inequality,  yields that the second term on the right-hand side of \eqref{l53.p2} is bounded by
\begin{equation}\nonumber
\begin{split}
&C_N\int_0^{t-\l}e^{-\nu_0(t-s)} \Big(\int_{|v|\leq N} \int_{|v'|\leq 2N}e^{-\f{|v|^2}{10}}|k_{w}(v,v')|^2\dd v'\dd v\Big)^{\f12} \\
&\quad \times \Big(\int_{|v|\leq N} \int_{|v'|\leq 2N}e^{-\f{|v|^2}{10}}\Fi_{\{t_1\leq s\}} |f(s,x-v(t-s),v')|^2\dd v' \dd v\Big)^{\f12}\,\dd s\\
&\leq C_N\int_0^{t-\l}e^{-\nu_0(t-s)}\Big(\int_{|v|\leq N} \int_{|v'|\leq 2N}e^{-\f{|v|^2}{10}}\Fi_{\{t_1\leq s\}} |f(s,x-v(t-s),v')|^2\dd v' \dd v\Big)^{\f12}\,\dd s.
\end{split}
\end{equation}
Since $y:=x-v(t-s)\in\Omega$, by making a change of variable $v\mapsto y$ with $|\f{dy}{dv}|=(t-s)^3\geq \l^3$ for $t-s\geq\l$, it holds that 
\begin{equation}\label{5.29}
\int_{|v|\leq N} \int_{|v'|\leq 2N}e^{-\f{|v|^2}{10}}\Fi_{\{t_1\leq s\}} |f(s,x-v(t-s),v')|^2\,\dd v' \dd v
\leq \f{C}{\la^{3}}\|f(s)\|^2_{L^2}.
\end{equation}
Plugging those estimates back to \eqref{l53.p2}, it holds that
\begin{align*}
\int_0^{t-\l}\int_{|v|\leq N} e^{-\f{|v|^2}{10}} |J_{11}|\,\dd v\dd s
\leq \f{C}{N}\sup_{0\leq s\leq t} \|h(s)\|_{L^\infty}+C_{N,\l}\sup_{0\leq s\leq t}\|f(s)\|_{L^2}.
\end{align*}
For  $J_{12}$, one has that for fixed $(x,v)$,
\begin{align}
\int_0^{t-\l}|J_{12}|\,\dd s&\leq \int_0^{t-\l}\f{e^{-\nu(v)(t-t_1)}}{\widetilde{w}(v)}
\int_{\Pi_{j=1}^{k-1}\mathcal{V}_j} \sum_{l=1}^{k-1} \Fi_{\{t_{l+1}\leq s< t_l\}} \Big|(K_wh)(s,x_l-v_l(t_l-s),v_l)\Big|\,\dd\Sigma_{l}(s)\dd s\notag\\
&\leq C\sum_{l=1}^{k-1}\int_0^{t-\l} e^{-\nu_0(t-s)} \int_{\Pi_{j=1}^{l-1}\mathcal{V}_j}\int_{\mathcal{V}_l} \Fi_{\{t_{l+1}\leq s< t_l\}}\notag\\
&\qquad\quad\times  \int_{\mathbb{R}^3}\Big|k_w(v_l,v')h(s,x_l-v_l(t_l-s),v')\Big|e^{-\f{|v_l|^2}8}\,\dd v'\dd v_{l}\dd\s_{l-1}\cdots \dd\s_1\dd s.\label{5.31}
\end{align}
For 
each $l=1,\cdots,k-1$, we estimate the right-hand term in the way that 
\begin{align}
&\int_0^{t-\la}\int_{\mathcal{V}_l} \int_{\mathbb{R}^3}\Fi_{\{t_{l+1}\leq s< t_l\}}\{\cdots\}\,\dd v' \dd v_l \dd s\nonumber\\
&=\int_0^{t-\la}\left\{\int_{|v_l|\geq N} \int_{\mathbb{R}^3}+\int_{|v_l|\leq N} \int_{|v'|\geq2N}\right\}\Fi_{\{t_{l+1}\leq s< t_l\}}\{\cdots\}\,\dd v' \dd v_l \dd s\nonumber\\
&\quad +\int_0^{t-\la}\int_{|v_l|\leq N} \int_{|v'|\leq2N}\Fi_{\{t_{l+1}\leq s< t_l-\frac{1}{N}\}}\{\cdots\}\,\dd v' \dd v_l \dd s\nonumber\\
&\quad +\int_0^{t-\la}\int_{|v_l|\leq N} \int_{|v'|\leq2N}\Fi_{\{ t_l-\frac{1}{N}\leq s< t_l\}}\{\cdots\}\,\dd v' \dd v_l \dd s,\label{in.vvs}
\end{align}
so that each term $(l=1,\cdots,k-1)$  on the right-hand side of \eqref{5.31} is bounded by 
\begin{align}\label{5.32}
&
\f{C}{N}\sup_{0\leq s\leq t}\|h(s)\|_{L^\infty}+C\int_0^{t-\l} e^{-\nu_0(t-s)} \int_{\Pi_{j=1}^{l-1}\mathcal{V}_j}\int_{|v_l|\leq N} \int_{|v'|\leq2N}\Fi_{\{t_{l+1}\leq s\leq t_l- \f{1}{N}\}}\{\cdots\}.
\end{align}
The last term above can be bounded in terms of the Cauchy-Schwarz by
\begin{multline}
C\int_0^{t-\l} e^{-\nu_0(t-s)} \int_{\Pi_{j=1}^{l-1}\mathcal{V}_j}\dd\s_{l-1}\cdots \dd\s_1\dd s\left(\int_{|v_l|\leq N} \int_{|v'|\leq2N} |k_w(v_l,v')|^2 e^{-\f{|v_l|^2}{8}}\dd v'\dd v_l \right)^{\f12}\nonumber\\
\times\left(\int_{|v_l|\leq N} \int_{|v'|\leq2N}\Fi_{\{t_{l+1}\leq s\leq t_l- \f{1}{N}\}} |h(s,x_l-v_l(t_l-s),v')|^2 e^{-\f{|v_l|^2}{8}}\dd v'\dd v_l \right)^{\f12},\nonumber
\end{multline}
and is further bounded by
\begin{multline}
C\int_0^{t-\l} e^{-\nu_0(t-s)} \int_{\Pi_{j=1}^{l-1}\mathcal{V}_j}\dd\s_{l-1}\cdots \dd\s_1\dd s\\
\times\left(\int_{|v_l|\leq N} \int_{|v'|\leq2N}\Fi_{\{t_{l+1}\leq s\leq t_l- \f{1}{N}\}} |h(s,x_l-v_l(t_l-s),v')|^2 e^{-\f{|v_l|^2}{8}} \dd v'dv_l\right)^{\f12}.\label{5.32.p1}
\end{multline}
Since $s\in[t_{l+1},t_l-1/N]$, one has $y_l:=x_l-v_l(t_l-s)\in\Omega$. By taking a change of variable $v_l\mapsto y_l=x_l-v_l(t_l-s)$ with 
$$
\left|\f{dy_l}{dv_l}\right|=(t_l-s)^3\geq \f{1}{N^3},
$$ for $t_l-s\geq 1/N$, we see that \eqref{5.32.p1} and hence the second term of \eqref{5.32}  are bounded by
\begin{align}\label{5.34}
C_N\int_0^{t-\l}e^{-\nu_0(t-s)}\|f(s)\|_{L^2}\,\dd s\leq C_N \sup_{0\leq s\leq t}\|f(s)\|_{L^2}.
\end{align}
Combining \eqref{5.34}, \eqref{5.32} and \eqref{5.31}, one obtains that 
\begin{align*}
\int_0^{t-\l}\int_{|v|\leq N} e^{-\f{|v|^2}{8}}|J_{12}|\,\dd v\dd s
\leq \f{C_{\v,T_0}}{N}\sup_{0\leq s\leq t}\|h(s)\|_{L^\infty}
+C_{\v,N,T_0}\sup_{0\leq s\leq t}\|f(s)\|_{L^2},
\end{align*}
since  $k=k(\v,T_0)+1$.

It remains to estimate the same integrals involving $J_{21}$ and $J_{22}$. Setting $y:=x-v(t-s)$, 
it follows from  \eqref{2.4} and the Cauchy-Schwarz that
\begin{align}\label{5.36}
&|w(v)\Gamma(f,f)(s,y,v)|\nonumber\\
&\leq |w(v)\Gamma_+(f,f)(s,y,v)|+ |w(v)\Gamma_-(f,f)(s,y,v)|\nonumber\\
&\leq C\|h(s)\|_{L^\infty} \left(\int_{\mathbb{R}^3}(1+|\eta|)^{4}|e^{\varpi|\eta|^2}f(s,y,\eta)|^2\,\dd\eta\right)^{\f12}
+C\nu(v)\|h(s)\|_{L^\infty}\int_{\mathbb{R}^3}e^{-\f{|u|^2}{8}} |f(s,y,u)|\,\dd u\nonumber\\
&\leq C\nu(v)\|h(s)\|_{L^\infty}\left(\int_{\mathbb{R}^3}(1+|\eta|)^{-4\b+4}|h(s,y,\eta)|^2\,\dd\eta\right)^{\f12}.
\end{align}
For $J_{21}$, by applying the above estimate, one then has
\begin{align*}
&\int_0^{t-\l}\int_{|v|\leq N} e^{-\f{|v|^2}{8}}|J_{21}|\,\dd v\dd s\nonumber\\
&\leq C\int_0^{t-\l}\int_{|v|\leq N} e^{-\nu_0(t-s)} e^{-\f{|v|^2}{10}}  \|h(s)\|_{L^\infty}\left(\int_{\mathbb{R}^3}(1+|\eta|)^{-4\b+4}\Fi_{\{t_1\leq s\}} |h(s,y,\eta)|^2\,\dd\eta\right)^{\f12}\dd v\dd s,\notag
\end{align*}
which by the H\"older is further bounded by
\begin{equation}
\label{5.37.p1}
C\int_0^{t-\l}e^{-\nu_0(t-s)} \|h(s)\|_{L^\infty} \left(\int_{|v|\leq N} \int_{\mathbb{R}^3}e^{-\f{|v|^2}{10}} (1+|\eta|)^{-4\b+4}\Fi_{\{t_1\leq s\}} |h(s,y,\eta)|^2\,\dd\eta \dd v\right)^{\f12}\dd s. 
\end{equation}
For the above term, we consider the integral in $\eta$ over $\{|\eta|>N\}\cup \{|\eta|\leq N\}$. On one hand, over $\{|\eta|>N\}$, by using the fact that 
\begin{align}
\int_{|v|\leq N}\int_{|\eta|\geq N}e^{-\f{|v|^2}{10}} (1+|\eta|)^{-4\b+4}\,\dd\eta \dd v\leq \f{C}{N^2},\nonumber
\end{align}
for $\be\geq 5/2$, we see that  \eqref{5.37.p1} is bounded by
$\f{C}{N} \sup_{0\leq s\leq t}\|h(s)\|^2_{L^\infty}$.
On the other hand, over $\{|\eta|\leq N\}$,  \eqref{5.37.p1} is bounded by
\begin{align}
& C_{N}\int_0^{t-\l}e^{-\nu_0(t-s)} \|h(s)\|_{L^\infty} \left(\int_{|v|\leq N} \int_{|\eta|\leq N} \Fi_{\{t_1\leq s\}} |f(s,y,\eta)|^2\dd\eta \dd v\right)^{\f12}\dd s\nonumber\\
&\leq C_{\l,N}\int_0^{t-\l}e^{-\nu_0(t-s)} \|h(s)\|_{L^\infty}\|f(s)\|_{L^2}\,\dd s\nonumber\\
&\leq \f{C}{N} \sup_{0\leq s\leq t}\|h(s)\|^2_{L^\infty}+C_{\l,N} \sup_{0\leq s\leq t}\|f(s)\|^2_{L^2}, \nonumber
\end{align}
where as in \eqref{5.29}, since $y=x-v(t-s)\in\Omega$, we have made a change of variable $v\mapsto y$ with $|\f{dy}{dv}|=(t-s)^3\geq \l^3$ for $t-s\geq\l$. Combining those estimates, we then have
\begin{equation*}
\int_0^{t-\l}\int_{|v|\leq N} e^{-\f{|v|^2}{8}}|J_{21}|\,\dd v\dd s\leq \f{C}{N} \sup_{0\leq s\leq t}\|h(s)\|^2_{L^\infty}+C_{\l,N} \sup_{0\leq s\leq t}\|f(s)\|^2_{L^2}.
\end{equation*} 
For $J_{22}$, setting $y_l=x_l-v_l(t_l-s)$, we notice that for fixed $(x,v)$,
 \begin{equation}\nonumber
\int_0^{t-\l}|J_{22}|\,\dd s\leq \int_0^{t-\l}\f{e^{-\nu(v)(t-t_1)}}{\widetilde{w}(v)}
\int_{\Pi_{j=1}^{k-1}\mathcal{V}_j} \sum_{l=1}^{k-1} \Fi_{\{t_{l+1}\leq s< t_l\}}\Big| [w\Gamma(f,f)](s,y_l,v_l)\Big|\,\dd\Sigma_{l}(s)\dd s.
\end{equation}
From \eqref{5.36}, the term on the right can be bounded as
\begin{multline}\nonumber
C\sum_{l=1}^{k-1}\int_0^{t-\l} e^{-\nu_0(t-s)} \int_{\Pi_{j=1}^{l-1}\mathcal{V}_j}\int_{\mathcal{V}_l} \Fi_{\{t_{l+1}\leq s< t_l\}}e^{-\f{|v_l|^2}8}\nu(v_l) \|h(s)\|_{L^\infty}\\
\times\left(\int_{\mathbb{R}^3}{(1+|\eta|)^{-4\b+4}}|h(s,y_l,\eta)|^2\,\dd\eta\right)^{\f12} \,\dd v_{l}\dd\s_{l-1}\cdots \dd\s_1\dd s.
\end{multline}
To estimate the above term, we consider the integral over either $\{|v_l|\geq N\}$ or $\{|\eta|\geq N\}$ or $\{t_l-1/N\leq s<t_l\}$ or $\{|v_l|<N,|\eta|<N, t_{l+1}\leq s<t_l-1/N\}$. Over $\{|v_l|\geq N\}$ or $\{|\eta|\geq N\}$ or $\{t_l-1/N\leq s<t_l\}$, it is bounded by
\begin{equation}\nonumber
\f{C_{\v,T_0}}{N} \sup_{0\leq s\leq t}\|h(s)\|^2_{L^\infty}, 
\end{equation}
and over $\{|v_l|<N,|\eta|<N, t_{l+1}\leq s<t_l-1/N\}$, it is bounded by
\begin{equation*}
\begin{split}
&C_N\sum_{l=1}^{k-1}\int_0^{t-\l} e^{-\nu_0(t-s)} \|h(s)\|_{L^\infty}\nonumber\\
&\quad\times\int_{\Pi_{j=1}^{l-1}\mathcal{V}_j}\left(\int_{|v_l|\leq N} \int_{|\eta|\leq N}\Fi_{\{t_{l+1}\leq s< t_l-\f{1}{N}\}} |f(s,y_l,\eta)|^2\,\dd\eta \dd v_{l}\right)^{\f12} \dd\s_{l-1}\cdots \dd\s_1\dd s\nonumber\\
&\leq C_{\v,N,T_0}\int_0^{t-\l}e^{-\nu_0(t-s)}\|h(s)\|_{L^\infty}\|f(s)\|_{L^2}\,\dd s\nonumber\\
&\leq  \f{C_{\v,T_0}}{N} \sup_{0\leq s\leq t}\|h(s)\|^2_{L^\infty}+ C_{\v,N,T_0}\sup_{0\leq s\leq t}\|f(s)\|^2_{L^2}, 
\end{split}
\end{equation*}
where we have used the change of variable $v_l\rightarrow y_l$ as in \eqref{5.34} above. Hence, for $J_{22}$, it follows that
\begin{equation}\nonumber
\int_0^{t-\l}|J_{22}|\,\dd s\leq\f{C_{\v,T_0}}{N} \sup_{0\leq s\leq t}\|h(s)\|^2_{L^\infty}+ C_{\v,N,T_0}\sup_{0\leq s\leq t}\|f(s)\|^2_{L^2},
\end{equation}
and hence
\begin{align*}
\int_0^{t-\l}\int_{|v|\leq N} e^{-\f{|v|^2}{8}}|J_{22}|\,\dd v\dd s &\leq \int_{\R^3}e^{-\f{|v|^2}{8}}\,\dd v \int_0^{t-\l}|J_{22}|\,\dd s\nonumber\\
&\leq \f{C_{\v,T_0}}{N}\sup_{0\leq s\leq t}\|h(s)\|_{L^\infty}^2
+C_{\v,N,T_0}\sup_{0\leq s\leq t}\|f(s)\|_{L^2}^2.
\end{align*}
Therefore, we have completed all estimates on those terms on the right-hand side of \eqref{l53.p1}. Combining them all, 
one obtains that 
\begin{align}\label{5.39}
\int_{\mathbb{R}^3} e^{-\f{|v|^2}{8}} |h(t,x,v)|\,\dd v&\leq C_7e^{-\f{\nu_0}{2}t}\|h_0\|_{L^\infty}
+C_7\Big(\v+\l+\f{C_{\v,T_0}}{N} \Big) \sup_{0\leq s\leq t}\Big[\|h(s)\|_{L^\infty}+\|h(s)\|^2_{L^\infty}\Big]\nonumber\\
&\quad + C_{\v,\l,N,T_0} \sup_{0\leq s\leq t}\Big[\|f(s)\|_{L^2}+\|f(s)\|^2_{L^2}\Big],
\end{align}
where $C_7\geq1$ is a generic constant.
Letting
$$
\tilde{t}:=\f{2}{\nu_0}\ln\left(4C_6C_7M_0\right),
$$
it then holds that 
\begin{align}\label{5.40}
C_7e^{-\f{\nu_0}{2}t}\|h_0\|_{L^\infty}\leq C_7e^{-\f{\nu_0}{2}t}M_0\leq \f1{4C_6},
\end{align}
for all $t\geq \tilde{t}$.
On the other hand,  it follows from \eqref{5.0} and \eqref{5.5} that the remaining part on the right-hand side of \eqref{5.39} is bounded as
\begin{multline}
 C_7\Big(\v+\l+\f{C_{\v,T_0}}{N} \Big) \sup_{0\leq s\leq t}\Big[\|h(s)\|_{L^\infty}+\|h(s)\|^2_{L^\infty}\Big]+ C_{\v,\l,N,T_0} \sup_{0\leq s\leq t}\Big[\|f(s)\|_{L^2}+\|f(s)\|^2_{L^2}\Big]
\\
\leq 2C_7\Big(\v+\l+\f{C_{\v,T_0}}{N} \Big)\bar{M}^2+C_{\v,\l,N,T_0} e^{2\tilde{C}_1\bar{M} T_0} \Big[\|f_0\|_{L^2}+\|f_0\|^2_{L^2}\Big],
\label{5.40.ap1}
\end{multline}
for $0\leq t\leq T_0$. 
We first take $\l>0$ and $\v>0$ small enough, then let $N>0$ suitably large, and finally take initial data $f_0$ such that $\|f_0\|_{L^2}\leq \epsilon_1$ with $\epsilon_1=\epsilon_1(\bar{M},T_0)>0$ sufficiently small, so that \eqref{5.40.ap1} is bounded by
\begin{align}\label{5.42}
2C_7\Big(\v+\l+\f{C_{\v,T_0}}{N} \Big)\bar{M}^2+C_{\v,\l,N,T_0} e^{2\tilde{C}_1\bar{M} T_0} \left[\epsilon_1+\epsilon_1^2\right]\leq \f1{4C_6}.
\end{align}
Notice that $\epsilon_1(\bar{M},T_0)$ can be decreasing in $\bar{M}$ and $T_0$. Applying  \eqref{5.40}, \eqref{5.40.ap1} and \eqref{5.42} for \eqref{5.39}, one has
\begin{align}\nonumber
\int_{\mathbb{R}^3} e^{-\f{|v|^2}{8}} |h(t,x,v)|\,\dd v\leq \f1{2C_6},
\end{align}
for all $(t,x)\in [\tilde{t},T_0)\times \Omega$. This proves \eqref{5.13.ap1} and hence \eqref{5.13} follows. 
Setting $\tilde{C}_2:=4C_6C_7\geq1$, we then complete the proof of Lemma \ref{lem5.3}.   
\end{proof}

\subsection{Estimate in $L^\infty_{x,v}$}

To treat the nonlinear $L^\infty$ estimate for solutions with large amplitude, we need to use the time-decay property of the linear solution operator over the time interval $[0,T_0]$, where we recall that $T_0(>\tilde{t})$ depending only on $M_0$ is to be determined in the end of the proof. From the proof of the below lemma, such time-decay property is an immediate consequence of Lemma \ref{lem5.3}. Note that Lemma \ref{lem5.3} is based on the $L^\infty_xL^1_v$ estimate on the solution. It thus means that for the nonlinear equation we are going to make a bootstrap argument from $L^\infty_xL^1_v$ to $L^\infty$ in finite time. 

\begin{lemma}\label{lem5.4}
	Let ${\tilde{G}(t,s):=G^f(t,s)}$ be the linear operator defined in Lemma \ref{lem3.02} with $\varphi(t,x,v)$ replaced by the solution of the Boltzmann equation $f(t,x,v)$. Then, under the a priori assumption \eqref{5.0}, it holds that 
	\begin{align}\label{5.43}
	\|{\tilde{G}(t,s)}h_0\|_{L^\infty}\leq Ce^{\f34\nu_0 \tilde{t}} {e^{-\f14\nu_0(t-s)}}\|h_0\|_{L^\infty}\quad\mbox{for all}\ 0\leq s\leq t\leq T_0,
	\end{align}
	provided that initial data $f_0(x,v)$ satisfies $\|f_0\|_{L^2}\leq \epsilon_1(\bar{M},T_0)$, where $C>0$ is a generic constant.
\end{lemma}

\begin{proof}
We assume $\|f_0\|_{L^2}\leq \epsilon_1=\epsilon_1(\bar{M},T_0)$. 
Then, by Lemma \ref{lem5.3}, it holds that 
\begin{equation}\label{5.44}
R(f)(t,x,v)\geq 
\left\{\begin{aligned}
&0,\qquad\quad t\in[0,\tilde{t}],\\
&\f12\nu(v),\quad t\in[\tilde{t},T_0],
\end{aligned}\right.
\end{equation}
for all $(x,v)\in \Omega\times \R^3$. {We only consider the estimate for $\tilde{G}(t,0)$ since the estimate for $\tilde{G}(t,s)$ can be obtained similarly.} Thus, \eqref{5.43} follows by two cases:

\medskip
\noindent{\it Case of $t\in[0,\tilde{t}]$.} In this case, one can see $t/\rho\in [m,m+1)$ for some  $m\in \{0,1,\cdots, [\tilde{t}/\rho ]\}$. Thus it follows from \eqref{5.00} and  Lemma \ref{lem3.02} that 
\begin{align}\label{5.45}
\|{\tilde{G}(t,0)}h_0\|_{L^\infty}\leq C_3\r^{2\b} \|{\tilde{G}(m\r,0)}h_0\|_{L^\infty}&\leq C_3\r^{2\b} (C_3\r^{2\b})^m\|h_0\|_{L^\infty}\nonumber\\
&\leq C_3\r^{2\b} \left[(C_3\r^{2\b})^{\f1\r}\right]^{\tilde{t}}\|h_0\|_{L^\infty}\leq C e^{\f12\nu_0 \tilde{t}}\|h_0\|_{L^\infty},
\end{align}
and hence
\begin{equation}
\label{5.45-1}
\|{\tilde{G}(t,0)}h_0\|_{L^\infty}\leq Ce^{\f34\nu_0\tilde{t}}e^{-\f14\nu_0t}\|h_0\|_{L^\infty}.
\end{equation}

\medskip
\noindent{\it Case of $t\in[\tilde{t},T_0]$}. In this case  we write 
${\tilde{G}(t,0)}h_0={\tilde{G}(t,\tilde{t})}{\tilde{G}(\tilde{t},0)}h_0$.
From \eqref{5.44}, we know  that ${\tilde{G}(t,\tilde{t})}$  has the property in \eqref{3.24}.  Thus it follows from \eqref{3.24} and \eqref{5.45} that 
\begin{equation}\label{5.45-2}
\|{\tilde{G}(t,0)}h_0\|_{L^\infty}\leq C e^{-\f14\nu_0(t-\tilde{t})} \|{\tilde{G}(\tilde{t},0)}h_0\|_{L^\infty} 
\leq C e^{-\f14\nu_0(t-\tilde{t})} e^{\f12\nu_0\tilde{t}}\|h_0\|_{L^\infty}
\leq Ce^{\f34\nu_0\tilde{t}} e^{-\f14\nu_0t}\|h_0\|_{L^\infty}.
\end{equation}
Therefore, combing both cases, \eqref{5.45-1} and \eqref{5.45-2}   yield \eqref{5.43}.  The proof of Lemma \ref{lem5.4} is completed.     
\end{proof}


In what follows we try to apply Lemma \ref{lem5.4} to treat the nonlinear $L^\infty$ estimate over the time interval $[0,T_0]$. For that purpose, we recall the Boltzmann equation \eqref{1.9} and use the Duhamel principle to rewrite it in terms of $h(t,x,v)=w(v)f(t,x,v)$ as 
\begin{align}\label{hw}
h(t,x,v)&={\tilde{G}(t,0)}h_0+\int_0^t{\tilde{G}(t,s)}K_wh(s)\,\dd s+\int_0^t{\tilde{G}(t,s)}w\Gamma_+(f,f)(s)\,\dd s,
\end{align}
where ${\tilde{G}(t,s):=G^f(t,s)}$ is defined as in Lemma \ref{lem5.4}.  As the first step for the $L^\infty$ estimate, we have the following 

\begin{lemma}\label{lem.ad1}
Assume $\|f_0\|_{L^2}\leq \epsilon_1=\epsilon_1(\bar{M},T_0)$. Let $(t,x,v)\in (0,T_0]\times \Omega\times \R^3$, then it holds that 
\begin{align}\label{5.65}
|h(t,x,v)|
&\leq e^{\f{\nu_0}{4}\tilde{t}}\int_0^{t-\l}\Fi_{\{t_1\leq s\}} e^{-\f{\nu_0}{4}(t-s)}\Big|(K_wh)(s,x-v(t-s),v)\Big|\,\dd s\nonumber\\
&\quad+e^{\f{\nu_0}{4}\tilde{t}}\int_0^{t-\l}\Fi_{\{t_1\leq s\}} e^{-\f{\nu_0}{4}(t-s)}\Big|[w\Gamma_+(f,f)](s,x-v(t-s),v)\Big|\,\dd s\nonumber\\
&\quad+Ce^{\nu_0\tilde{t}}\left\{e^{-\f{\nu_0}{4}t} \|h_0\|_{L^\infty}+\Big(\l+\v+\f{C_{\v,T_0}}{N}\Big)\sup_{0\leq s\leq t}\left[\|h(s)\|_{L^\infty}+\|h(s)\|^2_{L^\infty}\right]\right.\nonumber\\
&\qquad\qquad\quad\left.+C_{\v, N, T_0}\sup_{0\leq s\leq t}\left[\|f(s)\|_{L^2}+\|f(s)\|^2_{L^2}\right]\right\},
\end{align}
where $\la>0$ and $\v>0$ can be arbitrarily small, and $N>0$ can be arbitrarily large. 
\end{lemma}

\begin{proof}
Take $(t,x,v)\in (0,T_0]\times \Omega\times \R^3$. We rewrite \eqref{hw} as 
\begin{align}\label{5.47}
h(t,x,v)
&={\tilde{G}(t,0)}h_0+\int_{t-\l}^t\Big[{\tilde{G}(t,s)}K_wh(s)+{\tilde{G}(t,s)}w\Gamma_+(f,f)(s)\Big]\,\dd s\nonumber\\
&\quad+\int_0^{t-\l}{\tilde{G}(t,s)}K_wh(s)\,\dd s+\int_0^{t-\l}{\tilde{G}(t,s)}w\Gamma_+(f,f)(s)\,\dd s,
\end{align}
where $\la>0$ is to be fixed suitably small later. 
Firstly, it follows from Lemma \ref{lem5.4}  that 
\begin{align}\label{5.48}
\Big|{\tilde{G}(t,0)}h_0\Big|\leq Ce^{\f34\nu_0\tilde{t}}e^{-\f{\nu_0}{4}t} \|h_0\|_{L^\infty},
\end{align}
and 
\begin{equation}\label{5.48-1}
\int_{t-\l}^t{\tilde{G}(t,s)}K_wh(s)\,\dd s+\int_{t-\l}^t{\tilde{G}(t,s)}w\Gamma_+(f,f)(s)\,\dd s 
\leq C\l e^{\f34\nu_0\tilde{t}} \sup_{0\leq s\leq t}\Big[\|h(s)\|_{L^\infty}+\|h(s)\|^2_{L^\infty}\Big].
\end{equation}
It remains to estimate the last two terms on the right-hand of \eqref{5.47}.

\medskip
\noindent{\underline{Estimate on $\int_0^{t-\l}{\tilde{G}(t,s)}K_wh(s)\,\dd s$:}} It is noticed that 
\begin{align}\label{5.49}
I_1:=&\Big({\tilde{G}(t,s)}K_wh(s)\Big)(t,x,v)\nonumber\\
=&\Fi_{\{t_1\leq s\}} I^f(t,s) (K_wh)(s,x-v(t-s),v)\nonumber\\
&+\f{I^f(t,t_1)}{\widetilde{w}(v)}
\int_{\Pi_{j=1}^{k-1}\mathcal{V}_j} \sum_{l=1}^{k-1} \Fi_{\{t_{l+1}\leq s< t_l\}} (K_wh)(s,x_l-v_l(t_l-s),v_l)\,\dd\Sigma^f_{l}(s)\nonumber\\
&+\f{I^f(t,t_1)}{\widetilde{w}(v)}
\int_{\Pi_{j=1}^{k-1}\mathcal{V}_j} \Fi_{\{t_k>s\}} \Big({\tilde{G}(t,s)}K_wh(s)\Big)(t_k,x_k,v_{k-1})\,\dd\Sigma^f_{k-1}(t_k)\nonumber\\
:= &I_{11}+I_{12}+I_{13},
\end{align}
where $I_{11}$, $I_{12}$ and $I_{13}$ denote the right-hand terms, respectively.  
By the definition of $I^f(t,s)$ and Lemma \ref{lem5.3}, it holds that
\begin{equation}\label{5.50}
\left\{\begin{aligned}
&I^f(t,s)\leq e^{\f{\nu_0}{4}\tilde{t}} e^{-\f{\nu_0}{4}(t-s)},\\
&I^f(t,t_1)I^f(t_1,t_2)\cdots I^f(t_{l-1},t_{l})I^f(t_l,s)\leq e^{\f{\nu_0}{4}\tilde{t}} e^{-\f{\nu_0}{4}(t-s)}.
\end{aligned}\right.
\end{equation}
For $I_{13}$, it follows from Lemma \ref{lem5.4} that
\begin{align}\label{5.52}
\Big|\Big({\tilde{G}(t,s)}K_wh(s)\Big)(t_k,x_k,v_{k-1})\Big|
&\leq Ce^{\f{3\nu_0}{4}\tilde{t}} e^{-\f14\nu_0(t_k-s)}\|K_wh(s)\|_{L^\infty}\nonumber\\
&\leq Ce^{\f{3\nu_0}{4}\tilde{t}} e^{-\f14\nu_0(t_k-s)}\|h(s)\|_{L^\infty}.
\end{align}
Thus, for $t-s\geq\l>0$, then it follows from  \eqref{5.50}, \eqref{5.52} and \eqref{4.7} that
\begin{align}\label{5.53}
I_{13}&\leq Ce^{\nu_0\tilde{t}} e^{-\f14\nu_0(t-s)}\|h(s)\|_{L^\infty} 
\int_{\Pi_{j=1}^{k-1}\mathcal{V}_j} \Fi_{\{t_{k}>s\}} \widetilde{w}(v_{k-1})\,\dd\s_{k-1}\dd\s_{k-2}\cdots \dd\s_1\nonumber\\
& \leq C\v e^{\nu_0\tilde{t}} e^{-\f14\nu_0(t-s)}\|h(s)\|_{L^\infty},
\end{align}
where we have taken $k=k(\v,T_0)+1$ such that \eqref{4.7} holds.

For $I_{12}$, we estimate it as follows. By denoting $I_{12l}$ to be each term $l=1,\cdots,k-1$ in the summation and using  \eqref{5.50}, one has
\begin{multline}\nonumber
\int_0^{t-\l} I_{12l}\,\dd s \leq \f{e^{\f{\nu_0}{4}\tilde{t}}}{\widetilde{w}(v)}\int_0^{t-\l} e^{-\f14\nu_0(t-s)} \int_{\Pi_{j=1}^{l-1}\mathcal{V}_j} \int_{\mathcal{V}_l}\Fi_{\{t_{l+1}\leq s< t_l\}} \widetilde{w}(v_l)\\
 \times\int_{\mathbb{R}^3} |k_{w}(v_l,v') h(s,x_l-v_l(t_l-s),v')|\,\dd v' \dd\s_l \dd\s_{l-1}\cdots \dd\s_1\dd s.
\end{multline}
For the term on the right, as for treating \eqref{5.31}, we consider the integral $\dd v'\dd v_l\dd s$ in the way of \eqref{in.vvs}. Then, the right-hand term above is bounded by
\begin{multline}\nonumber
 \f{C}{N}e^{\f{\nu_0}{4}\tilde{t}}\sup_{0\leq s\leq t}\|h(s)\|_{L^\infty}
 +Ce^{\f{\nu_0}{4}\tilde{t}}\int_0^{t-\l} e^{-\f14\nu_0(t-s)} \int_{\Pi_{j=1}^{l-1}\mathcal{V}_j} \int_{|v_l|\leq N}\int_{|v'|\leq2 N} \Fi_{\{t_{l+1}\leq s< t_l-\f1N\}}\\
\times e^{-\f{|v_l|^2}{8}}|k_{w}(v_l,v') h(s,x_l-v_l(t_l-s),v')|\,\dd v'\dd v_l \dd\s_{l-1}\cdots \dd\s_1\dd s\\
\leq \f{C}{N}e^{\f{\nu_0}{4}\tilde{t}}\sup_{0\leq s\leq t}\|h(s)\|_{L^\infty}+C_Ne^{\f{\nu_0}{4}\tilde{t}} \sup_{0\leq s\leq t}\|f(s)\|_{L^2},
\end{multline}
where we have used the change of variable $v_l\mapsto y_l=x_l-v_l(t_l-s)$ as before.
Thus, after taking $k=k(\v,T_0)+1$,  it holds that
\begin{align}\label{5.55}
\int_0^{t-\l} I_{12}\,\dd s=\sum_{l=1}^{k-1}\int_0^{t-\l} I_{12l}\,\dd s
\leq \f{C_{\v,T_0}}{N}e^{\f{\nu_0}{4}\tilde{t}}\sup_{0\leq s\leq t}\|h(s)\|_{L^\infty}+C_{\v,N,T_0}e^{\f{\nu_0}{4}\tilde{t}} \sup_{0\leq s\leq t}\|f(s)\|_{L^2}.
\end{align}
Plugging \eqref{5.53} and \eqref{5.55} to \eqref{5.49},  and then applying \eqref{5.50} to $I_{11}$, one obtains that 
\begin{align}\label{5.56}
&\left|\int_0^{t-\l}{\tilde{G}(t,s)}K_wh(s)\,\dd s\right|\nonumber\\
&\leq \int_0^{t-\l}\Fi_{\{t_1\leq s\}} e^{\f{\nu_0}{4}\tilde{t}}e^{-\f{\nu_0}{4}(t-s)} \Big|(K_wh)(s,x-v(t-s),v)\Big|\,\dd s\nonumber\\
&\quad +C\Big(\v+\f{C_{\v,T_0}}{N}\Big)e^{\nu_0\tilde{t}}\sup_{0\leq s\leq t}\|h(s)\|_{L^\infty}+C_{\v,N,T_0}e^{\nu_0\tilde{t}} \sup_{0\leq s\leq t}\|f(s)\|_{L^2}.
\end{align}

\medskip
\noindent{\underline{Estimate on $\int_0^{t-\l}{\tilde{G}(t,s)}w\Gamma_+(f,f)(s)\,\dd s$:}} It is noticed that 
\begin{align}\label{5.57}
I_2:=&\Big({\tilde{G}(t,s)}[w\Gamma_+(f,f)(s)]\Big)(t,x,v)\nonumber\\
=&\Fi_{\{t_1\leq s\}} I^f(t,s)[w\Gamma(f,f)](s,x-v(t-s),v)\nonumber\\
&+\f{I^f(t,t_1)}{\widetilde{w}(v)}
\int_{\Pi_{j=1}^{k-1}\mathcal{V}_j} \sum_{l=1}^{k-1} \Fi_{\{t_{l+1}\leq s< t_l\}} [w\Gamma_+(f,f)](s,x_l-v_l(t_l-s),v_l)\,\dd\Sigma_{l}^f(s)\nonumber\\
&+\f{I^f(t,t_1)}{\widetilde{w}(v)}
\int_{\Pi_{j=1}^{k-1}\mathcal{V}_j} I_{\{t_k>s\}} \Big({\tilde{G}(t,s)}[w\Gamma_+(f,f)(s)]\Big)(t_k,x_k,v_{k-1})\,\dd\Sigma_{k-1}^f(t_k)\nonumber\\
:=&I_{21}+I_{22}+I_{23}.
\end{align}
We estimate the integral of each term on the right as follows. For $I_{23}$, it follows from \eqref{2.4-1} and \eqref{5.43} that 
\begin{equation}\nonumber
\Big|\Big({\tilde{G}(t,s)}[w\Gamma_+(f,f)(s)]\Big)(t_k,x_k,v_{k-1})\Big|
\leq Ce^{\f34\nu_0 \tilde{t}} e^{-\f14\nu_0(t_k-s)}\|h(s)\|^2_{L^\infty}.
\end{equation}
This together with \eqref{5.50} and \eqref{4.7} yield that 
\begin{align*}
I_{23}&\leq Ce^{\nu_0 \tilde{t}} e^{-\f14\nu_0(t-s)}\|h(s)\|^2_{L^\infty}\int_{\Pi_{j=1}^{k-1}\mathcal{V}_j} \Fi_{\{t_{k}>s\}} \widetilde{w}(v_{k-1})\,\dd\s_{k-1}\dd\s_{k-2}\cdots \dd\s_1\nonumber\\
 &\leq C\v e^{\nu_0\tilde{t}} e^{-\f14\nu_0(t-s)}\|h(s)\|^2_{L^\infty},\nonumber
\end{align*}
and hence
\begin{equation}
\label{5.59}
\int_0^{t-\la} I_{23}\,\dd s\leq C\v e^{\nu_0\tilde{t}}\sup_{0\leq s\leq t}\|h(s)\|^2_{L^\infty}.
\end{equation}
To estimate $I_{22}$, denoting each term $l=1,\cdots, k-1$ in the summation by $I_{22l}$, 
one has that 
\begin{multline}\label{5.60.a1}
\int_0^{t-\l} I_{22l}\,\dd s\leq C\int_0^{t-\l}e^{\f14\nu_0\tilde{t}} e^{-\f14\nu_0(t-s)}\,\dd s\int_{\Pi_{j=1}^{l-1}\mathcal{V}_j} \dd\s_{l-1}\cdots \dd\s_1\\
\times \int_{\mathcal{V}_l}\Fi_{\{t_{l+1}\leq s<t_l\}}e^{-\f{|v_l|^2}{8}}\Big|[w\Gamma_+(f,f)](s,x_l-v_l(t_l-s),v_l)\Big|\,\dd v_l .
\end{multline}
Over $\{|v_l|>N\}$ or $\{t_l-1/N\leq s<t_l\}$, one can see that \eqref{5.60.a1} is bounded by
\begin{equation}\nonumber
\f{C}{N}e^{\f14\nu_0\tilde{t}}\sup_{0\leq s\leq t}\|h(s)\|^2_{L^\infty}.
\end{equation}
Thus, it remains to estimate
 \begin{multline}\label{5.60}
 C\int_0^{t-\l}e^{\f14\nu_0\tilde{t}} e^{-\f14\nu_0(t-s)}\,\dd s\int_{\Pi_{j=1}^{l-1}\mathcal{V}_j} \dd\s_{l-1}\cdots \dd\s_1\\
\times \int_{\mathcal{V}_l\cap \{|v_l|\leq N\}}\Fi_{\{t_{l+1}\leq s<t_l-\frac{1}{N}\}}e^{-\f{|v_l|^2}{8}}\Big|[w\Gamma_+(f,f)](s,x_l-v_l(t_l-s),v_l)\Big|\,\dd v_l .
\end{multline}
In fact, it follows from \eqref{2.4} that
\begin{align}\label{5.60.a2}
&\int_{|v_l|\leq N} \Fi_{\{t_{l+1}\leq s< t_l-\f1N\}}e^{-\f{|v_l|^2}{8}}\Big|[w\Gamma_+(f,f)](s,x_l-v_l(t_l-s),v_l)\Big| \,\dd v_l\nonumber\\
&\leq C\|h(s)\|_{L^\infty}
\int_{|v_l|\leq N} \dd v_l\, \Fi_{\{t_{l+1}\leq s< t_l-\f1N\}}e^{-\f{|v_l|^2}{8}}\nonumber\\
&\qquad\qquad\qquad\qquad\times\left(\int_{\mathbb{R}^3}(1+|\eta|)^{-4\b+4}|h(s,x_l-v_l(t_l-s),\eta)|^2d\eta\right)^{\f12},
\end{align}
which by the Cauchy-Schwarz is further bounded by
\begin{equation*}
C\|h(s)\|_{L^\infty}\left(\int_{|v_l|\leq N}\int_{\mathbb{R}^3} \Fi_{\{t_{l+1}\leq s< t_l-\f1N\}}e^{-\f{|v_l|^2}{8}}(1+|\eta|)^{-4\b+4}|h(s,x_l-v_l(t_l-s),\eta)|^2\dd\eta \dd v_l\right)^{\f12}.
\end{equation*}
Over $\{|\eta|>N\}$, the above term is bounded by
\begin{equation}\nonumber
\f{C}{N}\sup_{0\leq s\leq t}\|h(s)\|^2_{L^\infty},
\end{equation}
and over $\{|\eta|\leq N\}$, it is bounded by making the change of variable $v_l\mapsto y_l=x_l-v_l(t_l-s)$ as before by
\begin{equation}\nonumber
C_N\|h(s)\|_{L^\infty} \|f(s)\|_{L^2}\leq \f{C}{N}\|h(s)\|^2_{L^\infty}+C_N \|f(s)\|^2_{L^2}.
\end{equation}
Thus, substituting the above estimates to \eqref{5.60.a2}, one can see that \eqref{5.60} is bounded by
\begin{equation}\nonumber
e^{\f14\nu_0\tilde{t}} \left(\f{C}{N}\|h(s)\|^2_{L^\infty}+C_N \|f(s)\|^2_{L^2}\right).
\end{equation}
Hence, by applying those estimates for \eqref{5.60.a1}, 
it holds that
\begin{align}\nonumber
\int_0^{t-\l} I_{22l}\,\dd s\leq \f{C}{N}e^{\f14\nu_0\tilde{t}}\sup_{0\leq s\leq t}\|h(s)\|^2_{L^\infty}+C_Ne^{\f14\nu_0\tilde{t}}\sup_{0\leq s\leq t} \|f(s)\|^2_{L^2},
\end{align}
which yields that 
\begin{align}\label{5.63}
\int_0^{t-\l} I_{22}\,\dd s=\sum_{l=1}^{k-1}\int_0^{t-\l} I_{22l}\,\dd s\leq \f{C_{\v,T_0}}{N}e^{\f14\nu_0\tilde{t}}\sup_{0\leq s\leq t}\|h(s)\|^2_{L^\infty}+C_{\v,N,T_0}e^{\f14\nu_0\tilde{t}}\sup_{0\leq s\leq t} \|f(s)\|^2_{L^2}.
\end{align}
Combining \eqref{5.63}, \eqref{5.59} and \eqref{5.57}, one obtains that 
\begin{align}\label{5.64}
&\left|\int_0^{t-\l}{\tilde{G}(t,s)}w\Gamma_+(f,f)(s)\,\dd s\right|\nonumber\\
&\leq \int_0^{t-\l}\Fi_{\{t_1\leq s\}} e^{\f{\nu_0}{4}\tilde{t}}e^{-\f{\nu_0}{4}(t-s)}\Big|[w\Gamma_+(f,f)](s,x-v(t-s),v)\Big|\,\dd s\nonumber\\
&\quad+C\Big(\v+\f{C_{\v,T_0}}{N}\Big)e^{\nu_0\tilde{t}}\sup_{0\leq s\leq t}\|h(s)\|^2_{L^\infty}+C_{\v,N,T_0}e^{\nu_0\tilde{t}}\sup_{0\leq s\leq t} \|f(s)\|^2_{L^2}.
\end{align}
Therefore, the desired estimate \eqref{5.65} follows by plugging \eqref{5.64}, \eqref{5.56}, \eqref{5.48-1} and \eqref{5.48} into \eqref{5.47}. This completes the proof of Lemma \ref{lem.ad1}. 
\end{proof}

\begin{lemma}\label{lem5.5}
Assume $\|f_0\|_{L^2}\leq \epsilon_1=\epsilon_1(\bar{M},T_0)$. There exists a generic constant  $\tilde{C}_3\geq1$ such that  
\begin{align}\label{5.46}
\|h(t)\|_{L^\infty}&\leq \tilde{C}_3e^{2\nu_0 \tilde{t}} \|h_0\|_{L^\infty} \left[1+\int_0^t\|h(s)\|_{L^\infty}\,\dd s\right]e^{-\f{\nu_0}{8}t}\nonumber\\
&\quad+  \tilde{C}_3e^{2\nu_0 \tilde{t}} \left\{\Big(\v+\l+\f{C_{\v,T_0}}{N} \Big)  \sup_{0\leq s\leq t}\Big[\|h(s)\|_{L^\infty}+\|h(s)\|^3_{L^\infty}\Big]\right.\nonumber\\
&\qquad\qquad\qquad\qquad\qquad\left.+C_{\v,\l,N,T_0}\sup_{0\leq s\leq t}\Big[\|f(s)\|_{L^2}+\|f(s)\|^3_{L^2}\Big]\right\},
\end{align}
holds true for all $0\leq t\leq T_0$, where $\la>0$ and $\v>0$ can be arbitrarily small, and $N>0$ can be arbitrarily large. 
\end{lemma}

\begin{proof}
In order to obtain \eqref{5.46}, in terms of Lemma \ref{lem.ad1}, we still need to 
estimate the first and second terms on the right-hand side of \eqref{5.65}. Take $(t,x,v)\in (0,T_0]\times \Omega\times \R^3$. Denoting $z:=x-v(t-s)$, one has $z\in\Omega$ since $t_1(t,x,v)\leq s$. Then,  to estimate the first term on the right-hand side of \eqref{5.65}, we notice 
\begin{equation*}
\int_0^{t-\l}\Fi_{\{t_1\leq s\}} e^{-\f{\nu_0}{4}(t-s)} \Big|(K_wh)(s,z,v)\Big|\,\dd s\leq \int_0^{t-\l}\Fi_{\{t_1\leq s\}} e^{-\f{\nu_0}{4}(t-s)}\int_{\mathbb{R}^3}|k_{w}(v,v')h(s,z,v')|\,\dd v'\dd s,
\end{equation*}
and then apply \eqref{5.65} again to bound the term $|h(s,z,v')|$ in the above integral, so as to obtain
\begin{align}\label{5.66}
&\int_0^{t-\l}\Fi_{\{t_1\leq s\}} e^{-\f{\nu_0}{4}(t-s)} \Big|(K_wh)(s,z,v)\Big|\,\dd s
\nonumber\\
&\leq Ce^{\nu_0\tilde{t}}\int_0^{t-\l}e^{-\f14\nu_0(t-s)}\bigg\{e^{-\f{\nu_0}{4}s} \|h_0\|_{L^\infty}+\Big(\l+\v+\f{C_{\v,T_0}}{N}\Big)\sup_{0\leq s\leq t}\left[\|h(s)\|_{L^\infty}+\|h(s)\|^2_{L^\infty}\right]\nonumber\\
&\qquad\qquad\qquad\qquad\qquad\qquad\qquad+C_{\v, N, T_0}\sup_{0\leq s\leq t}\left[\|f(s)\|_{L^2}+\|f(s)\|^2_{L^2}\right]\bigg\}\,\dd s\nonumber\\
&\quad+Ce^{\f{\nu_0}{4}\tilde{t}}\int_0^{t-\l} \int_0^{s-\l}\int_{\mathbb{R}^3}\int_{\mathbb{R}^3}\Fi_{\{t_1\leq s\}}  \Fi_{\{t_1'\leq \tau\}}  e^{-\f{\nu_0}{4}(t-\tau)} |k_{w}(v,v')k_{w}(v',v'')|\nonumber\\
&\qquad\qquad\qquad\qquad\qquad\qquad\qquad\qquad\times|h(\tau,z-v'(s-\tau),v'')|\,\dd v''\dd v'\dd\tau \dd s\nonumber\\
&\quad+Ce^{\f{\nu_0}{4}\tilde{t}}\int_0^{t-\l} \int_0^{s-\l}\int_{\mathbb{R}^3}\Fi_{\{t_1\leq s\}}  \Fi_{\{t_1'\leq \tau\}}  e^{-\f{\nu_0}{4}(t-\tau)} |k_{w}(v,v')|\nonumber\\ &\qquad\qquad\qquad\qquad\qquad\qquad\qquad\qquad\times\left|[w\Gamma_+(f,f)](\tau,z-v'(s-\tau),v')\right|\,\dd v'\dd\tau \dd s\nonumber\\
&:=A_1+A_2+A_3,
\end{align}
where we have denoted $t'_1:=t_1(s,z,v')$, and $A_1$, $A_2$ and $A_3$ denote the corresponding terms on the right. A direct computation shows that 
\begin{multline}\label{5.67}
A_1\leq Ce^{\nu_0\tilde{t}}\bigg\{e^{-\f{1}{8}\nu_0t} \|h_0\|_{L^\infty}+\Big(\l+\v+\f{C_{\v,T_0}}{N}\Big)\sup_{0\leq s\leq t}\left[\|h(s)\|_{L^\infty}+\|h(s)\|^2_{L^\infty}\right]\\
+C_{\v, N, T_0}\sup_{0\leq s\leq t}\left[\|f(s)\|_{L^2}+\|f(s)\|^2_{L^2}\right]\bigg\}.
\end{multline}
For $A_2$, we divide the estimate by 
the following cases. 

\medskip
\noindent{\it Case 1: $|v|\geq N$}.   It is straightforward to see that
\begin{align}
A_2\leq \f{C}{N}e^{\f{1}{4}\nu_0\tilde{t}} \sup_{0\leq s\leq t}\|h(s)\|_{L^\infty}.\nonumber
\end{align}

\noindent{\it Case 2: $|v|\leq N, |v'|\geq 2N$ or $|v'|\leq 2N, |v''|\geq 3N$}. 
In this case, note that either $|v-v'|\geq N$ or $|v'-v''|\geq N$ holds true. Thus, at least one of the following estimates is valid: 
\begin{align}\label{5.69}
|k_w(v,v')|\leq  Ce^{-\f{N^2}{32}}\Big|k_w(v,v')e^{\f{|v-v'|^2}{32}}\Big|,\quad
|k_w(v',v'')|\leq Ce^{-\f{N^2}{32}}\Big|k_w(v',v'')e^{\f{|v''-v'|^2}{32}}\Big|. 
\end{align}
Using \eqref{2.3}, it is direct to verify that 
\begin{equation}\label{5.70}
\left\{\begin{aligned}
&\int_{\mathbb{R}^3}\Big|k_w(v,v')e^{\f{|v-v'|^2}{32}}\Big|\,\dd v'\leq C(1+|v|)^{-1},\\
&\int_{\mathbb{R}^3}\Big|k_w(v',v'')e^{\f{|v''-v'|^2}{32}}\Big|\,\dd v''\leq C(1+|v'|)^{-1}.
\end{aligned}\right.
\end{equation}
Then, over $|v|\leq N, |v'|\geq 2N$ or $|v'|\leq 2N, |v''|\geq 3N$,   from \eqref{5.69} and \eqref{5.70}, we can bounded $A_2$ by
\begin{align*}
 Ce^{-\f{N^2}{32}}e^{\f{1}{4}\nu_0\tilde{t}} \sup_{0\leq s\leq t}\|h(s)\|_{L^\infty}.
\end{align*}

\noindent{\it Case 3: $|v|\leq N, |v'|\leq 2N, |v''|\leq 3N$}. 
Using the Cauchy-Schwarz, it follows that
\begin{multline}\nonumber
Ce^{\f{\nu_0}{4}\tilde{t}}\int_0^{t-\l} \int_0^{s-\l}e^{-\f{\nu_0}{4}(t-\tau)}
\Big( \int_{|v'|\leq2N}\int_{|v''|\leq3N} |k_{w}(v,v')k_{w}(v',v'')|^2\dd v''\dd v'\Big)^{\f12}\\
\times \Big(\int_{|v'|\leq2N}\int_{|v''|\leq3N}\Fi_{\{t_1\leq s\}}  \Fi_{\{t_1'\leq \tau\}} |h(\tau,z-v'(s-\tau),v'')|^2\dd v''\dd v' \Big)^{\f12}\dd\tau \dd s.
\end{multline}
Then, recalling $h=wf$, we can bound the above term as
\begin{equation}\nonumber 
C_N e^{\f{\nu_0}{4}\tilde{t}}\int_0^{t-\l} \int_0^{s-\l}e^{-\f{\nu_0}{4}(t-\tau)}
\Big(\int_{|v'|\leq2N}\int_{|v''|\leq3N}\Fi_{\{t_1\leq s\}}  \Fi_{\{t_1'\leq \tau\}} 
|f(\tau,z-v'(s-\tau),v'')|^2\dd v''\dd v' \Big)^{\f12}\dd\tau \dd s.
\end{equation}
Here, we note $z-v'(s-\tau)\in\Omega$ in the case when $t_1\leq s$ and $t'_1\leq \tau$. By taking the change of variable $v'\mapsto z_1:=z-v'(s-\tau)$, one can further bound the above term as  
\begin{align}\nonumber 
C_{\l,N}e^{\f{\nu_0}{4}\tilde{t}}\int_0^{t-\l} \int_0^{s-\l}e^{-\f{\nu_0}{4}(t-\tau)} \|f(\tau)\|_{L^2}\,\dd\tau \dd s\leq C_{\l,N}e^{\f{\nu_0}{4}\tilde{t}} \sup_{0\leq \tau \leq t}\|f(\tau)\|_{L^2}.
\end{align} 
Combining all estimates in three cases above, 
one obtains that 
\begin{align}\label{5.74}
A_2\leq \f{C}{N}e^{\f{\nu_0}{4}\tilde{t}} \sup_{0\leq s\leq t}\|h(s)\|_{L^\infty}
+ C_{\l,N}e^{\f{\nu_0}{4}\tilde{t}} \sup_{0\leq \tau \leq t}\|f(\tau)\|_{L^2}.
\end{align}
Next, we consider $A_3$. It follows from \eqref{2.4} that 
\begin{multline}\label{5.75}
A_3\leq Ce^{\f{\nu_0}{4}\tilde{t}}\sup_{0\leq \tau\leq t}\|h(\tau)\|_{L^\infty}\int_0^{t-\l} \int_0^{s-\l}\int_{\mathbb{R}^3}\Fi_{\{t_1\leq s\}}  \Fi_{\{t_1'\leq \tau\}}  e^{-\f{\nu_0}{4}(t-\tau)} |k_{w}(v,v')|\\ 
\times\left(\int_{\mathbb{R}^3}(1+|\eta|)^{-4\b+4}|h(\tau,z-v'(s-\tau),\eta)|^2\,\dd\eta\right)^{\f12}\dd v'\dd\tau \dd s.
\end{multline}
We also divide the further estimate on \eqref{5.75} by the following cases. 

\medskip
\noindent{\it Case 1: $|v|\geq N$}.  
It is direct to obtain
\begin{align}\nonumber
A_3\leq \f{C}{N}e^{\f{\nu_0}{4}\tilde{t}}\sup_{0\leq \tau\leq t}\|h(\tau)\|^2_{L^\infty}.
\end{align}
\noindent{\it Case 2: $|v|\leq N, |v'|\geq 2N$}. 
Using \eqref{5.69} and \eqref{5.70}, it is also direct to bound the corresponding part of  $A_3$ as
\begin{align}\nonumber
Ce^{-\f{N^2}{32}}e^{\f{\nu_0}{4}\tilde{t}}\sup_{0\leq \tau\leq t}\|h(\tau)\|^2_{L^\infty}.
\end{align}

\noindent{\it Case 3: $|v|\leq N$, $|v'|\leq 2N$, $|\eta|\geq N$}.  In this case, we can bound the corresponding part of  $A_3$ by
\begin{align}\nonumber
Ce^{\f{\nu_0}{4}\tilde{t}}\sup_{0\leq \tau\leq t}\|h(\tau)\|^2_{\infty}
\left(\int_{|\eta|\geq N}(1+|\eta|)^{-4\b+4}\dd \eta\right)^{\f12}
\leq  \f{C}{N}e^{\f{\nu_0}{4}\tilde{t}}\sup_{0\leq \tau\leq t}\|h(\tau)\|^2_{L^\infty}.
\end{align}
\noindent{\it Case 4: $|v|\leq N$, $|v'|\leq 2N$, $|\eta|\leq N$}. It follows from the Cauchy-Schwarz that  the corresponding part of $A_3$ can be bounded as
\begin{multline}\nonumber
C_N e^{\f{\nu_0}{4}\tilde{t}}\sup_{0\leq \tau\leq t}\|h(\tau)\|_{\infty}\int_0^{t-\l} \int_0^{s-\l}e^{-\f{\nu_0}{4}(t-\tau)} \Big(\int_{|v'|\leq2N}  |k_{w}(v,v')|^2\dd v'\Big)^{\f12}\\ 
\times\left(\int_{|v'|\leq2N}\int_{|\eta|\leq N}\Fi_{\{t_1\leq s\}}  \Fi_{\{t_1'\leq \tau\}}|f(\tau,z-v'(s-\tau),\eta)|^2\dd\eta \dd v'\right)^{\f12}\dd\tau \dd s.
\end{multline}
After making the change of variable $v'\mapsto z_1=z-v'(s-\tau)$ as before, the above estimate can be further bounded as 
\begin{equation}\nonumber
C_{\l,N} e^{\f{\nu_0}{4}\tilde{t}}\sup_{0\leq \tau\leq t}\|h(\tau)\|_{\infty}\sup_{0\leq\tau\leq t}\|f(\tau)\|_{L^2} 
\leq \f{C}{N}e^{\f{\nu_0}{4}\tilde{t}}\sup_{0\leq \tau\leq t}\|h(\tau)\|^2_{L^\infty}
+C_{\l,N} e^{\f{\nu_0}{4}\tilde{t}}\sup_{0\leq\tau\leq t}\|f(\tau)\|^2_{L^2}.
\end{equation}
Hence, combining all the above estimates in four cases, one obtains that
\begin{align}\label{est.a3}
A_3\leq \f{C}{N}e^{\f{\nu_0}{4}\tilde{t}}\sup_{0\leq \tau\leq t}\|h(\tau)\|^2_{L^\infty}
+C_{\l,N} e^{\f{\nu_0}{4}\tilde{t}}\sup_{0\leq\tau\leq t}\|f(\tau)\|^2_{L^2}.
\end{align}
Therefore, collecting three estimates  \eqref{5.67}, \eqref{5.74} and \eqref{est.a3} for  \eqref{5.66}
yields that 
\begin{align}\label{5.81}
&\int_0^{t-\l}\Fi_{\{t_1\leq s\}} e^{-\f{\nu_0}{4}(t-s)} \Big|(K_wh)(s,z,v)\Big|\,\dd s\nonumber\\
&\leq Ce^{\nu_0\tilde{t}}\bigg\{e^{-\f{1}{8}\nu_0t} \|h_0\|_{L^\infty}+\Big(\l+\v+\f{C_{\v,T_0}}{N}\Big)\sup_{0\leq s\leq t}[\|h(s)\|_{L^\infty}+\|h(s)\|^2_{L^\infty}]\nonumber\\
&\qquad\qquad\qquad\qquad+C_{\v, \l, N, T_0}\sup_{0\leq s\leq t}[\|f(s)\|_{L^2}+\|f(s)\|^2_{L^2}]\bigg\}.
\end{align}

Next, we consider the second term on the right-hand side of \eqref{5.65}.  To estimate this term, we need to make an iteration again in the nonlinear term, and the inequality \eqref{2.4} plays a key role in this procedure. Indeed, it follows from \eqref{2.4} that 
\begin{multline}\label{5.82.a1}
\int_0^{t-\l}\Fi_{\{t_1\leq s\}} e^{-\f{\nu_0}{4}(t-s)}\Big|[w\Gamma_+(f,f)](s,z,v)\Big|\,\dd s\\
\leq \int_0^{t-\l}\Fi_{\{t_1\leq s\}} e^{-\f{\nu_0}{4}(t-s)}
\|h(s)\|_{L^\infty}\left(\int_{\mathbb{R}^3}(1+|v'|)^{-4\b+4}|h(s,z,v')|^2\dd v'\right)^{\f12}\dd s.
\end{multline}
Hence by considering the above integral in $v'$ over $\{|v'|> N\}\cup\{|v'|\leq N\}$, 
\eqref{5.82.a1} is further bounded by
\begin{equation}\label{5.82}
\f{C}{N}\sup_{0\leq \tau\leq t}\|h(\tau)\|^2_{L^\infty} 
+\int_0^{t-\l}\Fi_{\{t_1\leq s\}} e^{-\f{\nu_0}{4}(t-s)}
\|h(s)\|_{L^\infty}\left(\int_{|v'|\leq N} (1+|v'|)^{-4\b+4}|h(s,z,v')|^2\dd v'\right)^{\f12}\dd s.
\end{equation}
Noting $z=x-v(t-s)\in\Omega$, we apply  \eqref{5.65} again for  $|h(s,z,v')|$ to get 
\begin{align}\label{5.83}
&\left(\int_{|v'|\leq N} (1+|v'|)^{-4\b+4}|h(s,z,v')|^2\dd v'\right)^{\f12}\nonumber\\
&\leq Ce^{\nu_0\tilde{t}}\bigg\{e^{-\f{\nu_0}{4}s} \|h_0\|_{L^\infty}+\Big(\l+\v+\f{C_{\v,T_0}}{N}\Big)\sup_{0\leq \tau\leq s}[\|h(\tau)\|_{L^\infty}+\|h(\tau)\|^2_{L^\infty}]\nonumber\\
&\qquad\qquad\qquad\qquad\qquad\qquad\qquad\qquad+C_{\v, N, T_0}\sup_{0\leq \tau\leq s}[\|f(\tau)\|_{L^2}+\|f(\tau)\|^2_{L^2}]\bigg\}\nonumber\\
&\quad+e^{\f{\nu_0}{4}\tilde{t}}\left(\int_{|v'|\leq N} (1+|v'|)^{-4\b+4}\Big|\int_0^{s-\l}\Fi_{\{t_1'\leq \tau\}} e^{-\f{\nu_0}{4}(s-\tau)}|(K_wh)(\tau,z_1,v')|\dd\tau\Big|^2\dd v'\right)^{\f12}\nonumber\\
&\quad+e^{\f{\nu_0}{4}\tilde{t}}\left(\int_{|v'|\leq N} (1+|v'|)^{-4\b+4}\Big|\int_0^{s-\l}\Fi_{\{t_1'\leq \tau\}} e^{-\f{\nu_0}{4}(s-\tau)}|[w\Gamma_+(f,f)](\tau,z_1,v')|ds\Big|^2\dd v'\right)^{\f12},
\end{align}
where we have denoted $z_1:=z-v'(s-\tau)$.
Then, substituting \eqref{5.83} into \eqref{5.82},  and using the Cauchy inequality as well as  \eqref{2.4}, one gets that 
\begin{align}
 & \int_0^{t-\l}\Fi_{\{t_1\leq s\}} e^{-\f{\nu_0}{4}(t-s)}\Big|[w\Gamma_+(f,f)](s,z,v)\Big|\,\dd s\nonumber\\
 &\leq Ce^{\nu_0\tilde{t}}\bigg\{\|h_0\|_{L^\infty} e^{-\f{\nu_0}{8}t}\int_0^t\|h(s)\|_{L^\infty} \,\dd s+\Big(\l+\v+\f{C_{\v,T_0}}{N}\Big)\sup_{0\leq \tau\leq t}[\|h(\tau)\|_{L^\infty}+\|h(\tau)\|^3_{L^\infty}]\nonumber\\
 &\qquad\qquad\qquad\qquad+ C_{\v, N, T_0}\sup_{0\leq \tau\leq t}[\|f(\tau)\|_{L^2}+\|f(\tau)\|^3_{L^2}] \bigg\}\nonumber\\
 &\quad+Ce^{\f14\nu_0\tilde{t}}\int_0^{t-\l}\Fi_{\{t_1\leq s\}} e^{-\f{\nu_0}{4}(t-s)}
 \|h(s)\|_{L^\infty}\cdot B_1\,\dd s\nonumber\\
 &\quad+Ce^{\f14\nu_0\tilde{t}}\int_0^{t-\l}\Fi_{\{t_1\leq s\}} e^{-\f{\nu_0}{4}(t-s)}
\sup_{0\leq \tau\leq s} \|h(\tau)\|^2_{L^\infty}\cdot B_2\,\dd s,\label{5.84}
\end{align}
with  
\begin{equation}\nonumber
B_1=
\left(\int_0^{s-\l}e^{-\f{\nu_0}{4}(s-\tau)}\int_{|v'|\leq N}\Fi_{\{t_1'\leq \tau\}} (1+|v'|)^{-4\b+4}\Big|\int_{\mathbb{R}^3}k_w(v',v'')h(\tau,z_1,v'')\,\dd v''\Big|^2\dd v' \dd\tau \right)^{\f12},
\end{equation}
and  $B_2=$
\begin{equation}\nonumber
\left(\int_0^{s-\l}e^{-\f{\nu_0}{4}(s-\tau)} \int_{|v'|\leq N}\int_{\mathbb{R}^3}\Fi_{\{t_1'\leq \tau\}}(1+|v'|)^{-4\b+4}(1+|v''|)^{-4\b+4}|h(\tau,z_1,v'')|^2\,\dd v'' \dd v' \dd\tau \right)^{\f12}.
\end{equation}
For $B_1$, it follows from \eqref{5.69} and \eqref{5.70} that 
\begin{multline}\label{5.85}
B_1\leq C_N\left(\int_0^{s-\l}e^{-\f{\nu_0}{4}(s-\tau)}\int_{|v'|\leq N} \int_{|v''|\leq2N}\Fi_{\{t_1'\leq \tau\}}|f(\tau,z_1,v'')|^2\,\dd v''\dd v' \dd\tau \right)^{\f12}\\
+Ce^{-\f{N^2}{32}}\sup_{0\leq\tau\leq s}\|h(\tau)\|_{L^\infty}
\leq Ce^{-\f{N^2}{32}}\sup_{0\leq\tau\leq s}\|h(\tau)\|_{L^\infty}+C_{\l,N}\sup_{0\leq\tau\leq s}\|f(\tau)\|_{L^2},
\end{multline}
where we have used the fact that $z_1=z-v'(s-\tau)\in\Omega$ since $t_1'\equiv t_1(s,z,v')\leq\tau$, and also made the change of variable
$v' \mapsto z_1$ with $\f{dz_1}{\dd v'}=(s-\tau)^{3}\geq\l^3>0$. Similarly, for $B_2$, one has that 
\begin{multline}\label{5.86}
B_2\leq C_N\left(\int_0^{s-\l}e^{-\f{\nu_0}{4}(s-\tau)} \int_{|v'|\leq N}
\int_{|v''|\leq N}\Fi_{\{t_1'\leq \tau\}}|f(\tau,z_1,v'')|^2\,\dd v'' \dd v' \dd\tau \right)^{\f12}\\
+\f{C}{N}\sup_{0\leq\tau\leq s}\|h(\tau)\|_{L^\infty} 
\leq \f{C}{N}\sup_{0\leq\tau\leq s}\|h(\tau)\|_{L^\infty}
+C_{\l,N}\sup_{0\leq\tau\leq s}\|f(\tau)\|_{L^2},
\end{multline}
where $\b\geq 5/2$ has been used.
Thus, substituting \eqref{5.86} and \eqref{5.85} into \eqref{5.84}, we obtain 
\begin{align}\label{5.87}
 & \int_0^{t-\l}\Fi_{\{t_1\leq s\}} e^{-\f{\nu_0}{4}(t-s)}\Big|[w\Gamma_+(f,f)](s,z,v)\Big|\,\dd s\nonumber\\
&\leq Ce^{\nu_0\tilde{t}}\bigg\{\|h_0\|_{L^\infty} e^{-\f{\nu_0}{8}t}\int_0^t\|h(s)\|_{L^\infty} \,\dd s
+\Big(\l+\v+\f{C_{\v,T_0}}{N}\Big)\sup_{0\leq \tau\leq t}\left[\|h(\tau)\|_{L^\infty}+\|h(\tau)\|^3_{L^\infty}\right]\nonumber\\
&\qquad\qquad\qquad\qquad+ C_{\l, \v, N, T_0}\sup_{0\leq \tau\leq t}\left[\|f(\tau)\|_{L^2}+\|f(\tau)\|^3_{L^2}\right] \bigg\}.
\end{align}
This then completes the estimate on the second term on the right-hand side of \eqref{5.65}. 
Plugging \eqref{5.87} and \eqref{5.81} into \eqref{5.65}, one obtains
that 
\begin{align}\label{5.88}
|h(t,x,v)|
&\leq C_8e^{2\nu_0\tilde{t}}\left[e^{-\f{\nu_0}{8}t} \|h_0\|_{L^\infty}+\|h_0\|_{L^\infty} e^{-\f{\nu_0}{8}t}\int_0^t\|h(s)\|_{L^\infty} \,\dd s\right]\nonumber\\
&\quad+C_8e^{2\nu_0\tilde{t}}\bigg\{\Big(\l+\v+\f{C_{\v,T_0}}{N}\Big)\sup_{0\leq s\leq t}\left[\|h(s)\|_{L^\infty}+\|h(s)\|^3_{L^\infty}\right]\nonumber\\
&\qquad\qquad\qquad\qquad\qquad+C_{\l,\v, N, T_0}\sup_{0\leq s\leq t}\left[\|f(s)\|_{L^2}+\|f(s)\|^3_{L^2}\right]\bigg\}, 
\end{align}
where $C_8\geq1$ is a generic constant. Then, \eqref{5.46} follows immediately from \eqref{5.88} by taking $\tilde{C}_3:=C_8$. Therefore, the proof of Lemma \ref{lem5.5} is completed.  
\end{proof}


\section{Global-in-time existence in $L^\infty_{x,v}$}\label{sec6}

\noindent{\bf Proof of Theorem \ref{thm1.2}.}  Recall \eqref{5.46} and also \eqref{5.12} for the definition of $\tilde{t}$. We now take  
$$
\tilde{C}_4:=\max\left\{C_0,C_3\r^{2\b},\tilde{C}_3\right\}>1,
$$
and define
\begin{align}\label{5.89}
\bar{M}:=4\tilde{C}_4^2M_0 \exp\left\{2\nu_0\tilde{t}+\f{8}{\nu_0}\tilde{C}_4M_0 e^{2\nu_0\tilde{t}}\right\}
\equiv 4\tilde{C}_4^2\tilde{C}_2^4M_0^5\exp\left\{ \f8{\nu_0}\tilde{C}_4\tilde{C}_2^4M_0^5\right\},
\end{align}
and 
\begin{align}\label{5.90}
T_0:=\f{16}{\nu_0}\Big[\ln\bar{M}+|\ln\d|\Big].
\end{align}
Here $\de>0$ is introduced in \eqref{def.delta}. 
We see that the above $\bar{M}$  in \eqref{5.89} depends only on $M_0$, 
and $T_0$ depends only on $\d$ 
and $M_0$.  Assume $\|f_0\|_{L^2}\leq \epsilon_1=\epsilon_1(\bar{M},T_0)$, where $\epsilon_1$ is defined in Lemma  \ref{lem5.3}. Then, it follows from  the {\it a priori} assumption \eqref{5.0} and Lemmas \ref{lem.l2} and \ref{lem5.5} that 
\begin{align}\label{5.91}
\|h(t)\|_{L^\infty}\leq \tilde{C}_4\tilde{C}_2^4  M_0^5  \left[1+\int_0^t\|h(s)\|_{L^\infty}\,\dd s\right]e^{-\f{\nu_0}{8}t}+D,
\end{align}
for all $0\leq t\leq T_0$, where 
\begin{equation}
\begin{split}
D:= &\tilde{C}_4\tilde{C}_2^4  M_0^4 \bigg\{\Big(\v+\l+\f{C_{\v,T_0}}{N} \Big)  \left[\bar{M}+\bar{M}^3\right] \\
&+C_{\v,\l,N,T_0}\left[e^{\tilde{C}_1\bar{M}T_0}\|f_0\|_{L^2}+\Big(e^{\tilde{C}_1\bar{M}T_0}\|f_0\|_{L^2}\Big)^3\right]\bigg\}.\label{def.D}
\end{split}
\end{equation}
Define
\begin{align}\nonumber
H(t):=1+\int_0^t\|h(s)\|_{L^\infty}\,\dd s.
\end{align}
It then follows from \eqref{5.91} that
\begin{align}\nonumber
H'(t)\leq \tilde{C}_4\tilde{C}_2^4  M_0^5  e^{-\f{\nu_0}{8}t}H(t)+D,
\end{align}
that is
\begin{align}\nonumber
\left(H(t) \exp\Big\{-\f{8}{\nu_0}\tilde{C}_4\tilde{C}_2^4  M_0^5 (1-e^{-\f{\nu_0}{8}t})\Big\}\right)'\leq D,
\end{align}
for all $0\leq t\leq T_0$.
Integrating the above inequality over $[0,t]$, we get that 
\begin{align}\label{5.96}
H(t)&\leq (1+Dt)\exp\Big\{\f{8}{\nu_0}\tilde{C}_4\tilde{C}_2^4  M_0^5 \left(1-e^{-\f{\nu_0}{8}t}\right)\Big\}\leq  (1+Dt)\exp\Big\{\f{8}{\nu_0}\tilde{C}_4\tilde{C}_2^4  M_0^5 \Big\},
\end{align}
for all $0\leq t\leq T_0$.
Substituting \eqref{5.96} back to \eqref{5.91}, one obtains that for all $0\leq t\leq T_0$, 
\begin{align}
\|h(t)\|_{L^\infty}&\leq \tilde{C}_4\tilde{C}_2^4  M_0^5 
\exp\Big\{\f{8}{\nu_0}\tilde{C}_4\tilde{C}_2^4  M_0^5 \Big\}\cdot(1+Dt) e^{-\f{\nu_0}{8}t}+D\nonumber\\
&\leq \f1{4\tilde{C}_4}\bar{M} [1+ Dt]e^{-\f{\nu_0}{8}t}+D\nonumber\\
&\leq \f1{4\tilde{C}_4}\bar{M}\left[1+\f{16}{\nu_0}D\right]e^{-\f{\nu_0}{16}t}+D.\label{5.97}
\end{align}
Now, recalling \eqref{def.D}, we firstly choose $\l>0$ and $\v>0$ sufficiently small, then $N>0$ large enough, and finally let $\|f_0\|_{L^2}\leq \epsilon_2$ with $\epsilon_2=\epsilon_2(\d, M_0)>0$ further sufficiently small,  so  that one has 
$D\leq \min\left\{\f{\nu_0}{64}, \f{\delta}{8}\right\}$.
This together with \eqref{5.97} immediately yield that 
\begin{align}
\|h(t)\|_{L^\infty}\leq \f5{16\tilde{C}_4}\bar{M}\exp\left\{-\f{\nu_0}{16}t\right\}+\f{\d}{8}< \f1{2\tilde{C}_4}\bar{M},
\label{5.99}
\end{align}
for all $0\leq t\leq T_0$.  Hence we have closed the {\it a priori} assumption \eqref{5.0} over $t\in[0,T_0]$  provided that 
$
\|f_0\|_{L^2}\leq \epsilon_0:=\min\{\epsilon_1,\epsilon_2\}.
$
Note that $\epsilon_0>0$ depends only on $\d$ and $M_0$.   

Combining \eqref{5.99} and the local existence Theorem \ref{LE}, we can extend the Boltzmann solution to time interval $t\in[0,T_0]$.   Indeed, it follows from Theorem \ref{LE} that the Boltzmann solution $f(t)$ exists on $t\in[0,\hat{t}_0]$ with 
\begin{equation}\label{5.101}
\sup_{0\leq t\leq \hat{t}_0}\|h(t)\|_{L^\infty}\leq 2\tilde{C}_4\|h_0\|_{L^\infty}\leq \f{1}{2\tilde{C}_4}\bar{M}.
\end{equation}
We define $t_*:=(\hat{C}_\r[1+(2\tilde{C}_4)^{-1}\bar{M}])^{-1}>0$ where $\hat{C}_\r$ is the constant in Theorem \ref{LE}. With $t=\hat{t}_0$ as the starting point, using \eqref{5.101} and  Theorem \ref{LE}, we extend the Boltzmann solution $f(t)$ to exist in time interval $t\in[0,\hat{t}_0+t_*]$ with 
\begin{equation}\nonumber
\sup_{\hat{t}_0\leq t\leq \hat{t}_0+t_*}\|h(t)\|_{L^\infty}\leq 2\tilde{C}_4\|h(\hat{t}_0)\|_{L^\infty}\leq \bar{M},
\end{equation}
which together with \eqref{5.101}, yields that $\sup_{0\leq t\leq \hat{t}_0+t_*}\|h(t)\|_{L^\infty}\leq \bar{M}$. Since $h(t)$ satisfies the assumption \eqref{5.0} over $t\in[0,\hat{t}_0+t_*]$, then we can apply the {\it a priori} estimate \eqref{5.99} to get that the Boltzmann solution $f(t)$ indeed is bounded by
\begin{equation}\nonumber
\sup_{0\leq t\leq \hat{t}_0+t_*}\|h(t)\|_{L^\infty}\leq \f{1}{2\tilde{C}_4}\bar{M}.
\end{equation}
Repeating the same procedure for finite times, we can extend the Boltzmann solution $f(t)$ to exist in  the time interval $t\in[0,T_0]$, and \eqref{5.99} is satisfied.

Next, for the case  $t\geq T_0$, we note from the first inequality of \eqref{5.99} and the definition \eqref{5.90} for $T_0$ that 
\begin{align}\nonumber
\|h(T_0)\|_{L^\infty}&\leq \f5{16\tilde{C}_4}\bar{M}\exp\left\{-\f{\nu_0}{16}T_0\right\}+\f{\d}{8}\leq \f{5\d}{16\tilde{C}_4}+\f{\d}{8}< \f{1}{2}\d.
\end{align}
With $t=T_0$ as the initial time and applying Proposition \ref{prop1.1}, we can extend the Boltzmann solution $f(t)$ from $[0,T_0]$ to $[T_0,\infty)$, and thus obtain the unique solution $f(t)$ globally in time on $[0,\infty)$ such that $F(t,x,v)=\mu+\sqrt{\mu}f(t,x,v)\geq 0$ and $\sup_{t\geq 0}\|wf(t)\|_{L^\infty}\leq \bar{M}$. This proves the global existence and uniqueness of solutions in weighted $L^\infty$ space.  

For the large time behavior of the obtained solution, we note that as an immediate consequence of Proposition \ref{prop1.1}, it holds that 
\begin{align}\label{5.100}
\|h(t)\|_{L^\infty}\leq C_0\|h(T_0)\|_{L^\infty}e^{-\vartheta(t-T_0)}\leq C_0\d e^{-\vartheta(t-T_0)},
\end{align}
for all $t\geq T_0$. By taking 
$\tilde{C}_0:=4\tilde{C}_4^3\tilde{C}_2^4$ and  
$\vartheta_1:=\min\left\{\vartheta,\f{\nu_0}{16}\right\}$,
it follows from \eqref{5.99}, \eqref{5.100} and a direct computation that for any $t\geq 0$,
\begin{align*}
\|h(t)\|_{L^\infty}&\leq \max\{\f12,C_0\} \bar{M} e^{-\vartheta_1 t}\\
&\leq \tilde{C}_4\bar{M} e^{-\vartheta_1 t}\\
&\leq 4\tilde{C}_4^3\tilde{C}_2^4M_0^5\exp\left\{ \f8{\nu_0}\tilde{C}_4\tilde{C}_2^4M_0^5\right\} e^{-\vartheta_1 t}\nonumber\\
&= \tilde{C}_0 M_0^5\exp\left\{ \f2{\nu_0}\tilde{C}_0M_0^5\right\} e^{-\vartheta_1 t},\nonumber
\end{align*}
which then proves \eqref{thm.ltb}.

Finally,  if $\Omega$ is strictly convex, and  
$F_0(x,v)$ is continuous except on $\gamma_0$ 
and satisfies \eqref{bc.i}, then $F(t,x,v)$ is continous in $[0,\infty)\times\{\overline{\Omega}\times\mathbb{R}^3\backslash \gamma_0\}$, according to Theorem \ref{LE} for the local-in-time existence result.
Therefore  the proof of Theorem \ref{thm1.2} is completed. \qed 

\section{Appendix}\label{sec7}

For completeness, we would list in this appendix some propositions obtained in \cite{Guo2}. First of all, it was  firstly observed in \cite{Guo2}  that the set in the space $\Pi _{l=1}^{k-1}\mathcal{V}_{l}$ not reaching $t=0$ after $k$ bounces is small when $k$ is large.
	
\begin{proposition}
For any $\varepsilon >0,$ there exists $k_{0}(\varepsilon,T_{0})>0$ such that for $k\geq k_{0}$, and  for all $(t,x,v)$ with $0\leq t\leq T_{0}$,  $x\in \overline{\Omega}$ and $v\in\mathbb{R}^{3},$
\begin{equation}\label{4.7}
\int_{\Pi _{l=1}^{k-1}\mathcal{V}_{l}}\mathbf{1}_{\{t_{k}(t,x,v,v_{1},v_{2}...,v_{k-1})>0\}}\,\Pi _{l=1}^{k-1}\dd\sigma _{l}\leq\varepsilon.
\end{equation}
Furthermore, for $T_0$ sufficiently large, there exist constants $C_1,C_2>0$, independent of $T_0$, such that for $k=C_1T_0^{\f54}$,
\begin{align}\label{4.8}
\int_{\Pi _{l=1}^{k-1}\mathcal{V}_{l}}\mathbf{1}_{\{t_{k}(t,x,v,v_{1},v_{2}...,v_{k-1})>0\}}\,\Pi _{l=1}^{k-1}\dd\sigma _{l}\leq\left\{\f12\right\}^{C_2T_0^{\f54}}.
\end{align}
\end{proposition}


Recall that we have denoted $h(t,x,v):=w(v)f(t,x,v)$. Then the  diffuse reflection boundary condition \eqref{diffuse} can be rewriten as
\begin{equation}\label{4.1-2}
	h(t,x,v)|_{\g_-}=\f{1}{\widetilde{w}(v)}\int_{\mathcal{V}(x)}h(t,x,v')\widetilde{w}(v')\,\dd\sigma,
\end{equation}
where $\mathcal{V}(x)$, $\dd\sigma$ are given in \eqref{def.snu} and \eqref{def.bm}, respectively. 
The blow proposition can be regarded as a special case of Lemma \ref{lem3.01} when $\varphi\equiv 0$.

\begin{proposition}
\label{lem3.1}
Assume that $h,\frac{q}{\nu }\in L^{\infty }$ satisfy $\{\partial_{t}+v\cdot \nabla _{x}+\nu(v) \}h=q(t,x,v),$ with the diffuse reflection boundary condition \eqref{4.1-2}. Recall the diffusive back-time cycles in \eqref{diffusecycle}. Then for any $0\leq s\leq t,$ for almost every $x,v$, if $t_{1}(t,x,v)\leq s,$ 
\begin{equation*}
h(t,x,v)=e^{-\nu(v) (t-s)}h(s,x-v(t-s),v)
+\int_{s}^{t}e^{-\nu(v) (t-\tau)}q(\tau,x-v(t-\tau),v)\,\dd\tau.  
\end{equation*}
If $t_{1}(t,x,v)>s,$ then for $k\geq 2,$ 
\begin{align*}
h(t,x,v)=&\int_{t_{1}}^{t}e^{-\nu(v) (t-\tau)}q(\tau ,x-v(t-\tau ),v)\,\dd\tau\nonumber\\
&+\frac{e^{-\nu (v)(t-t_{1})}}{\widetilde{w}(v)}\int_{\Pi _{j=1}^{k-1}\mathcal{V}_{j}}\sum_{l=1}^{k-1}\mathbf{1}_{\{t_{l+1}\leq s<t_l\}}h(s,x_{l}-v_{l}(t_{l}-s),v_{l})\,\dd\Sigma_{l}(s) \nonumber\\
&+\frac{e^{-\nu (v)(t-t_{1})}}{\widetilde{w}(v)}\int_{\Pi _{j=1}^{k-1}\mathcal{V}_{j}}\sum_{l=1}^{k-1}\int_{s}^{t_{l}}\mathbf{1}_{\{t_{l+1}\leq s<t_l\}}q(\tau ,x_{l}-v_{l}(t_{l}-\tau),v_{l})\,\dd\Sigma _{l}(\tau)\dd\tau  \nonumber\\
&+\frac{e^{-\nu (v)(t-t_{1})}}{\widetilde{w}(v)}\int_{\Pi _{j=1}^{k-1}\mathcal{V}_{j}}\sum_{l=1}^{k-1}\int_{t_{l+1}}^{t_{l}}\mathbf{1}_{\{t_{l+1}>s\}}q(\tau,x_{l}-v_{l}(t_{l}-\tau),v_{l})\,\dd\Sigma _{l}(\tau)\dd\tau  \nonumber \\
&+\frac{e^{-\nu (v)(t-t_{1})}}{\widetilde{w}(v)}\int_{\Pi _{j=1}^{k-1}\mathcal{V}_{j}}\mathbf{1}_{\{t_{k}>s\}}h(t_{k},x_{k},v_{k-1})\,\dd\Sigma _{k-1}(t_k),
\end{align*}
where $\dd\Sigma_{l}(s)$ is given by
\begin{equation*}
\dd\Sigma _{l}(s)=\{\Pi _{j=l+1}^{k-1}\dd\sigma _{j}\}\cdot \widetilde{w}(v_l) e^{-\nu (v_{l})(t_l-s)}\dd\s_l\cdot\Pi_{j=1}^{l-1}\{e^{-\nu(v_{j})(t_{j}-t_{j+1})}\dd\s_j\},
\end{equation*}
and $\dd\Sigma _{k-1}(t_k)$ is the value of $\dd\Sigma_{k-1}(s)$ at $s=t_k$.
\end{proposition}


The following result is also needed in the proof in the previous sections.


\begin{proposition}
\label{lem4.2}
Let $\r>0$ be sufficiently large and $\beta>3/2$ in the velocity weight function  \eqref{WF}. 
Assume $h_{0}\in L^{\infty }$.
Then, there exits a unique solution $h(t)=G(t)h_{0}\in L^{\infty }$ to the linear homogeneous equation 
\begin{equation}
\label{lhe.nu}
h_t+v\cdot\nabla h+\nu(v) h=0,
\end{equation}
with the diffuse reflection boundary condition \eqref{4.1-2}. 
Moreover, there exists a constant $C_\r>0$, depending only on $\r$, such that
\begin{equation}\label{4.17}
\|G(t)h_{0}\|_{L^\infty }\leq C_\r e^{-\f{\nu_0}2t}\max\left\{\left\|\frac{h_0}{1+|v|}\right\|_{L^\infty},\left\|e^{-\nu(v)+\nu_0}h_0\right\|_{L^\infty}\right\},
\end{equation}
for all $t\geq 0$, where $\nu_0:=\inf \nu(v)>0$ is a positive constant.
\end{proposition}

\medskip

\noindent{\bf Acknowledgments.}  Renjun Duan is partially supported by the General Research Fund (Project No. 409913) from RGC of Hong Kong and {a Direct Grant (No. 3132699) from CUHK.} Yong Wang is partially supported by National Natural Sciences Foundation of China No. 11771429 and 11688101. {The authors would thank the anonymous referees for the valuable and helpful comments on the paper.} 


\end{document}